\documentclass[10pt]{amsart}

\usepackage[latin2]{inputenc}
\usepackage{amsmath}
\usepackage{graphicx}
\usepackage{amssymb}
\usepackage{esint}
\usepackage[dvipsnames]{xcolor}
\usepackage{tikz}
\usepackage{xxcolor}
\usepackage{floatrow}
\usepackage{color}
\usepackage{amsthm}
\usepackage{epsfig}
\usepackage[english]{babel}

\newtheorem{theorem}{Theorem}
\newtheorem{assumption}[theorem]{Assumption}
\newtheorem{proposition}[theorem]{Proposition}
\newtheorem{lemma}[theorem]{Lemma}

\newtheorem{definition}[theorem]{Definition}

\newtheorem*{theorem*}{Theorem}

\def\Xint#1{\mathchoice
{\XXint\displaystyle\textstyle{#1}}%
{\XXint\textstyle\scriptstyle{#1}}%
{\XXint\scriptstyle\scriptscriptstyle{#1}}%
{\XXint\scriptscriptstyle\scriptscriptstyle{#1}}%
\!\int}
\def\XXint#1#2#3{{\setbox0=\hbox{$#1{#2#3}{\int}$ }
\vcenter{\hbox{$#2#3$ }}\kern-.6\wd0}}

\def\dashint{\Xint-}

\allowdisplaybreaks[2]

\newcommand{\Id}{\operatorname{Id}}

\newcommand{\dist}{\operatorname{dist}}

\newcommand{\RVE}{{\operatorname{RVE}}}
\newcommand{\SQS}{{\operatorname{sel-RVE}}}

\newcommand{\Var}{{\operatorname{Var}~}}
\newcommand{\Cov}{\operatorname{Cov}}
\newcommand{\osc}{{\operatorname{osc}}}

\newcommand{\Varaideal}{{\Var a^\RVE_{ij}|_{\operatorname{unexpl}}}}

\newcommand{\shom}{{\mathsf{hom}}}
\newcommand{\per}{{\operatorname{per}}}
\newcommand{\domain}{{[0,L]^d}}

\definecolor{Yellow}{rgb}{0.95,0.9,0.0} 
\definecolor{Red}{rgb}{0.8,0.1,0.1}
\definecolor{Green}{rgb}{0.1,0.65,0.2}
\definecolor{Blue}{rgb}{0.1,0.1,0.8}
\definecolor{Purple}{rgb}{0.7,0.1,0.7}
\definecolor{Grey}{rgb}{0.6,0.6,0.6}

\begin{document}

\title[The choice of representative volumes for random materials]{The choice of representative volumes in the approximation of effective properties of random materials}

\author{Julian Fischer}
\address{Institute of Science and Technology Austria (IST Austria),
Am Campus 1, 3400 Klosterneuburg, Austria, E-Mail:
julian.fischer@ist.ac.at}

\begin{abstract}
The effective large-scale properties of materials with random heterogeneities on a small scale are typically determined by the method of representative volumes: A sample of the random material is chosen -- the representative volume -- and its effective properties are computed by the cell formula.
Intuitively, for a fixed sample size it should be possible to increase the accuracy of the method  by choosing a material sample which captures the statistical properties of the material particularly well: For example, for a composite material consisting of two constituents, one would select a representative volume in which the volume fraction of the constituents matches closely with their volume fraction in the overall material. Inspired by similar attempts in material science, Le~Bris, Legoll, and Minvielle have designed a selection approach for representative volumes which performs remarkably well in numerical examples of linear materials with moderate contrast.
In the present work, we provide a rigorous analysis of this selection approach for representative volumes in the context of stochastic homogenization of linear elliptic equations. In particular, we prove that the method essentially never performs worse than a random selection of the material sample and may perform much better if the selection criterion for the material samples is chosen suitably.
\end{abstract}


\thanks{This project was initiated while the author enjoyed the hospitality of the Hausdorff Research Institute for Mathematics, Bonn, as a participant of the Trimester Program ``Multiscale Problems: Algorithms, Numerical Analysis and Computation''. The author would like to thank Sergio Conti, Mitia Duerinckx, Antoine Gloria, Claude Le~Bris, Fr\'ed\'eric Legoll, and Ben Schweizer for interesting discussions on the manuscript.}

\maketitle

\section{Introduction}

The most widely employed method for determining the effective large-scale properties of a material with random heterogeneities on a small scale is the method of representative volumes. It basically proceeds by taking a small sample of the material -- a ``representative volume element'' (RVE) -- and determining the properties of the sample by the cell formula. The criteria for the choice of the representative volume have been the subject of an ongoing debate; while in principle increasing the size of the material sample increases the accuracy of the approximation of the material properties, this comes at a correspondingly larger computational cost. It has been conjectured that for a fixed size of the material sample, selecting a material sample which captures certain statistical properties of the material in a particularly good way may be beneficial: For example, for a composite material consisting of two constituent materials, one would try to select a material sample for which the volume fraction of each constituent material within the sample matches the overall volume fraction of this constituent in the composite as closely as possible (see Figure~\ref{FigureExplanationLeBrisLegollMinvielle}). Alternatively, for linear materials one might try to match the averaged material coefficient in the sample with the average taken over the full material.
There have been efforts in material science and mechanics towards replicating further statistical properties of the material in a representative volume, an approach called ``special quasirandom structures'' \cite{PhysRevB.81.094203,PhysRevB.42.9622,PhysRevLett.65.353} or ``statistically similar representative volume elements'' \cite{BalzaniBrandsSchroeder,BalzaniBrandsSchroederCarstensen2010,
BalzaniScheunemannBrandsSchroeder,BalzaniSchroeder2009,
BrandsBalzaniScheunemannSchroederRichterRaabe,Schroeder2011}.
A particularly successful approach in this direction has been developed for linear materials by Le~Bris, Legoll, and Minvielle \cite{LeBrisLegollMinvielle}; their method proceeds by considering a large number of material samples, evaluating one or more cheaply computable statistical quantities of the samples (like, for example, the spatial average of the coefficient), and then choosing the sample as the representative volume that is most representative for the material as measured by these quantities. In the present work, in the context of stochastic homogenization of linear elliptic PDEs we provide the first rigorous justification of these approaches\footnote{Note that for one-dimensional linear elliptic PDEs -- a case in which homogenization is linear in the inverse of the coefficient and thus independent of the geometry of the material -- an analysis has directly been provided in \cite{LeBrisLegollMinvielle}.}.

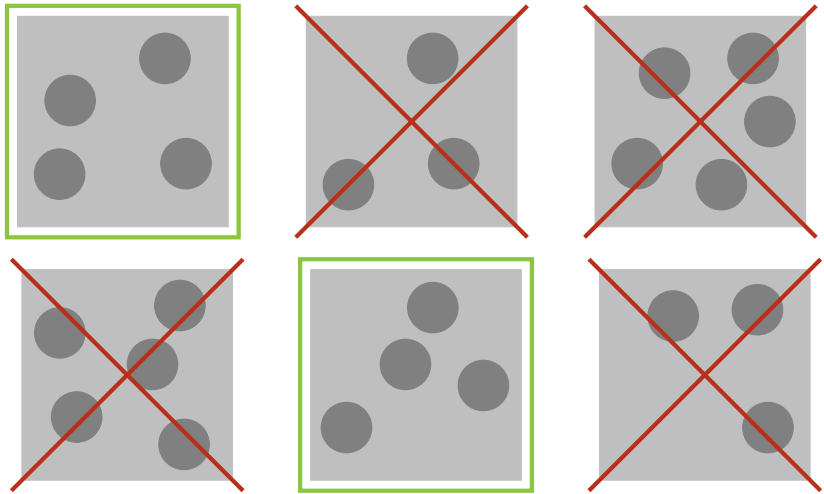
\begin{figure}[t]
\begin{tikzpicture}[scale=0.28]
\draw[color=white,ultra thick] (-0.5,-0.5) -- (10.5,10.5);
\draw[color=white,ultra thick] (-0.5,10.5) -- (10.5,-0.5);
\draw[color=LimeGreen,ultra thick] (-0.5,-0.5) -- (-0.5,10.5) -- (10.5,10.5) -- (10.5,-0.5) -- cycle;
\draw[fill=lightgray,color=lightgray] (0,0) rectangle (10,10);
\clip (0,0) rectangle (10,10);
\foreach \x in {(2.5,6.0),(7,8),(8,3),(2,2.5)}
\draw[fill=gray,color=gray] \x circle (1.2);
\end{tikzpicture}
$~~~~$
\begin{tikzpicture}[scale=0.28]
\draw[color=white,ultra thick] (-0.5,-0.5) -- (-0.5,10.5) -- (10.5,10.5) -- (10.5,-0.5) -- cycle;
\draw[fill=lightgray,color=lightgray] (0,0) rectangle (10,10);
\foreach \x in {(2,2),(6,8),(7,3)}
\draw[fill=gray,color=gray] \x circle (1.2);
\draw[ultra thick,color=BrickRed] (-0.5,-0.5) -- (10.5,10.5);
\draw[ultra thick,color=BrickRed] (-0.5,10.5) -- (10.5,-0.5);
\end{tikzpicture}
$~~~~$
\begin{tikzpicture}[scale=0.28]
\draw[color=white,ultra thick] (-0.5,-0.5) -- (-0.5,10.5) -- (10.5,10.5) -- (10.5,-0.5) -- cycle;
\draw[fill=lightgray,color=lightgray] (0,0) rectangle (10,10);
\foreach \x in {(2,3),(3.3,7.3),(7.5,8),(8.3,5),(6,2)}
\draw[fill=gray,color=gray] \x circle (1.2);
\draw[ultra thick,color=BrickRed] (-0.5,-0.5) -- (10.5,10.5);
\draw[ultra thick,color=BrickRed] (-0.5,10.5) -- (10.5,-0.5);
\end{tikzpicture}
\vspace{2mm}
\\
\begin{tikzpicture}[scale=0.28]
\draw[color=white,ultra thick] (-0.5,-0.5) -- (-0.5,10.5) -- (10.5,10.5) -- (10.5,-0.5) -- cycle;
\draw[fill=lightgray,color=lightgray] (0,0) rectangle (10,10);
\foreach \x in {(2.6,3),(1.8,7),(6.2,5.5),(7.7,1.7),(7.5,8.3)}
\draw[fill=gray,color=gray] \x circle (1.2);
\draw[ultra thick,color=BrickRed] (-0.5,-0.5) -- (10.5,10.5);
\draw[ultra thick,color=BrickRed] (-0.5,10.5) -- (10.5,-0.5);
\end{tikzpicture}
$~~~~$
\begin{tikzpicture}[scale=0.28]
\draw[color=white,ultra thick] (-0.5,-0.5) -- (10.5,10.5);
\draw[color=white,ultra thick] (-0.5,10.5) -- (10.5,-0.5);
\draw[color=LimeGreen,ultra thick] (-0.5,-0.5) -- (-0.5,10.5) -- (10.5,10.5) -- (10.5,-0.5) -- cycle;
\draw[fill=lightgray,color=lightgray] (0,0) rectangle (10,10);
\clip (0,0) rectangle (10,10);
\foreach \x in {(1.7,2.5),(5.8,8.2),(8.2,4.5),(4.5,5.5)}
\draw[fill=gray,color=gray] \x circle (1.2);
\end{tikzpicture}
$~~~~$
\begin{tikzpicture}[scale=0.28]
\draw[color=white,ultra thick] (-0.5,-0.5) -- (-0.5,10.5) -- (10.5,10.5) -- (10.5,-0.5) -- cycle;
\draw[fill=lightgray,color=lightgray] (0,0) rectangle (10,10);
\foreach \x in {(8,2.5),(7.5,8.1),(3.5,7.8)}
\draw[fill=gray,color=gray] \x circle (1.2);
\draw[ultra thick,color=BrickRed] (-0.5,-0.5) -- (10.5,10.5);
\draw[ultra thick,color=BrickRed] (-0.5,10.5) -- (10.5,-0.5);
\end{tikzpicture}
\caption{Among the six depicted material samples, the method of Le~Bris, Legoll, and Minvielle in its simplest realization would choose either the first sample or the fifth sample as the representative volume element and discard the others, as the volume fraction of the inclusions in the first and the fifth sample is closest to the overall material average. Note that in the depicted material samples the volume fraction of the inclusions is proportional to the number of inclusions, as all inclusions are of equal size. For a better illustration of the method, both the size and the number of the depicted samples have been chosen much smaller than in actual computations. \label{FigureExplanationLeBrisLegollMinvielle}}
\end{figure}

For materials with random heterogeneities on small scales, the approximation of the effective material coefficient by the method of representative volumes is a random quantity itself, as the outcome depends on the sample of the material. In the setting of linear elliptic PDEs with random coefficient fields -- which corresponds to the setting of heat conduction, electrical currents, or electrostatics in a material with random microstructure -- , Gloria and Otto \cite{GloriaOtto,GloriaOtto2,GloriaNumerical} have investigated the structure of the error of the approximation of the effective material coefficient by the method of representative volumes: The leading-order contribution to the error (with respect to the size of the RVE) consists of random fluctuations; in expectation the approximation of effective coefficients by the method of representative volumes is accurate to higher order, i.\,e.\ the systematic error of the RVE method is of higher order\footnote{At least if a suitable \emph{periodization} of the probability distribution of the coefficient field is available, see below for an explanation of this concept.}. For a given size of the RVE -- which corresponds to a fixed computational effort -- , the accuracy of the RVE method may therefore be increased significantly by reducing the variance of the approximations of the effective coefficient. It is precisely such a reduction of the variance by which the selection approach for representative volumes of Le~Bris, Legoll, and Minvielle \cite{LeBrisLegollMinvielle} achieves its gain in accuracy.

For linear elliptic PDEs with random coefficients and moderate ellipticity contrast, the reduction of the variance by the ansatz of Le~Bris, Legoll, and Minvielle \cite{LeBrisLegollMinvielle} is particularly remarkable: By selecting the representative volume according to the criterion that the averaged coefficient in the RVE should be particularly close to the averaged coefficient in the overall material, in numerical examples with ellipticity contrast $\sim 5$ they observed a variance reduction by a factor of $\sim 10$. Going beyond this simple selection criterion, they devised a criterion based on an expansion of the effective coefficient in the regime of small ellipticity contrast, which numerically achieves a remarkable variance reduction factor of $\sim 60$ even for a moderate ellipticity contrast $\sim 5$. Note that this basically corresponds to the gain of about one order of magnitude in accuracy for a negligible additional computational cost and implementation effort.

However, the analysis of the selection approach for representative volumes has been restricted to the one-dimensional setting \cite{LeBrisLegollMinvielle}, in which the homogenization of linear elliptic PDEs is linear in the inverse coefficient and therefore independent of the geometry of the material.
Besides the highly nonlinear dependence of the effective coefficient on the heterogeneous coefficient field in dimensions $d\geq 2$, one of the main challenges in the analysis of the selection method for representative volumes is the fact that it is only expected to increase the accuracy by a (though often very large) constant factor, at least for a fixed set of statistical quantities by which the selection is performed. At the same time, the available error estimates for the representative volume element method in stochastic homogenization are only optimal up to constant factors. For this reason, the analysis of the selection approach for representative volumes necessitates a fine-grained analysis of the structure of fluctuations in stochastic homogenization.

\subsection{Stochastic homogenization of linear elliptic PDEs: A brief outline}

The subject of the present contribution
is the rigorous justification of the selection method for representative volumes by Le~Bris, Legoll, and Minvielle \cite{LeBrisLegollMinvielle} in the context of linear elliptic equations
\begin{align}
\label{Equation}
-\nabla \cdot (a\nabla u)=f
\end{align}
with random coefficient fields $a$ on $\mathbb{R}^d$ for arbitrary spatial dimension $d$. Note that this setting describes e.\,g.\ heat conduction or electrostatics in a random material. Our assumptions on the probability distribution of the coefficient field $a$ are standard in the theory of stochastic homogenization: We assume just uniform ellipticity and boundedness, stationarity, and finite range of dependence (see conditions (A1)-(A3) below). In particular, our analysis includes the case of a two-material composite with random non-overlapping inclusions as depicted in Figure~\ref{FigureExplanationLeBrisLegollMinvielle}.

The theory of stochastic homogenization of linear elliptic PDEs predicts that for coefficient fields with only short-range correlations on a scale $\varepsilon\ll 1$ the solution $u$ to the equation with random coefficient field \eqref{Equation} may be approximated by the solution $u_\shom$ of an effective equation of the form
\begin{align}
\label{EffectiveEquation}
-\nabla \cdot (a_\shom \nabla u_\shom)=f,
\end{align}
where $a_\shom \in \mathbb{R}^{d\times d}$ is a constant effective coefficient which describes the effective behavior of the material. In this context of linear materials, the method of representative volumes is employed to compute the effective coefficient $a_\shom$.

Let us describe the method of representative volumes for the approximation of the effective material coefficient $a_\shom$ in more detail. It proceeds by choosing a sample of the material, say, a cube with side length $L\varepsilon$ for some $L\gg 1$, uniformly at random. Roughly speaking -- for the moment passing silently over the question of boundary conditions -- , by solving the equation for the homogenization corrector $\phi_i$ associated with the $i$-th coordinate direction on the representative volume
\begin{align}
\label{EquationCorrector}
-\nabla \cdot (a(e_i+\nabla \phi_i))=0
\quad\quad\text{ on }[0,L\varepsilon]^d
\end{align}
($e_i\in \mathbb{R}^d$ denoting the $i$-th vector of the standard basis)
one may obtain an approximation $a^\RVE$ for the effective coefficient $a_\shom$ in terms of the averaged fluxes
\begin{align}
\label{EquationRVE}
a^{\RVE}e_i :=\dashint_{[0,L\varepsilon]^d} a(e_i+\nabla \phi_i) \,dx.
\end{align}
This expression is also known in homogenization as the \emph{cell formula}.
As already mentioned before, the approximation $a^{\RVE}$ for the effective material coefficient $a_\shom$ is a random variable itself, as it depends on the realization of the random coefficient field $a$ on the sample volume $[0,L\varepsilon]^d$. It has been proven by Gloria and Otto \cite{GloriaOtto2,GloriaOttoNew} and also observed in numerical computations that the main contribution to the error of the RVE method is caused by the random fluctuations of the approximation $a^\RVE$, while the systematic error is of higher order:
For spatial dimensions $d\geq 1$ one has
\begin{align}
\label{VarianceBound}
\sqrt{\Var a^\RVE} \lesssim L^{-d/2}
\end{align}
but
\begin{align}
\label{SystematicErrorBound}
\big|\mathbb{E}[a^\RVE]-a_\shom \big|
\lesssim L^{-d} |\log L|^d.
\end{align}
As a consequence, a reduction of the fluctuations of the approximations $a^\RVE$ would lead to an increase in accuracy of the approximation for the effective coefficient $a_\shom$. It has been observed numerically by Le~Bris, Legoll, and Minvielle \cite{LeBrisLegollMinvielle} and shall be proven below rigorously that the selection approach for representative volumes achieves its gain in accuracy precisely by reducing the fluctuations of the approximations for the effective coefficients.

\subsection{Informal summary of our main results}

In the present work, we prove that in the setting of stochastic homogenization of linear elliptic equations the selection approach for representative volumes by Le~Bris, Legoll, and Minvielle \cite{LeBrisLegollMinvielle}
\begin{itemize}
\item essentially never performs worse than a completely random selection of the representative volume element, but may perform much better for suitable selection criteria,
\item basically maintains the order of the systematic error of the approximation for the effective coefficient, and
\item reduces also the error in the approximation for the effective coefficient that may occur with a given low probability, i.\,e.\ reduces also the ``outliers'' of the approximation for the effective coefficient.
\end{itemize}
As mentioned before, in the setting of linear elliptic PDEs the method of representative volumes is employed to obtain an approximation $a^\RVE$ for the effective (homogenized) coefficient $a_\shom$. The role of ``material samples'' is assumed by realizations of the random coefficient field $a:[0,L\varepsilon]^d\rightarrow\mathbb{R}^{d\times d}$, on which the computation of the approximations $a^\RVE$ is based.

The selection approach for representative volumes proposed in \cite{LeBrisLegollMinvielle} then proceeds as follows: At first, one or more statistical quantities $\mathcal{F}$ are chosen which assign a real number $\mathcal{F}(a)\in \mathbb{R}$ to any realization $a:[0,L\varepsilon]^d \rightarrow \mathbb{R}^{d\times d}$. Note that the simplest statistical quantity proposed in \cite{LeBrisLegollMinvielle} is the spatial average $\mathcal{F}(a):=\dashint_{[0,L\varepsilon]^d} a \,dx$. Next, one considers a sequence of independent samples of the random coefficient field until a sample meets the selection criterion
\begin{align}
\label{SelectionCondition}
\big|\mathcal{F}(a)-\mathbb{E}[\mathcal{F}(a)]\big|\leq \delta ~\sqrt{\Var \mathcal{F}(a)}
\end{align}
for some chosen parameter $\delta$ with $CL^{-d/2} |\log L|^C \leq \delta\leq 1$. Finally, the approximation for the effective coefficient is computed by solving the equation for the homogenization corrector \eqref{EquationCorrector} and using the cell formula \eqref{EquationRVE} for this sample of the random coefficient field.

To give a flavor of our main result, let us formulate it informally in the case of a single statistical quantity $\mathcal{F}(a)$.
We denote the approximation for the effective coefficient by the standard representative volume element method (without selection of material samples) by $a^\RVE$ and the approximation for the effective coefficient by the selection approach for representative volumes by $a^\SQS$.
In this case, our main theorems Theorem~\ref{TheoremSQS} and Theorem~\ref{TheoremSQSModerateDeviations} may be summarized as follows:
\begin{itemize}
\item The systematic error of the approximation $a^\SQS$ is essentially (up to powers of $\log L$ and some prefactors) of the same order as the systematic error of the standard representative volume element method $a^\RVE$: We have
\begin{align*}
\big|\mathbb{E}\big[a^\SQS\big]-a_\shom \big| \leq \frac{C\kappa^{3/2}}{\delta} L^{-d} |\log L|^C.
\end{align*}
The quantity $\kappa$ will be discussed below.
\item The fluctuations of the approximation $a^\SQS$ are reduced by the fraction of the variance of $a^\RVE$ that is explained by $\mathcal{F}(a)$: More precisely, we derive the estimate
\begin{align*}
~~~~~~~~~
\frac{\Var a^\SQS}{\Var a^\RVE} \leq & 1-(1-\delta^2)|\rho_{\mathcal{F}(a),a^\RVE}|^2
+\frac{C \kappa^{3/2} r_{\operatorname{Var}}}{\delta} L^{-d/2} |\log L|^C
\end{align*}
where $\rho_{\mathcal{F}(a),a^\RVE} \in [-1,1]$ denotes the correlation coefficient of $\mathcal{F}(a)$ and $a^\RVE$, given by
\begin{align*}
\rho_{\mathcal{F}(a),a^\RVE}:=
\frac{\Cov [a^\RVE,\mathcal{F}(a)]}{\sqrt{\Var \mathcal{F}(a) \Var a^\RVE}},
\end{align*}
and where $r_{\operatorname{Var}}:=\frac{L^{-d}}{\Var a^\RVE}$ denotes the ratio between the expected order of fluctuations of $a^\RVE$ and the actual magnitude of fluctuations. Note that the last term in the estimate on $\Var a^\SQS$ converges to zero as the size $L$ of the representative volume increases.
\item The probability of ``outliers'' is reduced by the selection method just as suggested by the variance reduction, at least in an ``intermediate'' region between the ``bulk'' and the ``outer tail'' of the probability distribution: One has a moderate-deviations-type estimate of the form
\begin{align*}
&~~~~~~~~~~~~~~
\mathbb{P}\Bigg[\frac{\big|a^\SQS_{ij}-a_{\shom,ij}\big|}{\sqrt{\big(1-|\rho_{\mathcal{F}(a),a^\RVE}|^2+\delta^2\big)\Var a^\RVE_{ij}+L^{-d/2-\beta}}}\geq s\Bigg]
\\&~~~~~~~~~~~~~~~~~~~~~~~~~
\leq \Big(1+\frac{C\delta}{\sqrt{1-|\rho|^2} s}+\frac{C}{\delta L^\beta}\Big)\mathbb{P}\big[|\mathcal{N}_1|\geq s\big]+\frac{C}{\delta}\exp(-L^{2\beta})
\end{align*}
for any $s\geq C\max\{(1-|\rho|^2)^{1/2} \delta^{-1},\delta (1-|\rho|^2)^{-1/2}\}$ and some $\beta=\beta(d)>0$, where $\mathcal{N}_1$ denotes the centered normal distribution with unit variance.
\item In the above bounds, $\kappa:=(1-|\rho_{\mathcal{F}(a),a^\RVE}|^2)^{-1}$ denotes (essentially) the condition number of the covariance matrix $\Var (a^\RVE,\mathcal{F}(a))$. For the case that the correlation $|\rho_{\mathcal{F}(a),a^\RVE}|$ is close to one, we derive bounds which are independent of $\kappa$ but come at the cost of a lower rate of convergence in $L$, namely
\begin{align*}
\big|\mathbb{E}\big[a^\SQS\big]-a_\shom \big| \leq \frac{C}{\delta} L^{-d/2-d/8} |\log L|^C
\end{align*}
and
\begin{align*}
~~~~~~~
\frac{\Var a^\SQS}{\Var a^\RVE} \leq & 1-(1-\delta^2)\big|\rho_{\mathcal{F}(a),a^\RVE}\big|^2
+\frac{C r_{\operatorname{Var}}}{\delta} L^{-d/8} |\log L|^C.
\end{align*}
\end{itemize}
Our estimate on the variance reduction achieved by the selection approach for representative volumes is implicit in the sense that it is determined by the correlation coefficient
\begin{align*}
\rho_{\mathcal{F}(a),a^\RVE}:=
\frac{\Cov [a^\RVE,\mathcal{F}(a)]}{\sqrt{\Var \mathcal{F}(a) \Var a^\RVE}}.
\end{align*}
In fact, the failure of the correlation coefficient $\rho_{\mathcal{F}(a),a^\RVE}$ to be nonzero also implies the failure of gaining accuracy by the selection approach for the representative volumes (see Theorem~\ref{TheoremFailureVarianceReduction}): In such a case of vanishing correlation, the method of Le~Bris, Legoll, and Minvielle \cite{LeBrisLegollMinvielle} is not superior (but essentially also not inferior) to the standard method of choosing a representative volume randomly.

This raises the question whether such a degeneracy of the correlation coefficient can occur for ``natural'' choices of the statistical quantity $\mathcal{F}(a)$. In Theorem~\ref{TheoremFailureVarianceReduction}, we shall prove that even for a ``natural'' choice like $\mathcal{F}(a):=\dashint_{[0,\varepsilon L]^d} a \,dx$ there is \emph{a~priori} no guarantee that there is a nonzero correlation between $a^\RVE$ and $\mathcal{F}(a)$: We construct an example of a probability distribution of $a$ for which the covariance of $a^\RVE$ and the average of the coefficient field $\dashint a$ in fact vanishes, while the variances $\Var \dashint_{[0,\varepsilon L]^d} a \,dx$ and $\Var a^\RVE$ are nondegenerate.

However, the failure of the variance reduction approaches to effectively reduce the variance is presumably limited to rather artificial examples: We prove that the covariance of $a^\RVE$ and the average of the coefficient field $\dashint a$ is positive for coefficient fields which are obtained from iid random variables by applying a ``monotone'' function, see Proposition~\ref{PropositionLowerBoundOnCorrelation}.

\subsection{Outline of our strategy}

\begin{figure}
\includegraphics[scale=0.1]{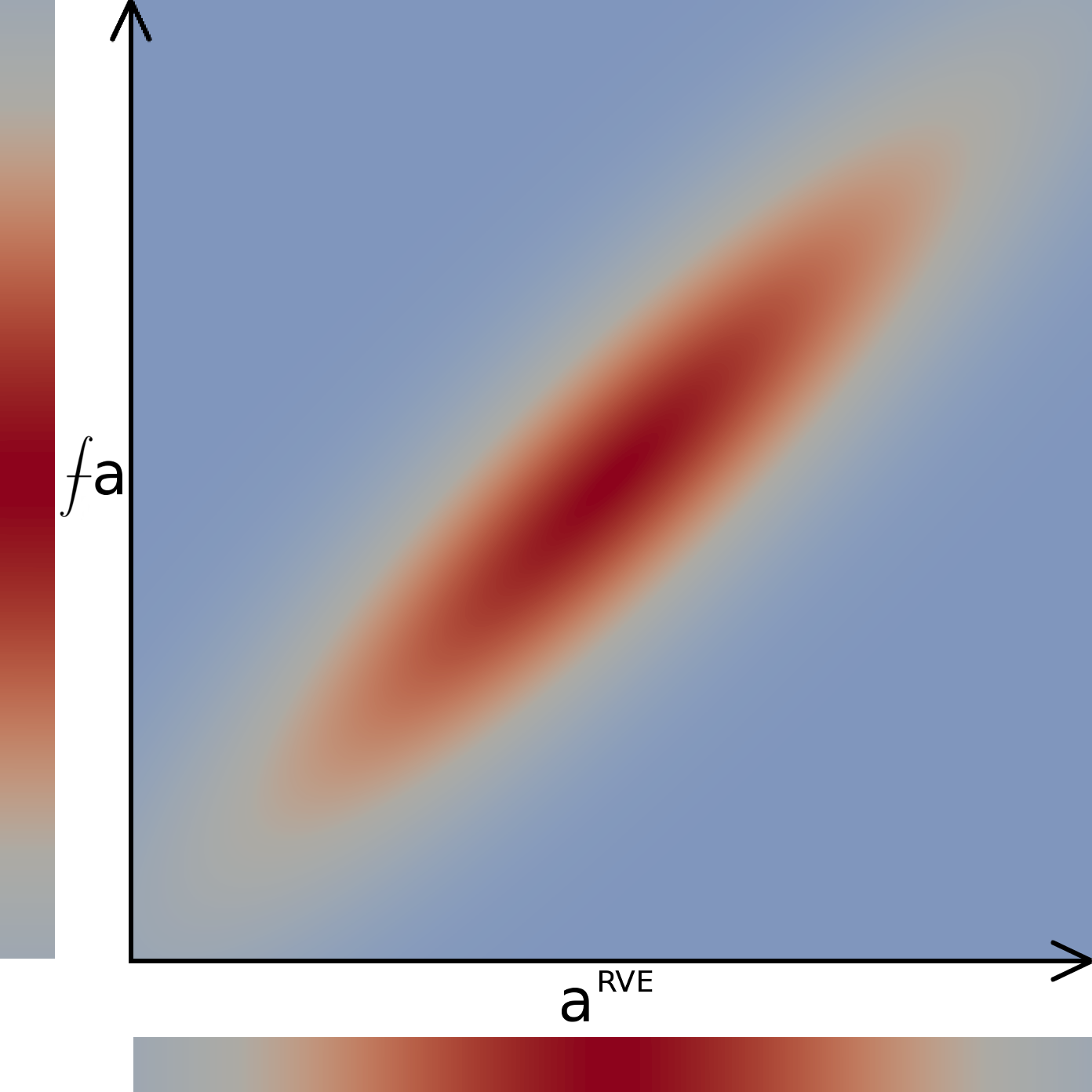}
~~~
\includegraphics[scale=0.1]{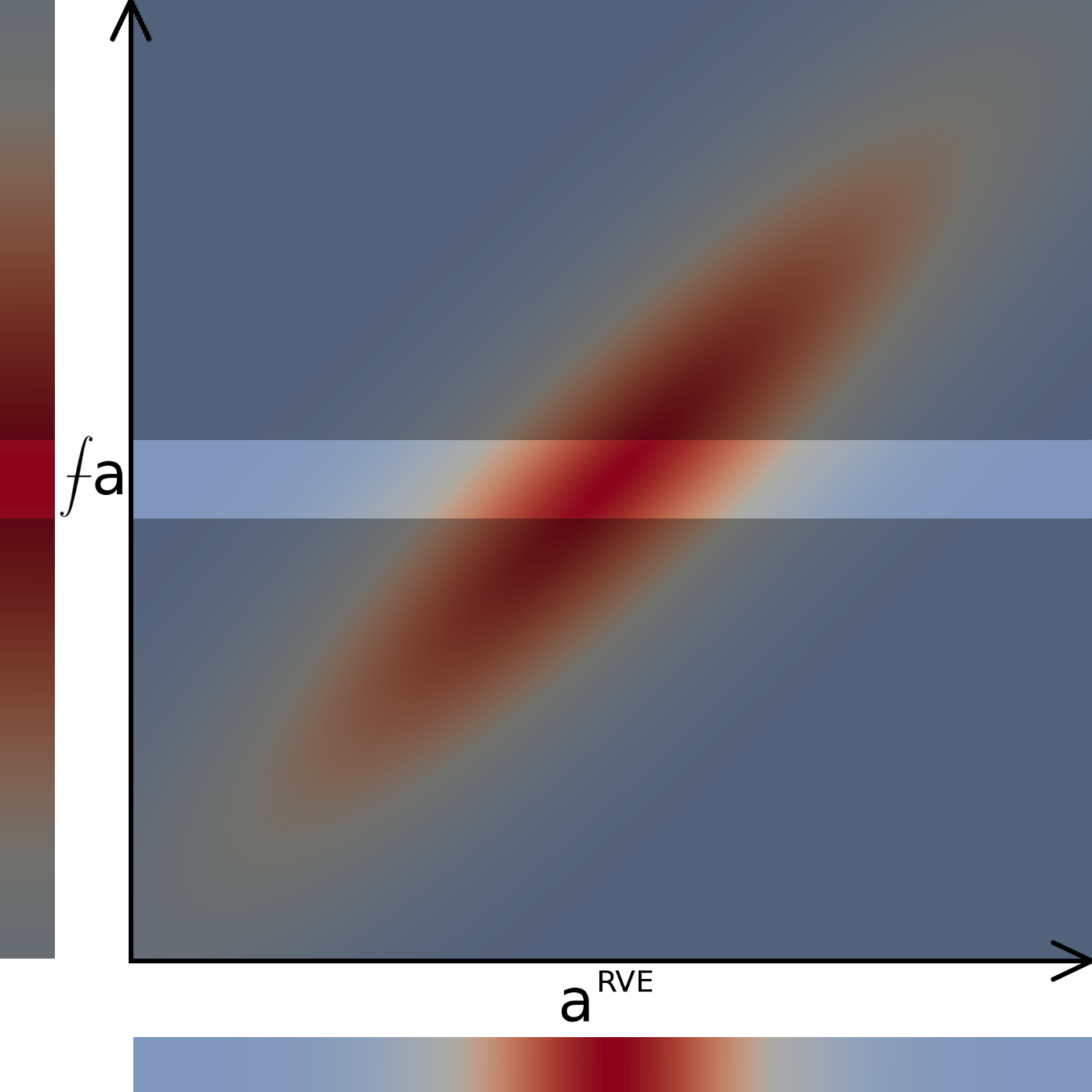}
\caption{For a multivariate Gaussian probability distribution, conditioning on the event of one variable being close to its expectation reduces the variance of the other variable, provided that the two random variables are nontrivially correlated. In our setting, conditioning on the event ``spatial average of coefficient field is close to its expectation'' reduces the variance of the random variable ``approximation for the effective conductivity'' $a^\RVE$, as their joint probability distribution is close to a multivariate Gaussian.\label{FigureConditioning}}
\end{figure}
The basic idea underlying our analysis of the selection approach for representative volumes is the observation that the joint probability distribution of the approximation for the effective coefficient $a^\RVE$ and one or more statistical quantities $\mathcal{F}(a)$ like the average of the coefficient field $\mathcal{F}(a):=\dashint_{[0,L\varepsilon]^d} a$ is close to a multivariate Gaussian, up to an error of the order $L^{-d} |\log L|^C$ in a suitable notion of distance between probability measures. The selection of representative volumes by the criterion \eqref{SelectionCondition} -- which amounts to conditioning on the event $|\mathcal{F}(a)-\mathbb{E}[\mathcal{F}(a)]|\leq \delta \sqrt{\Var \mathcal{F}(a)}$ -- then reduces the variance of the probability distribution of $a^\RVE$ by the variance explained by the statistical quantity $\mathcal{F}(a)$, up to error terms due to the deviation of the probability distribution from a multivariate Gaussian and the non-perfectness of the conditioning $\delta>0$, see Figure~\ref{FigureConditioning}. Note that for an ideal multivariate Gaussian distribution, the expected value of the approximation $a^\RVE$ would be left unchanged under conditioning since the criterion \eqref{SelectionCondition} is symmetric around $\mathbb{E}[\mathcal{F}(a)]$, i.\,e.\ the conditioning would not introduce a bias. As a consequence, for our approximate multivariate Gaussian $(a^\RVE,\mathcal{F}(a))$ the expectation of $a^\RVE$ is changed under conditioning only by the distance of our probability distribution to a multivariate Gaussian, which is a higher-order term. Note that both the reduction of the variance by conditioning and the estimate on the bias introduced by the conditioning rely crucially on the fact that our probability distribution is close to a multivariate Gaussian (and not another probability distribution): It is obvious from the picture in Figure~\ref{FigureConditioning} that a probability distribution other than a multivariate Gaussian could introduce a large bias under conditioning and even an \emph{increase} in variance.
Our analysis of the selection approach for representative volumes by Le~Bris, Legoll, and Minvielle \cite{LeBrisLegollMinvielle} is a first practical application of the beautiful theory of fluctuations in stochastic homogenization, which has been developed in recent years and which our work both draws ideas from and contributes to.

The underlying reason for the convergence of the joint probability distribution of $a^\RVE$ and one or more functionals $\mathcal{F}(a)$ towards a multivariate Gaussian is a central limit theorem for suitable collections of vector-valued random variables: We show that the approximation $a^\RVE$ for the effective coefficient $a_\shom$  -- and also the functionals $\mathcal{F}(a)$ that are used in the work of Le~Bris, Legoll, and Minvielle \cite{LeBrisLegollMinvielle} -- may be written as a sum of random variables with a local dependence structure with multiple levels, see Definition~\ref{ConditionRandomVariable} and Proposition~\ref{PropositionApproximabilityByMultilevel}.
For such sums of vector-valued random variables with multilevel local dependence, a proof of quantitative normal approximation is provided in the companion article \cite{FischerMultilevelLocalDependence} (see also Theorem~\ref{TheoremNormalApproximationMultilevelLocalDependence} below).
To the best of our knowledge such quantitative normal approximation results were previously known only for sums of random variables with local dependence structure \cite{ChenGoldsteinShao,ChenShao,RinottRotar} (corresponding more or less to just the lowest level of random variables in Figure~\ref{FigureMultilevel} below), a framework into which the approximation for the effective coefficient $a^\RVE$ does not fit.
Note that the sharp boundaries of the region defined by the selection criterion \eqref{SelectionCondition} (see also the sharp boundaries in Figure~\ref{FigureConditioning}) necessitate the use of a rather strong (though standard) distance between probability measures for our quantitative normal approximation result (see Definition~\ref{DefinitionDistance}); in particular, a stronger notion of distance between probability measures than the $1$-Wasserstein distance must be used.

As a by-product, our work also provides a proof of quantitative normal approximation for $a^\RVE$ in a different setting than available in the literature so far: To the best of our knowledge, the results on quantitative normal approximation for $a^\RVE$ in the literature always rely on an assumption that the coefficient field $a$ is obtained as a function of iid random variables \cite{DuerinckxGloriaOtto,GloriaNolen,Nolen} or that the probability distribution of $a$ is subject to a second-order Poincar\'e inequality like in \cite{DuerinckxGloriaPoincare}. In contrast, our result holds under the assumption of \emph{finite range of dependence}, in which to the best of our knowledge only a qualitative normal approximation result had been known \cite{ArmstrongKuusiMourratBook}.

The companion article \cite{FischerMultilevelLocalDependence} also provides a result on moderate deviations in the sense of Kramers for sums of random variables with multilevel local dependence structure, see Theorem~\ref{TheoremModerateDeviations}. Our result on the reduction of the error by the selection approach for representative volumes in the case of unlikely events (Theorem~\ref{TheoremSQSModerateDeviations}) is based on this moderate deviations theorem.

Our counterexample for the variance reduction -- which shows that even ``natural'' statistical quantities like the spatial average $\mathcal{F}(a) := \dashint_{[0,L\varepsilon]^d} a \,dx$ do not necessarily explain a positive fraction of the variance of $a^\RVE$ -- is based on the nonlinear dependence of the effective coefficient in periodic homogenization on the underlying coefficient field: More precisely, our counterexample consists of an interpolation between a standard random checkerboard and a random checkerboard with two types of tiles, one tile type being a constant coefficient field and one tile type being a second-order laminate microstructure. See Section~\ref{SectionCounterexample} for details of the construction.

\subsection{Computation of effective properties of random materials: A more detailed look}

In the homogenization of periodic linear materials -- i.\,e.\ in the homogenization of the linear elliptic PDE \eqref{Equation} with periodic coefficient field $a$ in the sense $a(x)=a(x+\varepsilon k)$ for all $k\in \mathbb{Z}^d$ -- it is possible to compute the effective coefficient $a_\shom$ by exploiting the periodicity of the coefficient field, basically reducing the problem to solving a PDE -- the PDE for the homogenization corrector -- on a single periodicity cell: For a period of length $\varepsilon$, the effective coefficient is given by the cell formula
\begin{align*}
a_\shom e_i\cdot e_j :=
\dashint_{[0,\varepsilon]^d} a (e_i+\nabla \phi_i)
\cdot e_j \,dx
\end{align*}
with the homogenization corrector $\phi_i$ defined as the unique $\varepsilon$-periodic solution with zero average to the PDE
\begin{align*}
-\nabla \cdot (a(e_i+\nabla \phi_i))&=0.
\end{align*}
As a consequence, in periodic homogenization the numerical computation of the effective coefficient $a_\shom$ typically requires only modest effort.

In contrast, in stochastic homogenization this simplification is no longer possible due to the absence of a periodic structure in the random coefficient field $a^{\mathbb{R}^{d}}:\mathbb{R}^{d}\rightarrow \mathbb{R}^{d\times d}$ and the computation of the effective coefficient becomes a computationally costly problem: The effective coefficient in stochastic homogenization is given by the infinite volume limit cell formula\footnote{This limit is to be read in an almost sure sense: By ergodicity, for almost every realization of $a$ this limit exists and is equal to a matrix which is independent of the realization.}
\begin{align*}
a_\shom e_i\cdot e_j :=
\lim_{L\rightarrow\infty} \dashint_{[0,L\varepsilon]^d} a^{\mathbb{R}^d} (e_i+\nabla \phi_i^{\operatorname{L,Dir}})
\cdot e_j \,dx
\end{align*}
with $\phi_i^{\operatorname{L,Dir}}$ denoting the solution to the corrector problem with Dirichlet boundary conditions
\begin{align*}
-\nabla \cdot (a^{\mathbb{R}^d}(e_i+\nabla \phi_i^{\operatorname{L,Dir}}))&=0&&\text{in }[0,L\varepsilon]^d,
\\
\phi_i^{\operatorname{L,Dir}}&\equiv 0&&\text{on }\partial [0,L\varepsilon]^d.
\end{align*}
In practice, in order to approximate the effective coefficient $a_\shom$ a representative volume $[0,L\varepsilon]^d$ of finite size must be chosen. However, the approximation of the effective coefficient by the standard cell formula with Dirichlet boundary conditions for the corrector
\begin{align*}
a_\shom e_i\cdot e_j \approx a^\RVE_{\operatorname{Dir}} e_i \cdot e_j:=
\dashint_{[0,L\varepsilon]^d} a^{\mathbb{R}^d} (e_i+\nabla \phi_i^{\operatorname{L,Dir}})
\cdot e_j \,dx
\end{align*}
is only of first-order accuracy $\mathbb{E}[|a^\RVE_{\operatorname{Dir}}-a_\shom|^2]^{1/2}\lesssim L^{-1}$ due to the presence of a boundary layer: The artificial Dirichlet boundary condition leads to the creation of a boundary layer in an $O(\varepsilon)$-neighborhood of the boundary $\partial [0,L\varepsilon]^d$. The limitation to first-order accuracy is present even in the systematic error $\mathbb{E}[a^\RVE]-a_\shom$. Note that while replacing the volume average in the cell formula by an average taken strictly in the interior of the representative volume typically increases the accuracy \cite{YuE}, for general probability distributions it does not increase the order of convergence due to global effects of the boundary layer. To achieve the convergence rates $|\mathbb{E}[a^\RVE]-a_\shom|\lesssim L^{-d}|\log L|^d$ and $\mathbb{E}[|a^\RVE-a_\shom|^2]^{1/2} \lesssim L^{-d/2}$ stated in \eqref{SystematicErrorBound} and \eqref{VarianceBound}, the boundary layer phenomenon must necessarily be addressed by the use of a more careful approximation technique than the method of correctors with Dirichlet boundary data.

One possibility of avoiding the creation of boundary layers is the use of a so-called ``periodization'' of the probability distribution: Given a probability distribution of coefficient fields $a^{\mathbb{R}^d}$, one first fixes the size $L\varepsilon$ of the desired representative volume and then attempts to construct a probability distribution of $L\varepsilon$-periodic coefficient fields $a$ such that the law of $a|_{x+[0,\frac{1}{2} L\varepsilon]^d}$ (i.\,e\ the law of $a$ restricted to some box of half the size of the representative volume) coincides with the law of $a^{\mathbb{R}^d}|_{x+[0,\frac{1}{2} L\varepsilon]^d}$ for any $x\in \mathbb{R}^d$. For one realization of the periodized probability distribution of coefficient fields $a$ one may then solve the corrector equation $-\nabla \cdot (a(e_i+\nabla \phi_i))=0$ with periodic boundary conditions on $\partial [0,L\varepsilon]^d$ and define the approximation $a^\RVE$ for the effective coefficient $a_\shom$ as
\begin{align}
\label{PeriodicaRVE}
a^\RVE e_i := \dashint_{[0,L\varepsilon]^d} a (e_i+\nabla \phi_i) \,dx.
\end{align}
This approximation $a^\RVE$ then has the desired approximation properties \eqref{VarianceBound} and \eqref{SystematicErrorBound}.
Note that this construction requires the knowledge of the probability distribution of $a^{\mathbb{R}^{d}}$ and must be done in a case-by-case basis; it is therefore not feasible in all practical situations.

To give an example, random non-overlapping inclusions like in Figure~\ref{FigureExplanationLeBrisLegollMinvielle} may be constructed by considering a Poisson point process on $\mathbb{R}^d\times [0,1]$, ordering the points $(x_k,y_k)\in \mathbb{R}^d\times [0,1]$ with respect to their last coordinate $y_k$, and then successively placing inclusions in $\mathbb{R}^d$ centered at the $x_k$ and with diameter $\varepsilon$ if the ``previous'' points $x_l$, $l<k$, have a distance of at least $\varepsilon$ from $x_k$ (i.\,e.\ $|x_l-x_k|\geq \varepsilon$). The result of such a construction is shown in Figure~\ref{PoissonPoint}a. For this probability distribution, one may define a periodization in a natural way by considering a Poisson point process on $[0,L\varepsilon)^d\times [0,1]$ and defining an $L\varepsilon$-periodic coefficient field with non-overlapping inclusions in the obvious way, replacing the Euclidean distance $|x_l-x_k|$ by the periodicity-adjusted distance $|x_l-x_k|_{\per}:= \inf_{z\in \mathbb{Z}^d} |x_l-x_k+L\varepsilon z|$. A sample from the periodized probability distribution is shown in Figure~\ref{PoissonPoint}b.

\begin{figure}
\begin{tikzpicture}[scale=0.35]
\draw (5,-1.5) node {(a)};
\draw[color=white,ultra thick] (-0.5,-0.5) -- (-0.5,10.5) -- (10.5,10.5) -- (10.5,-0.5) -- cycle;
\draw[fill=lightgray,color=lightgray] (0,0) rectangle (10,10);
\clip (0,0) rectangle (10,10);
\foreach \x in {(2.3,2.5),(1.2,5.6),(5.8,8.2),(1.5,8.2),(8.2,4.5),(8.5,6.9),(5.5,1.4),(8.6,1.8),(4.1,6.5),(6.1,4.6),(-0.2,0.1),(10.2,9.2),(7.6,10.2),(0.2,10.6),(3.1,-0.3),(-0.6,2.7),(3.8,9.5),(7.0,-0.6)}
\draw[fill=gray,color=gray] \x circle (0.9);
\end{tikzpicture}
$~~~~~$
\begin{tikzpicture}[scale=0.35]
\draw (5,-1.5) node {(b)};
\draw[color=white,ultra thick] (-0.5,-0.5) -- (-0.5,10.5) -- (10.5,10.5) -- (10.5,-0.5) -- cycle;
\draw[fill=lightgray,color=lightgray] (0,0) rectangle (10,10);
\clip (0,0) rectangle (10,10);
\foreach \ox in {0,-3.5,-7,3.5,7}
{
\foreach \oy in {0,-3.5,-7,3.5,7}
{
\foreach \x in {(0.5*2.3,0.5*2.5),(0.5*8.2,0.5*4.5),(0.5*8.5,0.5*6.9),(0.5*5.5,0.5*1.4),(0.5*4.1,0.5*6.5),(0.5*6.1,0.5*4.6)}
{
\draw[fill=gray,color=gray] \x+(\ox,\oy) circle (0.5*0.9);
}
}
}
\draw[color=black] (3.25,0) -- (3.25,10);
\draw[color=black] (6.75,0) -- (6.75,10);
\draw[color=black] (0,3.25) -- (10,3.25);
\draw[color=black] (0,6.75) -- (10,6.75);
\end{tikzpicture}
\caption{\label{PoissonPoint}(a) An example of random spherical inclusions distributed according to a Poisson point process, with overlapping inclusions removed. (b) A sample from the corresponding periodization of the probability distribution (rescaled); the periodicity cell is indicated by black lines.}
\end{figure}
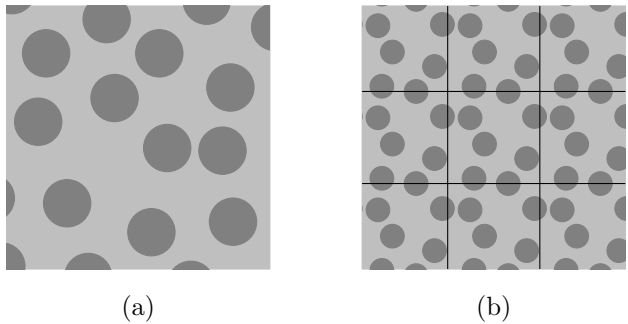

If no periodization of the probability distribution is available -- for example if only samples from the probability distribution are available and the underlying probability distribution is not known, like in applications where one has access to samples of the materials -- , one has to resort to an alternative means of increasing the rate of convergence of the method of representative volumes. One feasible option is to ``screen'' the effect of the boundary by introducing a ``massive'' term in the PDE for the homogenization corrector \cite{BlancLeBris,GloriaResonanceError,GloriaOtto2}: Fixing a scale $\sqrt{T} \sim \frac{L}{\log L}$, one replaces the equation for the homogenization corrector by the PDE
\begin{align*}
-\nabla \cdot (a^{\mathbb{R}^d}(e_i+\nabla \phi_i^{L,T})) + \frac{1}{T} \phi_i^{L,T} &=0&&\text{in }[0,L\varepsilon]^d,
\\
\phi_i^{L,T}&\equiv 0&&\text{on }\partial [0,L\varepsilon]^d
\end{align*}
and approximates the effective coefficient $a_\shom$ by
\begin{align*}
a_\shom e_i
\approx
a^\RVE e_i :=
\frac{1}{\int_{[0,L\varepsilon]^d} \eta \,dx}
\int_{[0,L\varepsilon]^d} \eta \, a^{\mathbb{R}^d} (e_i+\nabla \phi_i^{L,T}) \,dx,
\end{align*}
where $\eta$ is a smooth nonnegative weight supported in the slightly smaller box $[\frac{1}{8} L\varepsilon,(1-\frac{1}{8})L\varepsilon]^d$. In up to four spatial dimensions $d\leq 4$, this approximation also admits error estimates of the form
\begin{align*}
\sqrt{\Var a^\RVE} \lesssim L^{-d/2}
\end{align*}
and
\begin{align*}
\big|\mathbb{E}[a^\RVE]-a_\shom\big| \lesssim L^{-d} (\log L)^C.
\end{align*}

Due to the already substantial length of the present paper, we shall limit ourselves to the analysis of the selection approach for representative volumes in the context of periodizations of the probability distribution and defer the analysis of the screening approach to a future work.

Generally speaking, in the method of representative volumes the equation for the homogenization corrector may be solved by any numerical algorithm that is feasible for the given size of the representative volume: For example, standard finite element methods may be employed for representative volumes of moderate size, while for very large representative volumes one may use appropriate instances of modern computational homogenization methods like the multiscale finite element method, heterogeneous multiscale methods, and related approaches (see e.\,g.\ \cite{Abdulle,BabuskaCalozOsborn,BrezziEtAl,EEngquist,HughesEtAl,HouWu,
MatacheSchwab}) or the local orthogonal decomposition method by M\r{a}lqvist and Peterseim \cite{MalquistPeterseim}.

Note that besides the modern numerical homogenization methods -- which are in principle applicable to any elliptic PDE involving a heterogeneous coefficient field -- , there have been numerous numerical works on the more specific problem of the approximation of effective coefficients in stochastic homogenization, see for example \cite{AyoulGuilmardNouyBinetruy,CancesEhrlacherLegollStamm,
EfendievKronsbeinLegoll,EigelPeterseim,Khoromskij,MourratNumerical,
PeterseimCarstensen}.

\subsection{The selection approach for representative volumes by Le~Bris, Legoll, and Minvielle}
\label{SectionVarianceReductionApproaches}

Let us describe the selection approach for representative volumes by Le~Bris, Legoll, and Minvielle \cite{LeBrisLegollMinvielle} in more detail. The selection approach for representative volumes achieves its gain in accuracy of approximations $a^\RVE$ for the effective coefficient $a_\shom$ (as compared to the standard representative volume element method with completely random choice of the material sample) by selecting only those realizations of the random coefficient field $a|_{[0,L\varepsilon]^d}$ which capture some important statistical properties of the coefficient field $a$ in an exceptionally good way: For example, in the simplest setting Le~Bris, Legoll, and Minvielle \cite{LeBrisLegollMinvielle} propose to restrict one's attention to realizations of the coefficient field $a$ for which the average on $[0,L\varepsilon]^d$ is exceptionally close to its expected value
in the sense
\begin{align}
\label{ConditionSQSFirst}
\left| \dashint_{[0,L\varepsilon]^d} a \,dx - \mathbb{E}\bigg[\dashint_{[0,L\varepsilon]^d} a \,dx\bigg] \right| \leq \delta L^{-d/2}
\end{align}
for some $\delta\ll 1$. Note that for generic realizations of $a$ only
\begin{align*}
\bigg|\dashint_{[0,L\varepsilon]^d} a\,dx-\mathbb{E}\bigg[\dashint_{[0,L\varepsilon]^d} a\,dx\bigg]\bigg|\sim L^{-d/2}
\end{align*}
is true by the central limit theorem for the averages $\dashint_{[0,L\varepsilon]^d} a \,dx$ and the finite range of dependence $\varepsilon$.

On a numerical level, such a selection approach typically provides an increase in computational efficiency if the accuracy is indeed increased by conditioning on the event \eqref{ConditionSQSFirst}: Usually, the most expensive step in the computation of the approximations $a^\RVE$ is the computation of the homogenization corrector as the solution to the PDE \eqref{EquationCorrector}. In contrast, the generation of random coefficient fields $a$ and the evaluation of the average of $a$ is typically cheap. Therefore it is often worth generating about $\frac{1}{\delta}$ independent realizations of $a$ to obtain on average one realization of $a$ which satisfies \eqref{ConditionSQSFirst}; for this single realization, the corrector equation \eqref{EquationCorrector} is solved numerically and the approximation $a^\RVE$ for the effective coefficient is computed. This strategy is also applicable to situations in which the probability distribution of the coefficient field is not known, but one has only access to a large number of samples of the coefficient field, like in applications in which one has access to data from actual material samples.

The selection criterion \eqref{ConditionSQSFirst} based on the average of the coefficient field in the material sample is the first out of two selection criteria proposed by Le~Bris, Legoll, and Minvielle \cite{LeBrisLegollMinvielle}.
In order to reduce the variance of $a^\RVE$ further, they propose to consider several such statistical quantities at the same time, for example in addition to the spatial average
\begin{align*}
\mathcal{F}_{avg}(a):=\dashint_{[0,L\varepsilon]^d} a(x) \,dx
\end{align*}
the quantities
\begin{align}
\label{SecondOrderQuantity}
(\mathcal{F}_{2-point})_{i,j}(a):=
\dashint_{[0,L\varepsilon]^d} a\nabla v_i \cdot e_j \,dx \,dy
\end{align}
for some (approximation of the) solution $v_i$ to the constant-coefficient equation
\begin{align*}
-\Delta v_i = \nabla \cdot (ae_i),
\end{align*}
and require that all of these statistical quantities be close to their expectation at the same time. The quantities \eqref{SecondOrderQuantity} arise as a second-order correction to the effective conductivity $a^\RVE$ in the expansion in the regime of small ellipticity contrast: Expanding the homogenization corrector $\phi_i$ and the approximate effective conductivity $a^\RVE$ as a power series in $\nu$ for the family of coefficient fields
\begin{align*}
a=\operatorname{Id}+\nu \hat a,
\end{align*}
we deduce
\begin{align*}
\phi_i = \phi_i^0 + \nu \phi_i^1 + \nu^2 \phi_i^2 + O(\nu^3)
\end{align*}
with $\phi_i^0\equiv 0$, $\phi_i^1=v_i$, and $\phi_i^2$ defined as the solution to another PDE. As a consequence, for the approximation of the effective conductivity we obtain
\begin{align*}
a^\RVE e_i &= \dashint_{[0,L\varepsilon]^d} a e_i + \nu a \nabla v_i + \nu^2 \Id \nabla \phi_i^2 + O(\nu^3) \,dx
\\&
= \dashint_{[0,L\varepsilon]^d} a e_i + \nu a \nabla v_i \,dx + O(\nu^3)
\end{align*}
where in the last step we have used the periodicity of $\phi_i^2$. To see that the contribution of $v_i$ is actually of second order in $\nu$, one uses again $a=\Id+\nu \hat a$ and the periodicity of $v_i$.

By selecting the representative volumes by the \emph{two} criteria \eqref{ConditionSQSFirst} and
\begin{align}
\label{ConditionSQSSecond}
\Big|\mathcal{F}_{2-point} -\mathbb{E}\big[\mathcal{F}_{2-point}\big]\Big|
\leq \tilde \delta L^{-d/2}
\end{align}
at the same time, in the model problem of the random checkerboard with an ellipticity ratio of $5$ Le~Bris, Legoll, and Minvielle were able to reduce the variance of the approximations $a^\SQS$ for the effective conductivity by a factor of $50$, compared to the approximations $a^\RVE$ by the standard representative volume element method.

Another remarkable feature of the selection approach for representative volumes by Le~Bris, Legoll, and Minvielle is its compatibility with the vast majority of numerical homogenization methods: As the selection approach for representative volumes operates at the level of the choice of the coefficient field $a$, it may be combined with essentially any numerical discretization method for the corrector problem \eqref{CorrectorEquationRepeat}. Note that there exist many numerical homogenization methods that are particularly well-adapted to certain geometries of the microstructure; the selection approach for representative volumes may be employed in most of these methods to achieve a further speedup.

The selection approach for representative volumes is only one out of several variance reduction concepts in the context of stochastic homogenization: Blanc, Costaouec, Le~Bris, and Legoll \cite{BlancCostaouecLeBrisLegoll2,BlancCostaouecLeBrisLegoll,BlancLeBrisLegoll} have succeeded in reducing the variance by the method of antithetic variables; note that however for this approach the achievable variance reduction factor is much more limited.
The method of control variates has also been demonstrated to be successful in the context of the computation of effective coefficients in stochastic homogenization \cite{BlancLeBrisLegoll,LegollMinvielle}.

\subsection{A brief overview of quantitative stochastic homogenization}

For the sake of completeness, let us give a short overview of the tremendous progress that has been achieved in the quantitative theory of stochastic homogenization in recent years. The earliest (non-optimal) quantitative homogenization results for linear elliptic equations are due to Yurinski\u{\i} \cite{Yurinskii}. A decade later, Naddaf and Spencer \cite{NaddafSpencer} introduced the use of spectral gap inequalities in stochastic homogenization and derived optimal fluctuation estimates in the regime of small ellipticity contrast $||a-\Id||_{L^\infty} \ll 1$, i.\,e.\ in a perturbative setting. Another decade later, Caffarelli and Souganidis derived the first -- though only logarithmic -- rates of convergence for nonlinear stochastic homogenization problems \cite{CaffarelliSouganidis}. Gloria and Otto \cite{GloriaOtto,GloriaOtto2} and Gloria, Neukamm, and Otto \cite{GloriaNeukammOttoTwoScale} succeeded in the derivation of optimal homogenization rates for discrete linear elliptic equations with i.\,i.\,d.\ random conductances. Subsequently, these results were generalized to elliptic equations on $\mathbb{R}^d$ and correlated probability distributions by Gloria, Neukamm and Otto \cite{GloriaNeukammOtto,GloriaNeukammOttoInventiones}. For coefficient fields $a$ whose correlations decay quickly on scales larger than $\varepsilon>0$, these quantitative estimates for the homogenization error -- that is, for the difference between the solutions to the PDE with the random coefficient field \eqref{Equation} and its homogenized approximation \eqref{EffectiveEquation} -- read
\begin{align}
\label{TwoScaleError}
||u-u_\shom||_{L^p} \leq
\begin{cases}
\mathcal{C}(a) ||f||_{L^2} \varepsilon \sqrt{|\log \varepsilon|}&\text{for }d=2,
\\
\mathcal{C}(a) ||f||_{L^2} \varepsilon &\text{for }d\geq 3,
\end{cases}
\end{align}
with $\mathcal{C}(a)$ satisfying stretched exponential moment bounds and for suitable $p=p(d)$. Armstrong and Smart \cite{ArmstrongSmart} were the first to obtain power-law rates of convergence for nonlinear equations, deriving and employing an Avellanda-Lin type regularity estimate \cite{AvellanedaLin}; see also Armstrong and Mourrat \cite{ArmstrongMourrat}. Their estimates also come with optimal -- almost Gaussian -- stochastic moment bounds. Recently, the progress in stochastic homogenization culminated in the derivation of the optimal homogenization rates with optimal stochastic moment bounds by Armstrong, Kuusi, and Mourrat \cite{ArmstrongKuusiMourrat} and Gloria and Otto \cite{GloriaOttoNew}: For finite range of dependence $\varepsilon$, a quantitative error bound for the homogenization error of the form \eqref{TwoScaleError} holds true with a random constant $\mathcal{C}(a)$ with almost Gaussian moments $\mathbb{E}[\exp(\mathcal{C}(a)^{2-\delta}/C(\delta))]\leq 2$ for any $\delta>0$.

Higher-order approximation results in terms of homogenized problems have been derived in \cite{BellaFehrmanFischerOtto,BellaGiuntiOtto,BenoitGloria,Gu,LuOtto}, relying on the concept of higher-order correctors which was first used in the stochastic homogenization context in \cite{FischerOtto} to establish Liouville principles of arbitrary order in the spirit of Avellaneda and Lin's result in periodic homogenization \cite{AvellanedaLin2}. Further works in quantitative stochastic homogenization include the analysis of nondivergence form equations \cite{ArmstrongLin}, a regularity theory up to the boundary \cite{FischerRaithel}, denerate elliptic equations \cite{AndresNeukamm,GiuntiMourrat}, and the homogenization of parabolic equations \cite{ArmstrongBordasMourrat,LinSmart}.
Recently, Armstrong and Dario \cite{ArmstrongDario} and Dario \cite{Dario} succeeded in establishing quantitative homogenization for supercritical Bernoulli bond percolation on the standard lattice.
 
The fluctuations of the mathematical objects arising in the stochastic homogenization of linear elliptic PDEs have been the subject of a beatiful series of works, starting with the work of Nolen \cite{Nolen} and a subsequent work of Gloria and Nolen \cite{GloriaNolen} on quantitative normal approximation for (a single component of) the approximation of the effective conductivity $a^\RVE$ and a work of Mourrat and Otto \cite{MourratOtto} on the correlation structure of fluctuations in the homogenization corrector $\phi_i$. Mourrat and Nolen \cite{MourratNolen} have shown a quantitative normal approximation result for the fluctuations of the corrector.  Gu and Mourrat \cite{GuMourrat} have derived a description of fluctuations in the solutions to the equation with random coefficient field \eqref{Equation}. Recently, a pathwise description of fluctuations of the solutions to the equation with random coefficient field \eqref{Equation} -- namely, in terms of deterministic linear functionals of the so-called \emph{homogenization commutator} $\Xi:=(a-a_\shom)(e_i+\nabla \phi_i)$, a random field converging (for $\varepsilon\rightarrow 0$) towards white noise, -- was developed by Duerinckx, Gloria, and Otto \cite{DuerinckxGloriaOtto}. As far as quantitative normal approximation results are concerned, all of these works work under the assumption of i.i.d.\ coefficients (in the discrete setting) or second-order Poincar\'e inequalities. To the best of our knowledge, the present work provides the first quantitative description of fluctuations (though so far limited to the approximation of the effective conductivity $a^\RVE$) when the decorrelation in the coefficient field is quantified by the assumption of finite range of dependence instead of functional inequalities.

Note that despite its long history \cite{DalMasoModicaStochasticHomogenization,Kozlov,LionsSouganidis, PapanicolaouVaradhan}, the qualitative theory of stochastic homogenization has also been a very active area of research in the past years, see e.\,g.\  \cite{ArmstrongSouganidis,BraidesCicaleseRuf,HeidaSchweizerPlasticity, HornungPawelcykVelcic}; however, due to the substantial length of the present manuscript we shall not provide a more detailed discussion and refer the reader to these references instead.

{\bf Notation.}
Throughout the paper, we shall use standard notation for Sobolev spaces and weak derivatives; for a space-time function $v(x,s)$, we denote by $\nabla v$ its spatial gradient (in the weak sense) and by $\partial_s v$ its (weak) time derivative. The notation $\dashint_{B} f\,dx:=\frac{\int_B f \,dx}{\int_B 1 \,dx}$ is used for the average integral over a set $B$ of positive but finite Lebesgue measure. The space of measurable functions $f$ with $||f||_{L^p}:=(\int_{\mathbb{R}^d} |f|^p \,dx)^{1/p}<\infty$ will be denoted by $L^p$. By $L^p_{loc}$ we denote the space of functions $f$ with $f\chi_{\{|x|\leq R\}}\in L^p$ for all $R<\infty$. We shall also use the weighted space $L^p_{h}$ of functions with $||f||_{L^p_h}:=(\int_{\mathbb{R}^d} |f(x)|^p h(x) \,dx)^{1/p}<\infty$ for a nonnegative measurable weight function $h$. By $H^1(\mathbb{R}^d)$ we denote as usual the Sobolev space of functions $v\in L^2(\mathbb{R}^d)$ with $\nabla v\in L^2(\mathbb{R}^d)$; similarly, $H^1_{loc}(\mathbb{R}^d)$ is the space of functions $v$ with $v\in L^2_{loc}(\mathbb{R}^d)$ and $\nabla v\in L^2_{loc}(\mathbb{R}^d)$. For a Banach space $X$ we denote by $L^p([0,T];X)$ the usual Lebesgue-Bochner space.

As usual, we shall denote by $C$ and $c$ constants whose value may change from occurrence to occurrence. We are going to use the notation $\mathcal{C}(a)$ and similar expressions to denote a random constant subject to suitable moment bounds; again, the precise value of $\mathcal{C}(a)$ may change from occurrence to occurrence.

For a vector $v\in \mathbb{R}^m$ we denote by $|v|$ its Euclidean norm. We denote the identity matrix in $\mathbb{R}^{N\times N}$ by $\Id$ or $\Id_N$. For a matrix $A\in \mathbb{R}^{m\times m}$ we shall denote by $|A|$ its natural norm $|A|:=\max_{v,w\in \mathbb{R}^m,|v|=|w|=1} |v\cdot A w|$ and by $A^*$ its transpose (as all our matrices are real). For $x\in \mathbb{R}^d$ we denote by $|x|_\infty=\max_i |x_i|$ its supremum norm. By $|x-y|_\per$ respectively (for sets) $\dist_\per(U,V)$, we denote the periodicity-adjusted distance (in the context of the torus $[0,L\varepsilon]^d$). By $|x-y|_\infty^{\per}$ and $\dist^\per_\infty(x,y)$, we denote the corresponding distances associated with the maximum norm. For a positive definite matrix $A$, we denote by $\kappa(A)$ its condition number.

Given a positive definite symmetric matrix $\Lambda\in \mathbb{R}^{N\times N}$, we denote the Gaussian with covariance matrix $\Lambda$ by
\begin{align*}
\mathcal{N}_{\Lambda}(x):=\frac{1}{(2\pi)^{N/2}\sqrt{\det \Lambda}}
\exp\bigg(-\frac{1}{2}\Lambda^{-1} x \cdot x\bigg).
\end{align*}
For $\gamma>0$, we equip the space of random variables $X$ with stretched exponential moment $\mathbb{E}[\exp(|X|^\gamma/a)]<\infty$ for some $a=a(X)>0$ with the norm $||X||_{\exp^\gamma}:=\sup_{p\geq 1} p^{-1/\gamma} \mathbb{E}[|X|^p]^{1/p}$. For a discussion of this choice of norm, see Appendix~\ref{AppendixStretchedExponential}.

For a map $f:\mathbb{R}^N\rightarrow V$ into a normed vector space $V$, we denote for any $r>0$ by $\osc_r f(x_0):=\sup_{x,y\in \{|x-x_0|\leq r\}} |f(x)-f(y)|_V$ its oscillation in the ball of radius $r$ around $x_0$.

The conditional expectation of a random variable $X$ given $Y$ is denoted by $\mathbb{E}[X|Y]$.

\section{Main Results}

In the present work, we establish a rigorous justification of the selection approach for representative volumes by Le~Bris, Legoll, and Minvielle \cite{LeBrisLegollMinvielle} in the context of stochastic homogenization of linear elliptic PDEs for quite general probability distributions of the coefficient field $a^{\mathbb{R}^d}$: Our only assumptions on the probability distribution of the coefficient field $a^{\mathbb{R}^d}:\mathbb{R}^d\rightarrow \mathbb{R}^{d\times d}$  are uniform ellipticity and boundedness, stationarity, and finite range of dependence, which is a standard set of assumptions in stochastic homogenization \cite{ArmstrongSmart,GloriaOttoNew} (note that we equip the space of uniformly elliptic and bounded coefficient fields with the topology of Murat and Tartar's $H$-convergence \cite{MuratTartar}).
Let us remark that all of our results and proofs are also valid in the case of strongly elliptic systems, upon adapting the notation in the obvious way.
\begin{itemize}
\item[(A1)] \emph{Uniform ellipticity} of a coefficient field $a$ as usual means that there exists a positive real number $\lambda>0$ such that almost surely we have $a(x)v\cdot v \geq \lambda |v|^2$ for a.\,e.\ $x\in \mathbb{R}^d$ and every $v\in \mathbb{R}^d$. Furthermore we assume \emph{uniform boundedness} in the sense that almost surely $|a(x)v|\leq \frac{1}{\lambda}|v|$ holds for a.\,e.\ $x\in \mathbb{R}^d$ and every $v\in \mathbb{R}^d$.
\item[(A2)] \emph{Stationarity} means that the law of the shifted coefficient field $a(\cdot+x)$ must coincide with the law of $a(\cdot)$ for every $x\in \mathbb{R}^d$. On a heuristic level, this means that ``the probability distribution of $a$ is everywhere the same'' or, in other words, that the material is spatially statistically homogeneous.
\item[(A3)] \emph{Finite range of dependence} $\varepsilon$ means that for any two Borel sets $A,B\subset \mathbb{R}^d$ with $\dist(A,B)\geq \varepsilon$ the restrictions $a|_A$ and $a|_B$ must be stochastically independent. In particular, this assumption restricts the correlations in the coefficient field to the scale $\varepsilon\ll 1$.
\end{itemize}
Note that these assumptions include e.\,g.\ the case of a two-material composite with random (either overlapping or non-overlapping) inclusions of diameter $\varepsilon$, the centers distributed according to a Poisson point process (up to removal in case of overlap); see Figure~\ref{PoissonPoint}a. Further examples include coefficient fields $a^{\mathbb{R}^d}(x):=\xi(\tilde a(x))$ that arise by pointwise application of a nonlinear function $\xi:\mathbb{R}^{d\times d}\rightarrow \mathbb{R}^{d\times d}$ to a (tensor-valued) stationary Gaussian random field $\tilde a$ with finite range of dependence $\varepsilon$ and integrable correlations, provided that the function $\xi$ is Lipschitz and takes values in the set of uniformly elliptic and bounded matrices.

For the approximation of the effective coefficient $a_\shom$, it is of advantage to work with a so-called \emph{periodization} of the stationary ensemble of random coefficient fields $a^{\mathbb{R}^d}$ (employing terminology from statistical mechanics, a probability measure on the space of coefficient fields shall also be called an \emph{ensemble} of coefficient fields). By a periodization of an ensemble of coefficient fields $a^{\mathbb{R}^d}$ we understand an ensemble of coefficient fields $a$ which are almost surely $L\varepsilon \mathbb{Z}^d$-periodic for some $L\gg 1$ and for which the probability distribution of $a$ on each cube of size of half the period $\frac{L\varepsilon}{2}$ coincides with the probability distribution of the original coefficient field $a^{\mathbb{R}^d}$, i.\,e.\ for which the probability distribution of $a|_{x+[0,L\varepsilon/2]^d}$ coincides with the distribution of $a^{\mathbb{R}^d}|_{x+[0,L\varepsilon/2]^d}$ for all $x\in \mathbb{R}^d$. For such a periodization, the condition (A3) is replaced by the following conditions (A3$_a$), (A3$_b$), (A3$_c$):
\begin{itemize}
\item[(A3$_a$)] The coefficient field $a$ is almost surely $L \varepsilon \mathbb{Z}^d$-periodic.
\item[(A3$_b$)] There exists a \emph{finite range of dependence} $\varepsilon>0$ such that for any two measurable $L \varepsilon \mathbb{Z}^d$-periodic sets $A,B\subset \mathbb{R}^d$ with $\dist(A,B)\geq \varepsilon$ the restrictions $a|_A$ and $a|_B$ are stochastically independent.
\item[(A3$_c$)] For any $x_0\in \mathbb{R}^d$ the law of the restriction $a|_{x_0+[-\frac{L\varepsilon}{4},\frac{L\varepsilon}{4}]^d}$ coincides with the corresponding law for some (non-periodic) ensemble of coefficient fields $a^{\mathbb{R}^d}$ satisfying (A1)-(A3).
\end{itemize}
Furthermore, to include examples like the random checkerboard in our analysis, we need the following notion of discrete stationarity.
\begin{itemize}
\item[(A2')] We say that our probability distribution of coefficient fields $a$ satisfies \emph{discrete stationarity} if the law of the shifted coefficient field $a(\cdot+x)$ coincides with the law of $a(\cdot)$ for every shift $x\in \varepsilon \mathbb{Z}^d$.
\end{itemize}

Our main assumptions stated in Assumption~\ref{AssumptionsNotation} below consist of two parts: First, we assume that the probability distribution of coefficient fields $a^{\mathbb{R}^d}$ satisfies the standard assumptions from stochastic homogenization and that there exists a suitable periodization $a$ of the probability distribution. Second, we require the statistical quantities $\mathcal{F}(a)$ to admit a ``multilevel local dependence structure decomposition'' as introduced in Definition~\ref{ConditionRandomVariable} below.
Let us remark that both the spatial average
\begin{align*}
\mathcal{F}_{avg}(a):=\dashint_{[0,L\varepsilon]^d} a \,dx
\end{align*}
and the higher-order quantity $\mathcal{F}_{2-point}(a)$ considered by Le~Bris, Legoll, and Minvielle \cite{LeBrisLegollMinvielle} as defined in \eqref{SecondOrderQuantity} satisfy the conditions in Definition~\ref{ConditionRandomVariable}; a proof of this fact is provided in Proposition~\ref{PropositionApproximabilityByMultilevel} below. As a consequence, both the spatial average $\mathcal{F}_{avg}(a)$ and the higher-order quantity $\mathcal{F}_{2-point}(a)$ may be chosen as the statistical quantities by which the selection of representative volumes is performed in our main theorems Theorem~\ref{TheoremSQS} and Theorem~\ref{TheoremSQSModerateDeviations}.

\begin{assumption}[Assumptions and Notation]
\label{AssumptionsNotation}
Consider a probability distribution of random coefficient fields $a^{\mathbb{R}^d}$ on $\mathbb{R}^d$, $d\geq 1$, which satisfies the conditions of ellipticity, stationarity, and finite range of dependence (A1)-(A3). Let $L\geq 2$ and suppose that there exists an $L\varepsilon$-periodization $a$ of the probability distribution of $a^{\mathbb{R}^d}$ subject to (A1), (A2), (A3$_a$) - (A3$_c$). Denote by $a^\RVE$ the approximation for the effective coefficient $a_\shom$ by the standard representative volume element method with a material sample of size $[0,L\varepsilon]^d$, i.\,e.\ set
\begin{align*}
a^\RVE e_i := \dashint_{[0,L\varepsilon]^d} a(e_i+\nabla \phi_i) \,dx
\end{align*}
with $\phi_i$ being the unique $L\varepsilon$-periodic solution with vanishing average to the corrector equation
\begin{align*}
-\nabla \cdot (a(e_i+\nabla \phi_i))=0.
\end{align*}
Let $\mathcal{F}(a)=(\mathcal{F}_1(a),\ldots,\mathcal{F}_N(a))$ be a collection of statistical quantities of the coefficient field $a$ which are subject to the conditions of Definition~\ref{ConditionRandomVariable} with $K\leq C_0$, $B\leq C_0 |\log L|^{C_0}$, and $\gamma\geq c_0$ for some $0<c_0,C_0<\infty$. Suppose that the covariance matrix of $\mathcal{F}(a)$ is nondegenerate and bounded in the natural scaling in the sense
\begin{align}
\label{FNondegenerate}
L^{-d} \Id \leq \Var \mathcal{F}(a) \leq C_0 L^{-d} \Id.
\end{align}
For any $1\leq i,j\leq d$ introduce the condition number $\kappa_{ij}$ of the covariance matrix of $(a^\RVE_{ij},\mathcal{F}(a))$
\begin{align*}
\kappa_{ij}:=\kappa\big(\Var (a^\RVE_{ij},\mathcal{F}(a))\big)
\end{align*}
and the ratio $r_{\operatorname{Var},ij}$ between the expected order of fluctuations and the actual fluctuations of the approximation $a^\RVE_{ij}$
\begin{align*}
r_{\operatorname{Var},ij}:= \frac{L^{-d}}{\Var a^\RVE_{ij}}.
\end{align*}
Denote by $C$ a constant depending on $d$, $\lambda$, $\gamma$, $N$, and $C_0$.
\end{assumption}

Under the above assumptions, the selection approach for representative volumes to capture certain statistical properties of the material in the representative volume particularly well -- as proposed by Le~Bris, Legoll, and Minvielle \cite{LeBrisLegollMinvielle} -- leads to the following increase in accuracy of the computed material coefficients. 
\begin{theorem}[Justification of the Selection Approach for Representative Volumes]
\label{TheoremSQS}
Let the assumptions and notations of Assumption~\ref{AssumptionsNotation} be in place. Denote by $a^\SQS$ the approximation for the effective coefficient $a_\shom$ by the selection approach for representative volumes introduced by Le~Bris, Legoll, and Minvielle \cite{LeBrisLegollMinvielle} in the case of a representative volume of size $L\varepsilon$. Suppose that the representative volumes $a|_{[0,L\varepsilon]^d}$ are selected from the periodized probability distribution according to the criterion
\begin{align}
\label{SQSConditioning}
\left|\mathcal{F}(a)-\mathbb{E}\big[\mathcal{F}(a)\big]\right|
\leq
\delta L^{-d/2}
\end{align}
for some $\delta\in (0,1]$. Let the selection criterion be chosen not too strict in the sense that $\delta^N \geq C L^{-d/2} |\log L|^{C(d,\gamma,C_0)}$.
Then the selection approach for representative volumes is subject to the following error analysis:
\newline
a) The systematic error of the approximation $a^\SQS$ satisfies the estimate
\begin{align}
\label{SystematicErrorSQS}
\big|\mathbb{E}\big[a^\SQS\big]-a_\shom \big| \leq \frac{C \kappa_{ij}^{3/2}}{\delta^N} L^{-d} |\log L|^{C(d,\gamma)}.
\end{align}
b) The variance of the approximation $a^\SQS$ is estimated from above by
\begin{align}
\label{VarianceReductionSQS}
\frac{\Var a^\SQS_{ij}}{\Var a^\RVE_{ij}} \leq 1-(1-\delta^2) |\rho|^2 + \frac{C \kappa_{ij}^{3/2}r_{\operatorname{Var},ij}}{\delta^N} L^{-d/2} |\log L|^{C(d,\gamma)}
\end{align}
where $|\rho|^2$ is the fraction of the variance of $a^\RVE_{ij}$ explained by the $\mathcal{F}(a)$, that is, $|\rho|^2$ is the maximum of the squared correlation coefficient between $a^\RVE_{ij}$ and any linear combination of the $\mathcal{F}_n(a)$. The explained fraction of  the variance is given by the formula
\begin{align}
\label{FormulaRho}
|\rho|^2 := \frac{\Cov[a^\RVE_{ij},\mathcal{F}(a)] \cdot (\Var \mathcal{F}(a))^{-1} \Cov[\mathcal{F}(a),a^\RVE_{ij}]}{\Var a^\RVE_{ij}}.
\end{align}
c) The probability that a randomly chosen coefficient field $a$ satisfies the selection criterion \eqref{SQSConditioning} is at least
\begin{align}
\label{ProbabilitySQS}
\mathbb{P}\big[|\mathcal{F}(a)|\leq \delta L^{-d/2}\big] \geq c(N) \delta^N.
\end{align}
d) The systematic error and the variance of $a^\SQS$ may be estimated independently of $\kappa_{ij}$ at the price of lower rate of convergence in $L$
\begin{align}
\label{SystematicErrorSQS2}
\big|\mathbb{E}\big[a^\SQS\big]-a_\shom \big| \leq \frac{C}{\delta^N} L^{-d/2-d/8} |\log L|^{C(d,\gamma)}
\end{align}
and
\begin{align}
\label{VarianceReductionSQS2}
\frac{\Var a^\SQS_{ij}}{\Var a^\RVE_{ij}} \leq 1-(1-\delta^2) |\rho|^2 + \frac{Cr_{\operatorname{Var},ij}}{\delta^N} L^{-d/8} |\log L|^{C(d,\gamma)}.
\end{align}
\end{theorem}
The previous theorem states that the approximation of effective coefficients by the selection approach for representative volumes is essentially at least as accurate as a random selection of samples (except for a possible additional relative error of the order $C L^{-d/2} |\log L|^C$, which however converges to zero quickly as $L$ increases), at least when measuring the mean-square error. If the selection is based on a statistical quantity $\mathcal{F}(a)$ which is capable of explaining a large part of the variance of $a^\RVE_{ij}$, the selection approach achieves a much better accuracy than a random selection of samples (namely, by a factor of about $\sqrt{1-|\rho|^2}$).

However, the previous theorem only provides a statement about the reduction of the mean-square error by the selection approach for representative volumes. A natural question is whether this reduction of the error also applies to rare events: More precisely, if we fix a small probability $p>0$, is the bound on the error $|a^\SQS_{ij}-a_{\shom,ij}|$ which holds with probability $1-p$ also improved as suggested by the variance reduction estimate \eqref{VarianceReductionSQS}? The following theorem shows that this is in fact true for ``moderate deviations'', i.\,e.\ basically for probabilities $p\gtrsim \exp(-L^\beta)$ for some $\beta>0$. More precisely, the theorem is to be read as follows: Up to error terms that converge to zero as $L\rightarrow \infty$ and $s\rightarrow \infty$, the probability of $a^\SQS_{ij}$ deviating from $a_{\shom,ij}$ by more than $s$ times the ideally reduced standard deviation $\sqrt{(1-|\rho|^2)\Var a^\RVE_{ij}}$ behaves like the probability of a normal distribution deviating from its mean by more than $s$ standard deviations, at least in some regime $s\leq L^{\beta/3}$.
\begin{theorem}
\label{TheoremSQSModerateDeviations}
Let the assumptions and notations of Theorem~\ref{TheoremSQS} be in place. Suppose in addition $L\geq C$. Then the selection approach for representative volumes leads to a reduction of the ``outliers'' of the probability distribution of $a^\SQS$ in the sense of the moderate-deviations-type bound
\begin{align}
\label{ModerateDeviationsSQS}
&\mathbb{P}\left[\frac{\big|a^\SQS_{ij}-a_{\shom,ij}\big|}{\sqrt{(1+\frac{C\delta}{\sqrt{1-|\rho|^2}s})(1-|\rho|^2)\Var a_{ij}^\RVE + C L^{-d-\beta}}}
\geq s\right]
\\&~~~~~~~~~~~~
\nonumber
\leq \bigg(1+\frac{C}{\delta^N L^\beta}+\frac{C\delta}{\sqrt{1-|\rho|^2}s}\bigg)\mathbb{P}\big[|\mathcal{N}_1|\geq s\big]
+ \frac{C}{\delta^N} \exp\big(-L^{\beta}\big)
\end{align}
for any $s\geq \max\big\{1,\frac{\delta}{\sqrt{1-|\rho|^2}}\big\}$ and some $\beta=\beta(d)>0$.
\end{theorem}
We have shown in the preceding two theorems that the selection approach for representative volumes by Le~Bris et al.\ essentially does not increase the error; it succeeds in reducing the fluctuations of the approximations as soon as the functionals $\mathcal{F}(a)$ and the approximation $a^\RVE$ have a nonzero covariance.

However, as we shall show in the next theorem there exist cases in which the selection approach for representative volumes in fact fails to reduce the variance significantly, even for a ``natural'' statistical quantity like the average of the coefficient field
\begin{align*}
\mathcal{F}(a)
:=\dashint_{[0,L\varepsilon]^d} a \,dx.
\end{align*}

\begin{theorem}[Possible Failure of the Reduction of the Variance]
\label{TheoremFailureVarianceReduction}
Suppose that the assumptions of Theorem~\ref{TheoremSQS} hold. Then the estimate \eqref{VarianceReductionSQS} on the reduction of the variance is sharp in the sense
\begin{align}
\label{FailureVarianceReduction}
\frac{\Var a^\SQS_{ij}}{\Var a^\RVE_{ij}} \geq 1-|\rho|^2 - \frac{C \kappa^{3/2}_{ij}r_{\operatorname{Var},ij}}{\delta^N} L^{-d/2} |\log L|^{C(d,\gamma)}.
\end{align}
Furthermore, for $d\geq 2$ there exist $L\varepsilon$-periodic probability distributions of coefficient fields $a$ which satisfy the conditions of ellipticity, discrete stationarity, and finite range of dependence (A1), (A2'), (A3$_a$) - (A3$_c$) with the following property: The covariance of $a^\RVE$ and the spatial average $\dashint a$ vanishes
\begin{align}
\label{FailureVarianceReductionCov}
\Cov\left[a^\RVE ~,~~\dashint_{[0,L\varepsilon]^d} a \,dx\right]
=0,
\end{align}
while the fluctuations of $a^\RVE$ and $\dashint_{[0,L\varepsilon]^d} a$ are nondegenerate in the sense
\begin{align*}
\Var a^\RVE &\geq c L^{-d} \Id\otimes \Id,
\\
\Var \dashint_{[0,L\varepsilon]^d} a \,dx &\geq c L^{-d} \Id \otimes \Id,
\end{align*}
for some universal constant $c$.
These coefficient fields may be chosen to be of the form $a(x)=\tilde a(x)\Id$ for some scalar random field $\tilde a$.

As a consequence, for these probability distributions of coefficient fields the selection approach for representative volumes based on the spatial average $\dashint a$ fails to efficiently reduce the variance in the sense
\begin{align}
\label{FailureVarianceReduction2}
\frac{\Var a^\SQS_{ij}}{\Var a^\RVE_{ij}}
\geq 1-\frac{C\kappa^{3/2}_{ij}r_{\operatorname{Var},ij}}{\delta^N} L^{-d/2} |\log L|^{C(d,\gamma)}.
\end{align}
\end{theorem}
Let us note that it is presumably not too difficult to replace the random checkerboard in our construction of the counterexample featuring \eqref{FailureVarianceReductionCov} by random spherical inclusions distributed according to a Poisson point process (with overlaps of the inclusions). This would yield a counterexample subject to the continuous stationarity (A2).

The next theorem suggests that the failure of effective variance reduction is atypical and may be limited to rather artificial examples: For a large class of random coefficient fields -- namely for coefficient fields that are obtained from a collection of iid random variables $\xi_{k}$, $k\in \varepsilon \mathbb{Z}^d$, by applying a stationary monotone map with finite range of dependence -- the correlation coefficient between $a^\RVE$ and the average $\mathcal{F}(a):=\dashint a$ is bounded from below by a positive number. Therefore, for such (ensembles of) coefficient fields both the method of special quasirandom structures and the method of control variates in fact reduce the variance by some factor $\tau<1$ when applied with the choice $\mathcal{F}(a):=\dashint a$.
\begin{proposition}[Reduction of the Variance for a Large Class of Coefficient Fields]
\label{PropositionLowerBoundOnCorrelation}
Let $\varepsilon>0$ and let $L\geq 2$ be an integer and let $V$ denote some measure space.
Let $(\Gamma_k)$, $k\in \varepsilon\mathbb{Z}^d\cap [0,L\varepsilon)^d$, be a collection of independent identically distributed $V$-valued random variables, and denote by $(\tilde \Gamma_k)$ an independent copy. Extend $\Gamma_k$ to $k\in \varepsilon\mathbb{Z}^d$ by $L\varepsilon$-periodicity. For $k\in \varepsilon \mathbb{Z}^d$ and $z\in V$, denote by $\Delta_{k,z} \Gamma$ the collection $(\tilde \Gamma_k)$ obtained by setting $\tilde \Gamma_k:=z$ and $\tilde \Gamma_j=\Gamma_j$ for all $j\neq k$.

Let $a=a(x,\Gamma)$ be a measurable map into the uniformly elliptic $L\varepsilon$-periodic symmetric coefficient fields with the property that $a(x,\Gamma)$ depends only on the $\Gamma_k$ with $|x-k|_\per \leq K\varepsilon$ for some $K\geq 1$ (in a measurable way). Suppose that the map is stationary in the sense that $a(x+y,\Gamma)=a(x,\Gamma_{\cdot+y})$ for any $y\in \varepsilon\mathbb{Z}^d$.

Suppose that the dependence of $a$ on $\Gamma$ is monotone in the sense that for every $k\in \varepsilon \mathbb{Z}^d$ and every pair $z_1,z_2 \in V$ either for all $x$ the inequality
\begin{align*}
a(x,\Delta_{k,z_1}\Gamma)\geq a(x,\Delta_{k,z_2} \Gamma)
\end{align*}
holds or for all $x$ the reverse inequality
\begin{align*}
a(x,\Delta_{k,z_1}\Gamma)\leq a(x,\Delta_{k,z_2} \Gamma)
\end{align*}
holds. Suppose furthermore that there exists $\nu>0$ such that we have the quantified monotonicity
\begin{align}
\label{QuantifiedMonotonicity}
&\mathbb{E}
\Bigg[
\sum_{k\in \varepsilon \mathbb{Z}^d \cap [0,L\varepsilon)^d}
\sqrt{\varepsilon^{-d}\int_{[0,L\varepsilon]^d} \big|(a-a(\Delta_{k,\tilde \Gamma_k}\Gamma))\xi\cdot \xi\big| \,dx}
\big(a(x,\Gamma)-a(x,\Delta_{k,\tilde \Gamma_k}\Gamma)\big)_+^{1/2}
~\Bigg|~\Gamma
\Bigg]
\\&\nonumber
\geq \nu \Id
\end{align}
for all $x\in [0,L \varepsilon)^d$ and all $\Gamma$, where $\big(a(x,\Gamma)-a(x,\Delta_{k,\tilde \Gamma_k}\Gamma)\big)_+^{1/2}$ denotes the matrix square root and where $\tilde \Gamma$ denotes an independent copy of $\Gamma$.

Then the probability distribution of $a=a(x,\Gamma)$ satisfies the conditions of ellipticity, periodicity, and finite range of dependence (A1), (A3$_a$), and (A3$_b$) (with $\varepsilon$ replaced by $4K\varepsilon$), as well as the discrete stationarity (A2'). Furthermore, for such coefficient fields $a$ the correlation between $\xi \cdot a^\RVE\xi$ (where $\xi\in \mathbb{R}^d$ is any nonzero vector) and the average
\begin{align*}
\mathcal{F}_{avg}(a):=\dashint_{[0,L\varepsilon]^d} \xi \cdot a\xi \,dx
\end{align*}
is bounded from below by a positive number in the sense
\begin{align*}
\rho = \frac{\Cov[a^\RVE_{ij},\mathcal{F}_{avg}(a)]}{\sqrt{\Var a^\RVE_{ij} ~ \Var \mathcal{F}_{avg}(a)}}
\geq
\frac{\nu^2}{C(d,\lambda,K)}.
\end{align*}
\end{proposition}

In the statements of our main theorems, we have made use of the following notion of ``multilevel local dependence decomposition''; this structure will also be at the heart of the proof of our main results. An illustration of this decomposition is provided in Figure~\ref{FigureMultilevel}.

\begin{definition}[Sums of Random Variables with Multilevel Local Dependence Structure]
\label{ConditionRandomVariable}
Let $d\geq 1$, $N\in \mathbb{N}$, $\varepsilon>0$, and $L\geq 2$. Consider a probability distribution of coefficient fields $a$ on $\mathbb{R}^d$ subject to the assumptions of ellipticity and boundedness, stationarity, and finite range of dependence $\varepsilon$ (A1), (A2), and (A3), or the periodization of such an ensemble subject to the conditions (A1), (A2), and (A3$_a$) - (A3$_c$). Let $X=X(a)$ be an $\mathbb{R}^N$-valued random variable of the periodized ensemble.

We then say that $X$ is a sum of random variables with multilevel local dependence if there exist random variables $X_y^m=X_y^m(a)$, $0\leq m\leq 1+\log_2 L$ and $y\in 2^m \varepsilon\mathbb{Z}^d\cap [0,L\varepsilon)^d$, and constants $K\geq 1$, $\gamma\in (0,2]$, and $B\geq 1$ with the following properties:
\begin{itemize}
\item The random variable $X_y^m(a)$ only depends on $a|_{y+K \log L \, [-2^m \varepsilon,2^m \varepsilon]^d}$. More precisely, $X_y^m(a)$ is a measurable function of $a|_{y+K \log L \, [-2^m \varepsilon,2^m \varepsilon]^d}$ equipped with the topology of $H$-convergence.
\item We have
\begin{align*}
X=\sum_{m=0}^{1+\log_2 L} \sum_{y\in 2^m \varepsilon \mathbb{Z}^d\cap [0,L\varepsilon)^d} X_y^m.
\end{align*}
\item The random variables $X_y^m$ satisfy the bound
\begin{align}
\label{BoundMultilevelDependenceStructure}
||X_y^m||_{\exp^\gamma} \leq B L^{-d}.
\end{align}
\end{itemize}
\end{definition}

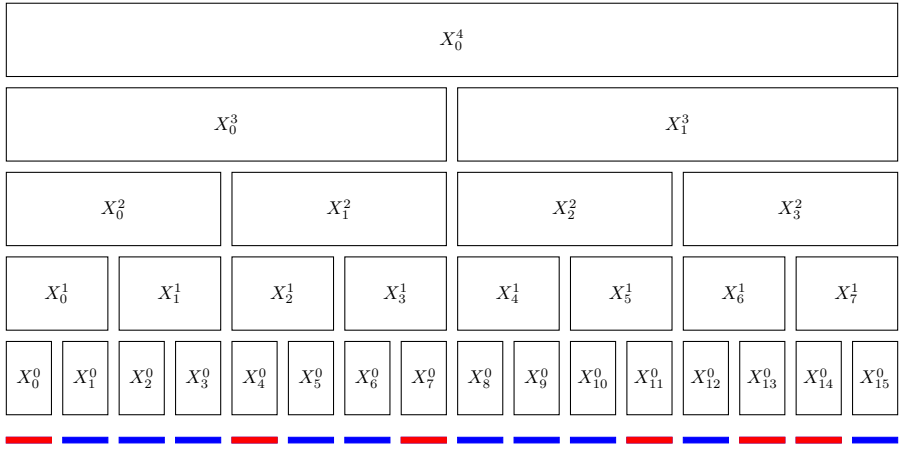
\begin{figure}
\begin{tikzpicture}[scale=0.75]
\foreach \x in {0}
   \draw (16*\x+0.1,4*1.5+0.1) rectangle (16*\x+15.9,4*1.5+1.4);
\foreach \x in {0,1}
   \draw (8*\x+0.1,3*1.5+0.1) rectangle (8*\x+7.9,3*1.5+1.4);
\foreach \x in {0,1,2,3}
   \draw (4*\x+0.1,2*1.5+0.1) rectangle (4*\x+3.9,2*1.5+1.4);
\foreach \x in {0,1,2,3,4,5,6,7}
   \draw (2*\x+0.1,1*1.5+0.1) rectangle (2*\x+1.9,1*1.5+1.4);
\foreach \x in {0,1,2,3,4,5,6,7,8,9,10,11,12,13,14,15}
   \draw (\x+0.1,0*1.5+0.1) rectangle (\x+0.9,0*1.5+1.4);
\foreach \x in {0,1,2,3,4,5,6,7,8,9,10,11,12,13,14,15}
   \draw[draw=blue,fill=blue] (\x+0.1,-0.5+0.1) rectangle (\x+0.9,-0.5+0.2);
\foreach \x in {0,4,7,11,13,14}
   \draw[draw=red,fill=red] (\x+0.1,-0.5+0.1) rectangle (\x+0.9,-0.5+0.2);
\foreach \x in {0,1,2,3,4,5,6,7,8,9,10,11,12,13,14,15}
		\node[scale=0.7] at (\x+0.5,0*1.5+0.75){$X_{\x}^0$};
\foreach \x in {0,1,2,3,4,5,6,7}
		\node[scale=0.7] at (2*\x+1.0,1*1.5+0.75){$X_{\x}^1$};
\foreach \x in {0,1,2,3}
		\node[scale=0.7] at (4*\x+2.0,2*1.5+0.75){$X_{\x}^2$};
\foreach \x in {0,1}
		\node[scale=0.7] at (8*\x+4.0,3*1.5+0.75){$X_{\x}^3$};
\foreach \x in {0}
		\node[scale=0.7] at (16*\x+8.0,4*1.5+0.75){$X_{\x}^4$};
\end{tikzpicture}
\caption{An illustration of the ``multilevel local dependence structure'' introduced in Definition~\ref{ConditionRandomVariable} (in a one-dimensional setting). At the bottom, a sample of the random coefficient field $a$ is depicted; the $X_y^k$ may depend not only on the values of the coefficient field directly below their box, but on the coefficient field in a region that is wider by a factor of $K \log L$.\label{FigureMultilevel}}
\end{figure}

The following proposition shows that the approximation $a^\RVE$ of the effective coefficient by the method of representative volumes may indeed be rewritten as a sum of random variables with a multilevel local dependence structure. We establish the same result for the spatial average of the coefficient field $\mathcal{F}_{avg}(a):=\dashint_{[0,L\varepsilon]^d} a \,dx$ and the second-order term $\mathcal{F}_{2-point}(a)$ in the low ellipticity contrast expansion of $a^\RVE$ given by \eqref{SecondOrderQuantity}.

Furthermore, the last result of the next proposition shows that the fraction of the variance of $a^\RVE$ that is explained by the statistical quantities $\mathcal{F}_{avg}(a)$ and $\mathcal{F}_{2-point}(a)$ -- that is, the gain in accuracy achieved by the selection approach for representative volumes when employing these statistical quantities -- stabilizes as the size $L$ of the representative volume increases; more precisely, it converges to some limit with rate $L^{-d/2}|\log L|^C$.

\begin{proposition}
\label{PropositionApproximabilityByMultilevel}
Let the assumptions (A1), (A2), (A3$_a$) - (A3$_c$) be satisfied, that is consider the periodization of a stationary ensemble of random coefficient fields. For any coefficient field $a$, denote by $\phi_i$ the unique (up to additions of constants) periodic solution to the corrector equation
\begin{align*}
-\nabla \cdot (a(e_i+\nabla \phi_i))=0.
\end{align*}
Then the approximation $a^\RVE$ of the effective coefficient $a_\shom$ by the representative volume element method, given by
\begin{align*}
a^\RVE e_i:=\dashint_{[0,L\varepsilon]^d} a(e_i+\nabla \phi_i) \,dx,
\end{align*}
is a sum of a family of random variables with multilevel local dependence. More precisely, $a^\RVE$ satisfies the criteria of Definition~\ref{ConditionRandomVariable} for any $\gamma<1$ with $K:=C(d,\lambda)$ and $B:=C(d,\gamma,\lambda) |\log L|^{C(d,\gamma)}$.

Furthermore, the spatial average
\begin{align*}
\mathcal{F}_{avg}(a):=\dashint_{[0,L\varepsilon]^d} a \,dx
\end{align*}
is also a sum of a family of random variables with multilevel local dependence. The criteria of Definition~\ref{ConditionRandomVariable} are satisfied by $\mathcal{F}_{avg}(a)$ for any $\gamma<\infty$ with $K:=C(d)$ and $B:=C(d,\gamma)$.

Additionally, the second-order correction to the effective conductivity in the setting of small ellipticity contrast $\mathcal{F}_{2-point}$, given by
\begin{align}
\label{F2}
\mathcal{F}_{2-point}(a):=-\dashint_{[0,L\varepsilon]^d} a \nabla v_i \cdot e_j \,dx
\end{align}
with $v_i$ denoting the solution to
\begin{align}
\label{F2Equation}
-\Delta v_i &= \nabla \cdot (a\nabla e_i),
\end{align}
is a sum of random variables with multilevel local dependence structure: The random variable $\mathcal{F}_{2-point}(a)$ satisfies the criteria of Definition~\ref{ConditionRandomVariable} for any $\gamma<1$ with $K:=C(d,\lambda)$ and $B:=C(d,\gamma,\lambda) |\log L|^{C(d,\gamma)}$.

Finally, the rescaled variances and covariances of $a^\RVE$ and the statistical quantities $\mathcal{F}_{avg}(a)$ and $\mathcal{F}_{2-point}(a)$ converge as $L\rightarrow \infty$: There exist positive semidefinite matrices $V_{\RVE}$, $V_{avg}$, $V_{2-point}$ and matrices $V_{c,\RVE,avg}$, $V_{c,\RVE,2-point}$, $V_{c,avg,2-point}$ independent of $L$ such that the estimates
\begin{align*}
|L^d \Var a^\RVE - V_{\RVE}|
\leq C L^{-d/2} (\log L)^C,
\\
|L^d \Var \mathcal{F}_{avg}(a) - V_{avg}|
\leq C L^{-d/2} (\log L)^C,
\\
|L^d \Var \mathcal{F}_{2-point}(a) - V_{2-point}|
\leq C L^{-d/2} (\log L)^C,
\end{align*}
and
\begin{align*}
|L^d \Cov[a^\RVE,\mathcal{F}_{avg}(a)] - V_{c,\RVE,avg}|
\leq C L^{-d/2} (\log L)^C,
\\
|L^d \Cov[a^\RVE,\mathcal{F}_{2-point}(a)] - V_{c,\RVE,2-point}|
\leq C L^{-d/2} (\log L)^C,
\\
|L^d \Cov[\mathcal{F}_{avg}(a),\mathcal{F}_{2-point}(a)] - V_{c,avg,2-point}|
\leq C L^{-d/2} (\log L)^C,
\end{align*}
hold true.
\end{proposition}

\section{Strategy of the proof and intermediate results}

Our main result relies on a quantitative normal approximation result for the joint probability distribution of the approximation of the effective conductivity $a^\RVE$ and auxiliary random variables $\mathcal{F}(a)$ like the spatial average $\dashint_{[0,L\varepsilon]^d} a\,dx$. The distance of the probability distribution to a multivariate Gaussian will be quantified through the following notion of distance between probability measures. Note that this distance is a standard choice in the theory of multivariate normal approximation, see e.\,g.\ \cite{ChenGoldsteinShao} and the references therein.

\begin{definition}
\label{DefinitionDistance}
Given a symmetric positive definite matrix $\Lambda\in \mathbb{R}^{N\times N}$ and some $\bar L<\infty$, we consider the classes $\Phi_{\Lambda}^{\bar L}$ of functions $\phi:\mathbb{R}^N\rightarrow \mathbb{R}$ subject to the following properties:
\begin{itemize}
\item $\phi$ is smooth and its first derivative is bounded in the sense $|\nabla \phi(x)| \leq \bar L$ for all $x\in \mathbb{R}^N$.
\item For any $r>0$ and any $x_0\in \mathbb{R}^N$, we have
\begin{align}
\label{DefinitionClassPhi}
\int_{\mathbb{R}^N} \osc_r \phi(x) ~\mathcal{N}_{\Lambda}(x-x_0) \,dx \leq r,
\end{align}
where $\osc_r \phi (x)$ is the oscillation of $\phi$ defined as
\begin{align*}
\osc_r\phi(x):=\sup_{|z|\leq r}\phi(x+z)-\inf_{|z|\leq r} \phi(x+z)
\end{align*}
and where 
\begin{align*}
\mathcal{N}_{\Lambda}(x):=\frac{1}{(2\pi)^{N/2}\sqrt{\det \Lambda}}
\exp\bigg(-\frac{1}{2}\Lambda^{-1} x \cdot x\bigg).
\end{align*}
\end{itemize}
The class $\Phi_\Lambda$ is defined as
\begin{align*}
\Phi_\Lambda:=\bigcup_{\bar L>0} \Phi_\Lambda^{\bar L}.
\end{align*}

Furthermore, we introduce the distance $\mathcal{D}$ between the law of an $\mathbb{R}^N$-valued random variable $X$ and the $N$-variate Gaussian $\mathcal{N}_\Lambda$ as
\begin{align}
\label{DefinitionD}
\mathcal{D}(X,\mathcal{N}_\Lambda) := \sup_{\phi\in \Phi_\Lambda} \bigg(\mathbb{E}[\phi(X)]-\int_{\mathbb{R}^N} \phi(x) \mathcal{N}_\Lambda(x)\,dx \bigg).
\end{align}
\end{definition}
Note that defining the distance $\mathcal{D}$ with the class of functions $\Phi_\Lambda^1$ instead of $\Phi_\Lambda$ would lead to the $1$-Wasserstein distance. The distance $\mathcal{D}$ is a stronger distance than the $1$-Wasserstein distance: The $1$-Wasserstein distance is defined by taking the supremum in \eqref{DefinitionD} only over all functions $\phi$ which are $1$-Lipschitz. In contrast, the condition \eqref{DefinitionClassPhi} corresponds more or less to a slightly stronger condition than an $L^1_{loc}$-type bound for $\nabla \phi$: It in particular implies by letting $r\rightarrow 0$
\begin{align}
\label{ClassPhiDifferentialBound}
\int_{\mathbb{R}^N} |\nabla \phi|(x) \mathcal{N}_{\Lambda}(x-x_0) \,dx \leq 1
\end{align}
for any $x_0\in \mathbb{R}^N$.

It is well-known that Stein's method of normal approximation allows to establish a quantitative result on normal approximation for sums of random variables with local dependence structure, see e.\,g.\ \cite{ChenGoldsteinShao,ChenShao,RinottRotar} and the references therein. However, the approximation of the effective coefficient $a^\RVE$ -- that is, the random variable $a^\RVE$ as defined by \eqref{EquationRVE} -- features global dependencies.
It is shown in Proposition~\ref{PropositionApproximabilityByMultilevel} that $a^\RVE$ may nevertheless be approximated by a sum of random variables with a \emph{multilevel local dependence} structure.
We then employ the following quantitative central limit theorem for sums of vector-valued random variables with a multilevel local dependence structure, which is not covered by the normal approximation results for sums of random variables with a given dependency graph in the literature and which is established in the companion article \cite{FischerMultilevelLocalDependence}.

\begin{theorem}[{\cite[Theorem~4]{FischerMultilevelLocalDependence}}]
\label{TheoremNormalApproximationMultilevelLocalDependence}
Consider a probability distribution of uniformly elliptic and bounded coefficient fields $a$ on $\mathbb{R}^d$ or a periodization of such a probability distribution, and suppose that assumptions (A1)-(A3) respectively (A1), (A2), (A3$_a$)-(A3$_c$) are satisfied. Let $X=X(a)$ be a random variable that is a sum of random variables with multilevel local dependence in the sense of Definition~\ref{ConditionRandomVariable}. Then the law of the random variable $X$ is close to a multivariate Gaussian in the sense
\begin{align}
\label{NormalApproximationMultilevel}
&\mathcal{D}(X-\mathbb{E}[X],\mathcal{N}_\Lambda)
\leq
C(d,\gamma,N,K) B^3 (\log L)^{C(d,\gamma)} \big(L^{-d} |\Lambda^{1/2}| |\Lambda^{-1/2}|^3\big) L^{-d},
\end{align}
where $\Lambda:=\Var X$ and where the constant $C(d,\gamma,N,K)$ depends in a polynomial way on $d$, $N$, and $K$.

Furthermore, we have for any symmetric positive definite $\Lambda\in \mathbb{R}^{d\times d}$ with $\Lambda\geq \Var X$ and $|\Lambda-\Var X|\leq L^{-d}$
\begin{align}
\label{NormalApproximationMultilevelDegenerate}
\mathcal{D}(X-\mathbb{E}[X],\mathcal{N}_\Lambda)
\leq&
C(d,\gamma,N,K) B^3 (\log L)^{C(d,\gamma)} \big(L^{-d} |\Lambda^{1/2}| |\Lambda^{-1/2}|^3\big) L^{-d}
\\&
\nonumber
+C(d,N) (\log L)^{C(d,\gamma)} |\Lambda-\Var X|^{1/2},
\end{align}
providing a better bound in the case of degenerate covariance matrices $\Var X$.
\end{theorem}
Our result on moderate deviations of the probability distribution of $a^\SQS$ is based on the following simple general moderate deviations result for sums of random variables with multilevel local dependence structure.
\begin{theorem}[{\cite[Theorem~5]{FischerMultilevelLocalDependence}}]
\label{TheoremModerateDeviations}
Consider an ensemble of coefficient fields $a$ on $\mathbb{R}^d$, $d\geq 1$, or its periodization for some $L\geq 1$, subject to the conditions (A1)-(A3) respectively (A1), (A2), and (A3$_a$)-(A3$_c$). Let $X=X(a)$ be a random variable that may be written as a sum of random variables with multilevel local dependence structure $X=\sum_{m=0}^{1+\log_2 L} \sum_{i\in 2^m \varepsilon \mathbb{Z}^d \cap [0,L\varepsilon)^d} X_i^m$ in the sense of Definition~\ref{ConditionRandomVariable}.

Then there exists $\beta=\beta(d,\gamma)>0$ and a positive definite symmetric matrix $\Lambda\in \mathbb{R}^{N\times N}$ with $|\Lambda-\Var X|\leq C(d,\gamma,N,K) B^2 L^{-2\beta} L^{-d}$ such that for any measurable $A\subset \mathbb{R}^N$ we have the estimate
\begin{align*}
\mathbb{P}\big[X\in A\big]
\leq \int_{\{x\in \mathbb{R}^N:\dist(x,A)\leq L^{-\beta} L^{-d/2}\}} \mathcal{N}_{\Lambda}(x) \,dx + C(d,\gamma,N,K) \exp\Big(-\frac{c}{B^C} L^{2\beta}\Big).
\end{align*}
\end{theorem}

\section{Justification of the selection approach for representative volumes}

We now provide the proof of our main result -- the error estimates for the selection approach for representative volumes by Le~Bris, Legoll, and Minvielle \cite{LeBrisLegollMinvielle} -- which is stated in Theorem~\ref{TheoremSQS} and Theorem~\ref{TheoremSQSModerateDeviations}.

The idea for the proof of all statements of Theorem~\ref{TheoremSQS} is the following: Theorem~\ref{TheoremNormalApproximationMultilevelLocalDependence} enables us in conjunction with Proposition~\ref{PropositionApproximabilityByMultilevel} to approximate the joint probability distribution of $a^\RVE$ and $\mathcal{F}(a)$ by a multivariate Gaussian with the same covariance matrix. The probability distribution of $a^\SQS$ arises as the probability distribution of $a^\RVE$ conditioned on the event \eqref{SQSConditioning}. As a consequence, the probability distribution of $a^\SQS$ may be approximated by the marginal of the conditional probability distribution of an ideal multivariate Gaussian. The results of Theorem~\ref{TheoremSQS} on the probability distribution of $a^\SQS$ are then a consequence of corresponding properties of multivariate normal distributions.

\begin{proof}[Proof of Theorem~\ref{TheoremSQS}]
For the proof of the theorem we may assume without loss of generality that $\mathbb{E}[\mathcal{F}(a)]=0$. Throughout the proof, the constants $c$ and $C$ may depend on $d$, $\lambda$, $N$, $\gamma$, $c_0$, and $C_0$, if not otherwise stated.

Recall that the probability distribution of $a^\SQS$ is given by the probability distribution of $a^\RVE$ conditioned on the event \eqref{SQSConditioning}. Theorem~\ref{TheoremNormalApproximationMultilevelLocalDependence} and Proposition~\ref{PropositionApproximabilityByMultilevel} entail that the joint probability distribution of any component $a^\RVE_{ij}$ of $a^\RVE$ and $\mathcal{F}(a)$ is close to a multivariate Gaussian $\mathcal{N}_{\Var (a^\RVE_{ij},\mathcal{F}(a))}(\cdot\,-\mathbb{E}[a^\RVE_{ij}],\cdot)$. As a consequence of this result, the probability distribution of $a^\SQS_{ij}$ may be approximated in a quantitative sense by the first-variable marginal of the conditional distribution of $\mathcal{N}_{\Var (a^\RVE_{ij},\mathcal{F}(a))}(\cdot\,-\mathbb{E}[a^\RVE_{ij}],\cdot)$ given the event $|\mathcal{F}(a)|\leq \delta L^{-d/2}$. As we shall show below, the latter marginal probability distribution has the density
\begin{align}
\label{LimitDistribution}
\mathcal{M}^\delta(x)&:=
\frac{1}{p} \int_{\mathbb{R}^{N}} \mathcal{N}_{\Varaideal}\big(x-\Cov[a^\RVE_{ij},\mathcal{F}(a)](\Var \mathcal{F}(a))^{-1} y-\mathbb{E}[a^\RVE_{ij}]\big)
\\&~~~~~~~~~~~~~~~~~~~
\nonumber
\times \chi_{\{|y|\leq \delta L^{-d/2}\}} \mathcal{N}_{\Var \mathcal{F}(a)}(y) \,dy
\end{align}
where the renormalization factor $p$ is given by
\begin{align*}
p=\int_{\mathbb{R}} \int_{\mathbb{R}^{N}} \chi_{\{|y|\leq \delta L^{-d/2}\}}(y) \mathcal{N}_{\Var (a^\RVE_{ij},\mathcal{F}(a))}(x,y) \,dy \,dx
\end{align*}
and where the unexplained variance $\Varaideal$ (i.\,e.\ the variance of $a^\RVE_{ij}$ which is not explained by the $\mathcal{F}_n(a)$) is given by
\begin{align*}
\Varaideal = \Var a^\RVE_{ij} - \Cov[a^\RVE_{ij},\mathcal{F}(a)] (\Var \mathcal{F}(a))^{-1} \Cov[\mathcal{F}(a),a^\RVE_{ij}].
\end{align*}
The assertions \eqref{SystematicErrorSQS} and \eqref{VarianceReductionSQS} on the systematic error and the variance reduction in Theorem~\ref{TheoremSQS} will be a consequence of the lower bound \eqref{ProbabilitySQS} on the probability of a random coefficient field satisfying the selection criterion, the related lower bound
\begin{align}
\label{ProbabilitySQSNormal}
\int_{\mathbb{R}^N} \int_{\mathbb{R}} \chi_{\{|y|\leq \delta L^{-d/2}\}} \mathcal{N}_{\Var (a^\RVE_{ij},\mathcal{F}(a))}(x,y) \,dx \,dy \geq c(N) C_0^{-N/2} \delta^N,
\end{align}
the stretched exponential moment bounds for any $\gamma<1/2$
\begin{subequations}
\begin{align}
\label{aRVEStretchedExponentialBound}
||a^\RVE-\mathbb{E}[a^\RVE]||_{\exp^\gamma} &\leq C(d,\lambda,\gamma) L^{-d/2} |\log L|^{C},
\\
\label{NStretchedExponentialBound}
||\mathcal{N}_{\Var (a^\RVE,\mathcal{F}(a))}||_{\exp^\gamma} &
\leq C(d,\lambda,\gamma,C_0) L^{-d/2} |\log L|^{C},
\end{align}
\end{subequations}
and the approximation result of the distribution of $a^\SQS_{ij}$ by $\mathcal{M}^\delta$
\begin{align}
\label{SQSNormalApproximation}
\bigg|\mathbb{E}\big[\tilde \phi(a^\SQS_{ij})\big]-\int_{\mathbb{R}} \tilde \phi(x) \mathcal{M}^\delta(x) \,dx\bigg|
\leq
\frac{C \kappa_{ij}^{3/2}}{\delta^N} L^{-d} |\log L|^{C(d,\gamma)}
\end{align}
for any continuous $\tilde \phi:\mathbb{R}\rightarrow \mathbb{R}$ satisfying
\begin{subequations}
\label{BoundTildePhi}
\begin{align}
\label{BoundTildePhia}
|\tilde \phi|\leq L^{-d/2}
\end{align}
and
\begin{align}
\label{BoundTildePhib}
\int_{\mathbb{R}}\osc_r \tilde \phi (x) \mathcal{N}_{\Varaideal} (x-x_0) \,dx
\leq r
\end{align}
for all $r>0$ and all $x_0\in \mathbb{R}$.
\end{subequations}
To obtain the $\kappa$-independent estimates \eqref{SystematicErrorSQS2} and \eqref{VarianceReductionSQS2}, the bound \eqref{SQSNormalApproximation} is replaced by
\begin{align}
\label{SQSNormalApproximation2}
\bigg|\mathbb{E}\big[\tilde \phi(a^\SQS_{ij})\big]-\int_{\mathbb{R}} \tilde \phi(x) \mathcal{M}^\delta(x) \,dx\bigg|
\leq
\frac{C}{\delta^N} L^{-d/2-d/8} |\log L|^{C(d,\gamma)}.
\end{align}

We defer the proof of \eqref{ProbabilitySQS} and \eqref{SQSNormalApproximation} (as well as \eqref{SQSNormalApproximation2}) to the last step and first demonstrate that these estimates entail the assertions \eqref{SystematicErrorSQS} and \eqref{VarianceReductionSQS} of our theorem.

{\bf Step 1: Estimate on the systematic error.}
In order to derive the estimate on the systematic error \eqref{SystematicErrorSQS}, we first use the formula \eqref{LimitDistribution} and Fubini's theorem to see that
\begin{align}
\label{FormulaExpectationIdealGaussian}
&\int x ~ \mathcal{M}^\delta(x) \,dx
\\&
\nonumber
= \frac{1}{p} \int_{\mathbb{R}^{N}} \big(\mathbb{E}[a^\RVE_{ij}] + \Cov[a^\RVE_{ij},\mathcal{F}(a)](\Var \mathcal{F}(a))^{-1} y \big)
\\&
\nonumber
~~~~~~~~~~~~~\times\chi_{\{|y|\leq \delta L^{-d/2}\}}
 \mathcal{N}_{\Var \mathcal{F}(a)}(y) \,dy
\\&
\nonumber
=\mathbb{E}[a^\RVE_{ij}],
\end{align}
where in the second step we have used the symmetry of the Gaussian $\mathcal{N}_{\Var \mathcal{F}(a)}$. In other words, if the probability distribution of $(a^\RVE,\mathcal{F}(a))$ were an ideal multivariate Gaussian, we would have the perfect equality $\mathbb{E}[a^\SQS]=\mathbb{E}[a^\RVE]$.

We would now like to transfer the property \eqref{FormulaExpectationIdealGaussian} (up to an error) from $\mathcal{M}^\delta$ to our actual probability distribution $a^\SQS$ by choosing $\tilde \phi(x):=x$ in the estimate \eqref{SQSNormalApproximation}. However, this choice is not possible due to the upper bound on $\tilde \phi$ in \eqref{BoundTildePhia}. Instead, for some cutoff factor $B_c\geq 1$ we consider the function $\tilde \phi(x) = \min\{\max\{x-\mathbb{E}[a^\RVE_{ij}],-B_c L^{-d/2}\},B_c L^{-d/2}\}$. Note that for this choice of $\tilde \phi$ we have $|\nabla \tilde \phi|\leq 1$ and $|\tilde \phi|\leq B_c L^{-d/2}$. As a consequence, $\frac{1}{B_c}\tilde \phi$ satisfies \eqref{BoundTildePhi} and hence is an admissible choice in \eqref{SQSNormalApproximation}, which gives by \eqref{FormulaExpectationIdealGaussian}
\begin{align*}
&\big| \mathbb{E}[a^\SQS_{ij}]-\mathbb{E}[a^\RVE_{ij}] \big|
\\&
=\bigg|\mathbb{E}[a^\SQS_{ij}-\mathbb{E}[a^\RVE_{ij}]]
-\int_{\mathbb{R}} (x-\mathbb{E}[a^\RVE_{ij}]) ~ \mathcal{M}^\delta(x) \,dx
\bigg|
\\&
\leq
\mathbb{E}\big[\big|(a^\SQS_{ij}-\mathbb{E}[a^\RVE_{ij}])-\tilde \phi(a^\SQS_{ij})\big|\big]
\\&~~~
+\int_{\mathbb{R}} |(x-\mathbb{E}[a^\RVE_{ij}])-\tilde \phi(x)| ~ \mathcal{M}^\delta(x) \,dx
\\&~~~
+\bigg|\mathbb{E}[\tilde \phi(a^\SQS_{ij})]
-\int_{\mathbb{R}} \tilde \phi(x) ~ \mathcal{M}^\delta(x) \,dx
\bigg|
\\&
\stackrel{\eqref{SQSNormalApproximation}}{\leq}
\mathbb{E}\big[\big(|a^\SQS_{ij}-\mathbb{E}[a^\RVE_{ij}]|-B_c L^{-d/2}\big)_+\big]
\\&~~~
+\int_{\mathbb{R}} \big(|x-\mathbb{E}[a^\RVE_{ij}]|-B_c L^{-d/2}\big)_+ ~ \mathcal{M}^\delta(x) \,dx
\\&~~~
+B_c\frac{C\kappa_{ij}^{3/2}}{\delta^N} L^{-d} |\log L|^{C(d,\gamma)}.
\end{align*}
Using first the lower bounds \eqref{ProbabilitySQS} and \eqref{ProbabilitySQSNormal} and the representation \eqref{LimitDistributionDirect} and then in the next step H\"older's inequality, the previous estimate implies
\begin{align*}
&\big|\mathbb{E}[a^\SQS_{ij}]
-\mathbb{E}[a^\RVE_{ij}]
\big|
\\&
\leq
\frac{C}{\delta^N}
\mathbb{E}\big[\big(|a^\RVE_{ij}-\mathbb{E}[a^\RVE_{ij}]|-B_c L^{-d/2}\big)_+\big]
\\&~~~
+\frac{C(N)}{\delta^N}
\int_{\mathbb{R}} \int_{\mathbb{R}^N} \big(|x-\mathbb{E}[a^\RVE_{ij}]|-B_c L^{-d/2}\big)_+
\\&~~~~~~~~~~~~~~~~~~~~~~~~~~~~~~\times
\mathcal{N}_{\Var (a^\RVE_{ij},\mathcal{F}(a))}(x-\mathbb{E}[a^\RVE_{ij}],y) \,dy \,dx
\\&~~~
+B_c \frac{C\kappa^{3/2}_{ij}}{\delta^N} L^{-d} |\log L|^{C(d,\gamma)}
\\&
\leq
\frac{C}{\delta^N}
\mathbb{E}\big[|a^\RVE_{ij}-\mathbb{E}[a^\RVE_{ij}]|^2\big]^{1/2}
\mathbb{P}\big[|a^\RVE_{ij}-\mathbb{E}[a^\RVE_{ij}]|\geq B_c L^{-d/2}\big]^{1/2}
\\&~~~
+\frac{C(N)}{\delta^N}
\mathbb{E}\big[|\mathcal{N}_{\Var a^\RVE_{ij}}|^2\big]^{1/2} \mathbb{P}\big[|\mathcal{N}_{\Var a^\RVE_{ij}}|\geq B_c L^{-d/2}\big]^{1/2}
\\&~~~
+B_c \frac{C\kappa^{3/2}_{ij}}{\delta^N} L^{-d} |\log L|^{C(d,\gamma)}.
\end{align*}
This yields by Lemma~\ref{CalculusStretchedExponential}b and the bounds \eqref{aRVEStretchedExponentialBound} and \eqref{NStretchedExponentialBound}
\begin{align*}
&\big|\mathbb{E}[a^\SQS_{ij}]
-\mathbb{E}[a^\RVE_{ij}]
\big|
\\&~~~~~~~~~~
\leq
\frac{C(N)}{\delta^N} \exp\bigg(-c\bigg(\frac{B_c}{|\log L|^C}\bigg)^\gamma\bigg)
+B_c \frac{C\kappa_{ij}^{3/2}}{\delta^N} L^{-d} |\log L|^{C(d,\gamma)}.
\end{align*}
Choosing $B_c:=C|\log L|^{C(\gamma)}$,
we deduce
\begin{align}
\label{DifferenceSystematicError}
\big|\mathbb{E}[a^\SQS_{ij}]
-\mathbb{E}[a^\RVE_{ij}]
\big|
\leq \frac{C\kappa_{ij}^{3/2}}{\delta^N} L^{-d} |\log L|^{C(d,\gamma)}.
\end{align}
Plugging in the bound for the systematic error of the standard representative volume element method $|\mathbb{E}[a^\RVE]-a_\shom|\leq C L^{-d} |\log L|^C$ from \cite{GloriaOttoNew} (note that this estimate for the systematic error of the standard representative volume element method may also be derived by slightly modifying the proof of our Proposition~\ref{PropositionApproximabilityByMultilevel}), we obtain \eqref{SystematicErrorSQS}. Repeating the previous proof but replacing the use of the estimate \eqref{SQSNormalApproximation} by \eqref{SQSNormalApproximation2}, we obtain \eqref{SystematicErrorSQS2}.

{\bf Step 2: Proof of the variance reduction estimate.}
To prove the variance estimate \eqref{VarianceReductionSQS}, we proceed similarly and define for a cutoff factor $B_c\geq 1$ the function $\phi(x):=\min\{(x-\mathbb{E}[a^\RVE_{ij}])^2, B_c^2 L^{-d}\}$. Note that this function satisfies the global bounds $|\nabla \phi|\leq 2B_c L^{-d/2}$ and $|\phi|\leq B_c^2 L^{-d}$. Thus, $\frac{1}{2 B_c^2 L^{-d/2}} \phi$ satisfies \eqref{BoundTildePhi} and is therefore an admissible choice in \eqref{SQSNormalApproximation}, yielding
\begin{align}
\label{VarianceEstimateSQSIntermediate}
&\bigg|\mathbb{E}\big[\min\{(a^\SQS_{ij}-\mathbb{E}[a^\RVE_{ij}])^2,B_c^2 L^{-d} \}\big]
\\&~~~~
\nonumber
-\int_{\mathbb{R}} \min\{(x-\mathbb{E}[a^\RVE_{ij}])^2, B_c^2 L^{-d} \} \mathcal{M}^\delta(x) \,dx\bigg|
\\&
\nonumber
\leq
2B_c^2 L^{-d/2} \cdot \frac{C\kappa^{3/2}_{ij}}{\delta^N} L^{-d} |\log L|^{C(d,\gamma)}.
\end{align}
The tails (subject to truncation in our choice of $\phi$) can be estimated by
\begin{align*}
&\mathbb{E}\big[\big|(a^\SQS_{ij}-\mathbb{E}[a^\RVE_{ij}])^2-\phi(a^\SQS_{ij})\big|\big]
\\&
+\int_{\mathbb{R}} |(x-\mathbb{E}[a^\RVE_{ij}])^2-\phi(x)| ~ \mathcal{M}^\delta(x) \,dx
\\&
\leq
\mathbb{E}\big[\big(|a^\SQS_{ij}-\mathbb{E}[a^\RVE_{ij}]|^2-B_c^2 L^{-d}\big)_+\big]
\\&~~~
+\int_{\mathbb{R}} \big(|x-\mathbb{E}[a^\RVE_{ij}]|^2-B_c^2 L^{-d}\big)_+ ~ \mathcal{M}^\delta(x) \,dx
\\&
\leq
\frac{C}{\delta^N} \mathbb{E}\big[\big(|a^\RVE_{ij}-\mathbb{E}[a^\RVE_{ij}]|^2-B_c^2 L^{-d}\big)_+\big]
\\&~~~
+\frac{C}{\delta^N} \int_{\mathbb{R}} \int_{\mathbb{R}^N} \big(|x-\mathbb{E}[a^\RVE_{ij}]|^2-B_c^2 L^{-d}\big)_+ ~ \mathcal{N}_{\Var (a^\RVE_{ij},\mathcal{F}(a))}(x,y) \,dy \,dx,
\end{align*}
where in the last step we have used \eqref{ProbabilitySQS}, \eqref{ProbabilitySQSNormal}, and \eqref{LimitDistributionDirect}. Applying H\"older's inequality, we obtain
\begin{align*}
&\mathbb{E}\big[\big|(a^\SQS_{ij}-\mathbb{E}[a^\RVE_{ij}])^2-\phi(a^\SQS_{ij})\big|\big]
\\&
+\int_{\mathbb{R}} |(x-\mathbb{E}[a^\RVE_{ij}])^2-\phi(x)| ~ \mathcal{M}^\delta(x) \,dx
\\&
\leq
\frac{C}{\delta^N} \mathbb{E}\big[|a^\RVE_{ij}-\mathbb{E}[a^\RVE_{ij}]|^4\big]^{1/2}
\mathbb{P}\big[|a^\RVE_{ij}-\mathbb{E}[a^\RVE_{ij}]|\geq B_c L^{-d/2}\big]^{1/2}
\\&~~~
+\frac{C}{\delta^N} \mathbb{E}[|\mathcal{N}_{\Var a^\RVE_{ij}}|^4]^{1/2} \cdot \mathbb{P}\big[|\mathcal{N}_{\Var a^\RVE_{ij}}|\geq B_c L^{-d/2}\big]^{1/2}
\\&
\leq
\frac{C}{\delta^N} \exp\bigg(-c\bigg(\frac{B_c}{|\log L|^C}\bigg)^\gamma\bigg),
\end{align*}
where in the last step we have used Lemma~\ref{CalculusStretchedExponential}b and the bounds \eqref{aRVEStretchedExponentialBound} and \eqref{NStretchedExponentialBound}.

Combining this estimate with \eqref{VarianceEstimateSQSIntermediate} and choosing $B_c:=C|\log L|^{C(d,\gamma)}$, we infer
\begin{align}
\label{SQSErrorVariance}
&\bigg|\mathbb{E}\big[(a^\SQS_{ij}-\mathbb{E}[a^\RVE_{ij}])^2\big]
-\int (x-\mathbb{E}[a^\RVE_{ij}])^2 \mathcal{M}^\delta(x) \,dx\bigg|
\\&~~~
\nonumber
\leq \frac{C\kappa^{3/2}_{ij}}{\delta^N} L^{-3d/2} |\log L|^{C(d,\gamma)}.
\end{align}
In other words, the variance of $a^\SQS_{ij}$ is determined up to an error by the variance of the probability distribution $\mathcal{M}^\delta$.
To estimate the latter, a straightforward computation yields
\begin{align*}
&\int (x-\mathbb{E}[a^\RVE_{ij}])^2 \mathcal{M}^\delta(x) \,dx
\\&
\stackrel{\eqref{LimitDistribution}}{=}\frac{1}{p} \int_{\mathbb{R}^{N}} \int_{\mathbb{R}}  \mathcal{N}_{\Varaideal}\big(x-\Cov[a^\RVE_{ij},\mathcal{F}(a)](\Var \mathcal{F}(a))^{-1} y-\mathbb{E}[a^\RVE_{ij}]\big)
\\&~~~~~~~~~~~~~~~~~~~
\nonumber
\times \chi_{\{|y|\leq \delta L^{-d/2}\}} \mathcal{N}_{\Var \mathcal{F}(a)}(y) \cdot (x-\mathbb{E}[a^\RVE_{ij}])^2 \,dx \,dy
\\&
=\frac{1}{p} \int_{\mathbb{R}^{N}} \int_{\mathbb{R}} \big(\tilde x + \Cov[a^\RVE_{ij},\mathcal{F}(a)](\Var \mathcal{F}(a))^{-1} y\big)^2 \mathcal{N}_{\Varaideal}(\tilde x) \,d\tilde x
\\&~~~~~~~~~~~~~~~~~~~
\nonumber
\times \chi_{\{|y|\leq \delta L^{-d/2}\}} \mathcal{N}_{\Var \mathcal{F}(a)}(y) \,dy.
\end{align*}
By the symmetry of the set $\{|y|\leq \delta L^{-d/2}\}$ and the probability density $\mathcal{N}_{\Var \mathcal{F}(a)}(y)$ we have $\int_{\mathbb{R}^N} y \chi_{\{|y|\leq \delta L^{-d/2}\}} \mathcal{N}_{\Var \mathcal{F}(a)}(y) \,dy =0$. As a consequence, we get
\begin{align*}
&\int (x-\mathbb{E}[a^\RVE_{ij}])^2 \mathcal{M}^\delta(x) \,dx
\\&
=\frac{1}{p} \int_{\mathbb{R}^{N}} \big(\Varaideal + \big(\Cov[a^\RVE_{ij},\mathcal{F}(a)](\Var \mathcal{F}(a))^{-1} y\big)^2\big)
\\&~~~~~~~~~~~~~~~~~~~
\nonumber
\times \chi_{\{|y|\leq \delta L^{-d/2}\}} \mathcal{N}_{\Var \mathcal{F}(a)}(y) \,dy
\\&
\stackrel{\eqref{FNondegenerate}}{\leq} \Big(\Varaideal + \delta^2 \Cov[a^\RVE_{ij},\mathcal{F}(a)] (\Var \mathcal{F}(a))^{-1} \Cov[\mathcal{F}(a),a^\RVE_{ij}]\Big)
\\&~~~
\times \frac{1}{p} \int_{\mathbb{R}^{N}} \chi_{\{|y|\leq \delta L^{-d/2}\}} \mathcal{N}_{\Var \mathcal{F}(a)}(y) \,dy
\\&
=\Big(\Varaideal + \delta^2 \Cov[a^\RVE_{ij},\mathcal{F}(a)] (\Var \mathcal{F}(a))^{-1} \Cov[\mathcal{F}(a),a^\RVE_{ij}]\Big)
\\&
=\big(1-(1-\delta^2)|\rho|^2\big)\Var a^\RVE_{ij}.
\end{align*}
Together with \eqref{SQSErrorVariance}, this entails \eqref{VarianceReductionSQS}.
To prove \eqref{VarianceReductionSQS2}, we repeat the proof of \eqref{SQSErrorVariance} and just replace the use of \eqref{SQSNormalApproximation} in the proof of \eqref{SQSErrorVariance} by \eqref{SQSNormalApproximation2}.

Note that the lower bound \eqref{FailureVarianceReduction} on the variance given in Theorem~\ref{TheoremFailureVarianceReduction} follows also from the estimates \eqref{SQSErrorVariance} and \eqref{SystematicErrorSQS} and the lower bound $\int (x-\mathbb{E}[a^\RVE_{ij}])^2 \mathcal{M}^\delta(x) \,dx\geq (1-|\rho|^2)\Var a^\RVE_{ij}$, the latter of which is derived analogously to the upper bound $\int (x-\mathbb{E}[a^\RVE_{ij}])^2 \mathcal{M}^\delta(x) \,dx\leq (1-(1-\delta^2)|\rho|^2)\Var a^\RVE_{ij}$.

{\bf Step 3: The probability density of the reference distribution.}
For the purpose of this subsection, introduce the abbreviation for the covariance matrix
\begin{align*}
\Lambda:=\Var (a^\RVE_{ij},\mathcal{F}(a))=
\begin{pmatrix}
\Var a^\RVE_{ij} & \Cov[a^\RVE_{ij},\mathcal{F}(a)]
\\
\Cov[\mathcal{F}(a),a^\RVE_{ij}] & \Var \mathcal{F}(a)
\end{pmatrix}.
\end{align*}
The probability density $\mathcal{M}^\delta$ of the first-variable marginal of the corresponding multivariate Gaussian conditioned on $|\mathcal{F}(a)|\leq \delta L^{-d/2}$, which is the probability distribution by which we approximate the distribution of $a^\SQS_{ij}$, is given by
\begin{align}
\label{LimitDistributionDirect}
\mathcal{M}^\delta(x)
=\frac{1} {\int_{\mathbb{R}} \int_{\mathbb{R}^N} \chi_{\{|y|\leq \delta L^{-d/2}\}} \mathcal{N}_\Lambda(\tilde x,y) \,dy \,d\tilde x}
\int_{\mathbb{R}^N} &\chi_{\{|y|\leq \delta L^{-d/2}\}}
\\&\nonumber
\times
\mathcal{N}_\Lambda(x-\mathbb{E}[a^\RVE_{ij}],y)  \,dy.
\end{align}
Our goal is to show that this probability density $\mathcal{M}^\delta$ may be rewritten in the form \eqref{LimitDistribution}. To this aim, we recall some basic linear algebra: The Schur complement of the symmetric block matrix
\begin{align*}
M:=
\begin{pmatrix}
A&B\\B^T&D
\end{pmatrix}
\end{align*}
(with $A^T=A$ and $D^T=D$) is given by $T:=A-BD^{-1}B^T$ and the inverse of the matrix may be written as
\begin{align*}
\begin{pmatrix}
A&B\\B^T&D
\end{pmatrix}^{-1}
=\begin{pmatrix}
T^{-1}&-T^{-1}BD^{-1}
\\
-D^{-1}B^T T^{-1}&D^{-1}+D^{-1}B^T T^{-1}BD^{-1}
\end{pmatrix}.
\end{align*}
The determinant may be expressed as $\det M =\det T \cdot \det D$. The Schur complement allows us to rewrite the quadratic form defined by $M^{-1}$ as
\begin{align*}
M^{-1} (x,y) \cdot (x,y) = T^{-1} (x-BD^{-1} y) \cdot (x-BD^{-1} y) + D^{-1} y \cdot y.
\end{align*}
As a consequence, we get for $M:=\Lambda$ that
\begin{align*}
T&=\Var a^\RVE_{ij}-\Cov[a^\RVE_{ij},\mathcal{F}(a)] (\Var \mathcal{F}(a))^{-1} \Cov[\mathcal{F}(a),a^\RVE_{ij}]
\\&
=\Varaideal
\end{align*}
and
\begin{align}
\label{FactorGaussian}
\mathcal{N}_\Lambda(x,y&)= \frac{1}{(2\pi)^{(N+1)/2} \sqrt{\det \Lambda}}
\exp\bigg(-\frac{1}{2} \Lambda^{-1}(x,y) \cdot (x,y)\bigg)
\\&
\nonumber
=\mathcal{N}_{\Varaideal} \big(x-\Cov[a^\RVE_{ij},\mathcal{F}(a)](\Var \mathcal{F}(a))^{-1} y\big) \mathcal{N}_{\Var \mathcal{F}(a)}(y).
\end{align}
Now, \eqref{LimitDistribution} and \eqref{LimitDistributionDirect} are seen to be equivalent.

{\bf Step 4: Proof of the normal approximation estimate and the lower bound on the probability of the event $|\mathcal{F}(a)|\leq \delta L^{-d/2}$.}
First, let us show the lower bound \eqref{ProbabilitySQSNormal}. We have
\begin{align*}
&\int_{\mathbb{R}^N} \int_{\mathbb{R}} \chi_{\{|y|\leq \delta L^{-d/2}\}} \mathcal{N}_{\Var (a^\RVE_{ij},\mathcal{F}(a))}(x,y) \,dx \,dy
\\&
=\int_{\mathbb{R}^N} \chi_{\{|y|\leq \delta L^{-d/2}\}} \mathcal{N}_{\Var \mathcal{F}(a)}(y) \,dy
\\&
\stackrel{\eqref{FNondegenerate}}{\geq} \int_{\mathbb{R}^N} \chi_{\{|y|\leq \delta L^{-d/2}\}} \mathcal{N}_{C_0 L^{-d}}(y) \,dy
\\&
\geq \int_{\mathbb{R}^N} \chi_{\{|y|\leq \delta L^{-d/2}\}} \frac{1}{(2\pi C_0 L^{-d})^{N/2}} \exp(-\delta^2) \,dy
\\&
\geq c(N) C_0^{-N/2} \delta^{N},
\end{align*}
establishing \eqref{ProbabilitySQSNormal}.

The estimate \eqref{NStretchedExponentialBound} is a consequence of the estimate on $\Var (a^\RVE,\mathcal{F}(a))$ which follows from \eqref{aRVEStretchedExponentialBound}, \eqref{FNondegenerate}, and the exponential moment bounds for Gaussians. The bound \eqref{aRVEStretchedExponentialBound} is a consequence of Lemma~\ref{MultilevelVariableStretchedExponentialBound} (note that by Proposition~\ref{PropositionApproximabilityByMultilevel}, Lemma~\ref{MultilevelVariableStretchedExponentialBound} is indeed applicable).

Our next goal is to show \eqref{SQSNormalApproximation} and \eqref{SQSNormalApproximation2}.
Let $\tilde \phi:\mathbb{R}\rightarrow \mathbb{R}$ satisfy \eqref{BoundTildePhi}
and suppose that we would like to estimate the error
\begin{align*}
\mathbb{E}\big[\tilde \phi\big(a^\SQS_{ij}\big)\big]
-\int_{\mathbb{R}} \tilde \phi(x) \mathcal{M}^\delta(x) \,dx.
\end{align*}
As the distribution of $a^\SQS_{ij}$ is obtained from the distribution of $a^\RVE_{ij}$ by conditioning on the event $|\mathcal{F}(a)|\leq \delta L^{-d/2}$, by \eqref{LimitDistribution} and \eqref{LimitDistributionDirect} this error expression is equal to
\begin{align}
\nonumber
&\mathbb{E}\big[\tilde \phi\big(a^\SQS_{ij}\big)\big]
-\int_{\mathbb{R}} \tilde \phi(x) \mathcal{M}^\delta(x) \,dx
\\&
\nonumber
=\frac{1}{\mathbb{P}[|\mathcal{F}(a)|\leq \delta L^{-d/2}]}
\bigg(
\mathbb{E}\big[\chi_{\{|\mathcal{F}(a)|\leq \delta L^{-d/2}\}} \tilde \phi\big(a^\RVE_{ij}\big)\big]
\\&~~~~~~~~~~~~~~~~~~~~~~~~~~~~~~~
\nonumber
-\int_{\mathbb{R}\times \mathbb{R}^{N}} \tilde \phi(x) \chi_{\{|y|\leq \delta L^{-d/2}\}} \mathcal{N}_\Lambda(x-\mathbb{E}[a^\RVE_{ij}],y) \,d(x,y)
\bigg)
\\&~~~~
\label{ErrorReformulated}
+\int_{\mathbb{R}} \tilde \phi(x) \mathcal{M}^\delta(x) \,dx
\Bigg(\frac{\int_{\mathbb{R} \times \mathbb{R}^{N}} \chi_{\{|y|\leq \delta L^{-d/2}\}} \mathcal{N}_\Lambda(x,y) \,d(x,y)}{\mathbb{P}[|\mathcal{F}(a)|\leq \delta L^{-d/2}]}
-1\Bigg).
\end{align}
Up to the normalizing factor $1/\mathbb{P}\big[|\mathcal{F}(a)|\leq \delta L^{-d/2}\big]$, the first term on the right-hand side is given by
\begin{align*}
\mathbb{E}\big[\phi\big(a^\RVE-\mathbb{E}[a^\RVE_{ij}]\big)\big]
-\int_{\mathbb{R}\times \mathbb{R}^N} \phi(x,y) \mathcal{N}_\Lambda(x,y) \,d(x,y),
\end{align*}
where $\phi:\mathbb{R} \times \mathbb{R}^{N} \rightarrow \mathbb{R}$ is defined as
\begin{align}
\label{DefinitionSmallPhi}
\phi(x,y) :=
\begin{cases}
\tilde \phi(x+\mathbb{E}[a^\RVE_{ij}])&\text{for }|y|\leq \delta L^{-d/2},
\\
0&\text{for }|y|> \delta L^{-d/2}.
\end{cases}
\end{align}
We would now like to show that (a suitable multiple of) the function $\phi$ is admissible in the error bound \eqref{NormalApproximationMultilevelDegenerate}.
By the estimate
\begin{align*}
\osc_r \phi (x,y)
\leq&
\chi_{\{|y|\leq \delta L^{-d/2}+r\}} \osc_r \tilde \phi(x+\mathbb{E}[a^\RVE_{ij}])
\\&~~~
+\chi_{|y|\in [\delta L^{-d/2}-r,\delta L^{-d/2}+r]} |\tilde \phi(x+\mathbb{E}[a^\RVE_{ij}])|,
\end{align*}
we obtain for any $z_0=(x_0-\mathbb{E}[a^\RVE_{ij}],y_0)\in \mathbb{R}\times \mathbb{R}^N$, making also use of the abbreviation $Q:=\Cov[a^\RVE,\mathcal{F}(a)](\Var \mathcal{F}(a))^{-1}$,
\begin{align*}
&\int_{\mathbb{R}\times \mathbb{R}^N} \osc_r \phi (z) \mathcal{N}_{\Lambda} (z-z_0) \,dz
\\&
\leq
\int_{\mathbb{R}^{N}}
\int_{\mathbb{R}} (\osc_r \tilde \phi) (x) \chi_{\{|y|\leq \delta L^{-d/2}+r\}} \mathcal{N}_{\Lambda} (x-x_0,y-y_0) \,dx \,dy
\\&~~~~
+\int_{\mathbb{R}^{N}}
\int_{\mathbb{R}} |\tilde \phi| (x) \chi_{|y|\in [\delta L^{-d/2}-r,\delta L^{-d/2}+r]} \mathcal{N}_{\Lambda} (x-x_0,y-y_0) \,dx \,dy
\\&
\stackrel{\eqref{FactorGaussian}}{\leq}
\int_{\mathbb{R}^{N}} \chi_{\{|y|\leq \delta L^{-d/2}+r\}}
\mathcal{N}_{\Var \mathcal{F}(a)}(y-y_0)
\\&~~~~~~~~~~~~~~~~\times
\int_{\mathbb{R}} (\osc_r \tilde \phi) (x)  \mathcal{N}_{\Varaideal} (x-x_0-Q(y-y_0)) \,dx \,dy
\\&~~~~
+\int_{\mathbb{R}^{N}} \chi_{|y|\in [\delta L^{-d/2}-r,\delta L^{-d/2}+r]} \mathcal{N}_{\Var \mathcal{F}(a)}(y-y_0)
\\&~~~~~~~~~~~~~~~~\times
\int_{\mathbb{R}} |\tilde \phi| (x) \mathcal{N}_{\Varaideal} (x-x_0-Q(y-y_0)) \,dx \,dy
\\&
\stackrel{\eqref{BoundTildePhi}}{\leq}
\int_{\mathbb{R}^{N}} \chi_{\{|y|\leq \delta L^{-d/2}+r\}}
\mathcal{N}_{\Var \mathcal{F}(a)}(y-y_0) \cdot r \,dy
\\&~~~~
+\int_{\mathbb{R}^{N}} \chi_{|y|\in [\delta L^{-d/2}-r,\delta L^{-d/2}+r]} \mathcal{N}_{\Var \mathcal{F}(a)}(y-y_0) \cdot L^{-d/2} \,dy
\end{align*}
and therefore
\begin{align*}
&\int_{\mathbb{R}\times \mathbb{R}^N} \osc_r \phi (z) \mathcal{N}_{\Lambda} (z-z_0) \,dz
\\&
\leq
r+L^{-d/2} \int_{\mathbb{R}^{N}} \chi_{|y+y_0|\in [\delta L^{-d/2}-r,\delta L^{-d/2}+r]} \mathcal{N}_{\Var \mathcal{F}(a)}(y) \,dy
\\&
\leq
r+L^{-d/2}\sup_{|W| \leq |B_1^N| \big((\delta L^{-d/2}+r)^{N}-(\delta L^{-d/2}-r)_+^{N}\big)}\int_{W} \mathcal{N}_{\Var \mathcal{F}(a)}(y) \,dy
\\&
\leq r + L^{-d/2} \min\bigg\{1,|B_1^N| \cdot \frac{(\delta L^{-d/2}+r)^{N}-(\delta L^{-d/2}-r)_+^{N}}{(2\pi)^{N/2} \sqrt{\det \Var \mathcal{F}(a)}}\bigg\}
\\&
\leq r + L^{-d/2} \min\bigg\{1,|B_1^N| \frac{N(\delta L^{-d/2}+r)^{N-1} \cdot 2r}{(2\pi)^{N/2} \sqrt{\det \Var \mathcal{F}(a)}}\bigg\}
\\&
\leq r + L^{-d/2}\min\bigg\{1,\frac{C(N)r^N + C(N)  (\delta L^{-d/2})^{N-1} r}{\sqrt{\det \Var \mathcal{F}(a)}}\bigg\}
\\&
\stackrel{\eqref{FNondegenerate}}{\leq} r + \frac{C(N)L^{-d/2}}{\sqrt{\det \Var \mathcal{F}(a)}^{1/N}} r
+ \frac{C(N) \delta^{N-1} L^{-d/2}}{\sqrt{\det \Var \mathcal{F}(a)}^{1/N}} r.
\end{align*}
By our assumption \eqref{FNondegenerate}, this yields for any $z_0\in \mathbb{R}\times \mathbb{R}^N$
\begin{align*}
&\int_{\mathbb{R}\times \mathbb{R}^N} \osc_r \phi (z) \mathcal{N}_{\Lambda} (z-z_0) \,dz
\leq C r.
\end{align*}
Looking at Definition~\ref{DefinitionDistance}, we would have $\frac{1}{C}\phi\in \Phi_\Lambda$ if it were not for the \emph{qualitative} Lipschitz continuity condition for functions in $\Phi_\Lambda$. However, for a standard family of mollifiers $\rho_\varepsilon$ supported in $\{|x|^2+|y|^2\leq \varepsilon\}$ the approximations $\phi_\varepsilon(x,y):= (\rho_\varepsilon \ast \phi) (x,(1-2\delta^{-1}L^{d/2}\varepsilon)y)$ satisfy $\frac{1}{C}\phi_\varepsilon\in \Phi_\Lambda$ for any $\varepsilon\in (0,\frac{1}{4}\delta L^{-d/2}]$ (see Definition~\ref{DefinitionDistance}) for some constant $C$.
Furthermore, the $\phi_\varepsilon$ converge poinwise to $\phi$ for $\varepsilon\rightarrow 0$ (by \eqref{DefinitionSmallPhi} and the continuity assumption on $\tilde \phi$; it is here that we need the dilation factor $(1-2\delta^{-1}L^{d/2}\varepsilon)$ in the second variable due to the discontinuity in the definition \eqref{DefinitionSmallPhi}) and satisfy a uniform bound of the form $|\phi_\varepsilon(x,y)|\leq L^{-d/2}$ (by \eqref{DefinitionSmallPhi} and \eqref{BoundTildePhia}).
Choosing the functions $\frac{1}{C} \phi_\varepsilon$ in the definition of the distance $\mathcal{D}$ and passing to the limit $\varepsilon\rightarrow 0$, we infer
\begin{align*}
&\bigg|\mathbb{E}\big[\chi_{\{|\mathcal{F}(a)|\leq \delta L^{-d/2}\}} \tilde \phi\big(a^\RVE_{ij}\big)\big]
-\int_{\mathbb{R}\times \mathbb{R}^{N}} \tilde \phi(x) \chi_{\{|y|\leq \delta L^{-d/2}\}} \mathcal{N}_\Lambda(x-\mathbb{E}[a^\RVE_{ij}],y) \,d(x,y)
\bigg|
\\&
\nonumber
\stackrel{\eqref{DefinitionSmallPhi},\eqref{DefinitionD}}{\leq} C \mathcal{D}((a^\RVE_{ij}-\mathbb{E}[a^\RVE_{ij}],\mathcal{F}(a)),\mathcal{N}_\Lambda).
\end{align*}
Theorem~\ref{TheoremNormalApproximationMultilevelLocalDependence} is applicable to the random variable $X:=(a^\RVE_{ij},\mathcal{F}(a))$ by our assumptions on $\mathcal{F}(a)$ (see Assumption~\ref{AssumptionsNotation}) and by the multilevel decomposition of $a^\RVE_{ij}$ provided by Proposition~\ref{PropositionApproximabilityByMultilevel}. In total, with the notation $\Lambda:=\Var (a^\RVE_{ij},\mathcal{F}(a))$ the application of Theorem~\ref{TheoremNormalApproximationMultilevelLocalDependence} to $(a^\RVE_{ij},\mathcal{F}(a))$ yields
\begin{align}
\label{ErrorInConditionedExpectation}
&\bigg|\mathbb{E}\big[\chi_{\{|\mathcal{F}(a)|\leq \delta L^{-d/2}\}} \tilde \phi\big(a^\RVE_{ij}\big)\big]
-\int_{\mathbb{R}\times \mathbb{R}^{N}} \tilde \phi(x) \chi_{\{|y|\leq \delta L^{-d/2}\}} \mathcal{N}_\Lambda(x-\mathbb{E}[a^\RVE_{ij}],y) \,d(x,y)
\bigg|
\\&
\nonumber
\leq C(d,\gamma,N,K) B^3 |\log L|^{C(d,\gamma)} (L^{-d} |\Lambda|^{1/2} |\Lambda^{-1/2}|^3) L^{-d}
\\&
\nonumber
\leq C(d,\lambda,\gamma,N,C_0) \kappa_{ij}^{3/2} L^{-d} |\log L|^{C(d,\gamma)},
\end{align}
where in the last step we have used \eqref{FNondegenerate} (which entails $L^{-d} \leq |\Lambda^{1/2}|^2$) and the definition of $\kappa_{ij}$.

Applying a similar line of argument to the random variable $\mathcal{F}(a)$ and the function
\begin{align*}
\phi(y) :=
\begin{cases}
1&\text{for }|y|\leq \delta L^{-d/2},
\\
0&\text{for }|y|> \delta L^{-d/2},
\end{cases}
\end{align*}
we obtain
\begin{align}
\label{ErrorInProbability}
&\bigg|\mathbb{E}\big[\chi_{\{|\mathcal{F}(a)|\leq \delta L^{-d/2}\}}\big]
-\int_{\mathbb{R}^{N}} \chi_{\{|y|\leq \delta L^{-d/2}\}} \mathcal{N}_{\Var \mathcal{F}(a)}(y) \,dy
\bigg|
\\&
\nonumber
\leq C(d,\gamma,N,C_0) C_0^{3N/2} L^{-d/2} |\log L|^{C(d,\gamma,C_0)}
\end{align}
where we have estimated $\kappa(\Var \mathcal{F}(a))$ by \eqref{FNondegenerate}.
Together with the lower bound \eqref{ProbabilitySQSNormal} and our assumption $\delta^N \geq CL^{-d/2} |\log L|^{C(d,\gamma,C_0)}$, this estimate implies \eqref{ProbabilitySQS}.

Plugging in the estimate \eqref{ErrorInConditionedExpectation}, the lower bound \eqref{ProbabilitySQS}, and the estimate \eqref{ErrorInProbability} as well as the assumption \eqref{BoundTildePhia} into \eqref{ErrorReformulated}, we deduce \eqref{SQSNormalApproximation}.
The estimate \eqref{SQSNormalApproximation2} follows by repeating the above steps, but appealing in the proof of \eqref{ErrorInConditionedExpectation} to the bound \eqref{NormalApproximationMultilevelDegenerate} instead of \eqref{NormalApproximationMultilevel} and choosing $\Lambda:=\Var (a^\RVE_{ij},\mathcal{F}(a))+L^{-d/2-d/8}\Id$ (which ensures by \eqref{FNondegenerate} that $\kappa(\Lambda)\leq CL^{d/8}$).
\end{proof}
We now turn to the proof of the moderate-deviations-type result for the selection approach for representative volumes stated in Theorem~\ref{TheoremSQSModerateDeviations}.
\begin{proof}[Proof of Theorem~\ref{TheoremSQSModerateDeviations}]
Fix $\tilde S\geq CL^{-d/2-\beta/2}$. 
Our goal is to estimate the probability
\begin{align}
\nonumber
&\mathbb{P}\big[|a^\SQS_{ij}-\mathbb{E}[a^\RVE_{ij}]|\geq \tilde S\big]
\\&
\nonumber
=\mathbb{P}\Big[|a^\RVE_{ij}-\mathbb{E}[a^\RVE_{ij}]|\geq \tilde S ~\Big|~
|\mathcal{F}(a)|\leq \delta L^{-d/2}\Big]
\\&
\label{ModerateDeviationsConditionalProbabilityRewritten}
=\frac{\mathbb{P}\big[|a^\RVE_{ij}-\mathbb{E}[a^\RVE_{ij}]|\geq \tilde S \text{ and }|\mathcal{F}(a)|\leq \delta L^{-d/2}\big]}{\mathbb{P}[|\mathcal{F}(a)|\leq \delta L^{-d/2}]}.
\end{align}
The main task is the derivation of a suitable estimate for the numerator.
To this aim, we apply the moderate deviations estimate from Theorem~\ref{TheoremModerateDeviations} to the random variable $(a^\RVE_{ij}-\mathbb{E}[a^\RVE_{ij}],\mathcal{F}(a))$ and the set $A:=A_1\times A_2$ with
\begin{align*}
A_1&:=\big\{x\in \mathbb{R}:|x|\geq \tilde S+CL^{-d/2-\beta} \big\},
\\
A_2&:=\big\{y\in \mathbb{R}^N:|y|\leq \delta L^{-d/2}\big\}.
\end{align*}
By Proposition~\ref{PropositionApproximabilityByMultilevel} and our assumptions, the application of Theorem~\ref{TheoremModerateDeviations} is possible, resulting in the estimate
\begin{align}
\nonumber
&\mathbb{P}\big[(a^\RVE_{ij}-\mathbb{E}[a^\RVE_{ij}],\mathcal{F}(a))\in A\big]
\\&
\nonumber
\leq \int_{\{(x,y)\in \mathbb{R}\times \mathbb{R}^{N}:\dist((x,y),A)\leq C L^{-\beta} L^{-d/2}\}} \mathcal{N}_{\tilde \Lambda}(x,y) \,d(x,y)
+ C \exp(-cL^{2\beta} |\log L|^{-C})
\\&
\label{ConsequenceGeneralModerateDeviationsTheorem}
\leq \int_{\mathbb{R}^{N}} \chi_{\{|y|\leq (\delta+CL^{-\beta}) L^{-d/2}\}}  \int_{\mathbb{R}\setminus [-\tilde S,\tilde S]} \mathcal{N}_{\tilde \Lambda}(x,y) \,dx \,dy+ C \exp(-L^{\beta})
\end{align}
for some positive definite matrix $\tilde \Lambda$ with
\begin{align}
\label{TildeLambdaError}
|\tilde \Lambda-\Var (a^\RVE_{ij},\mathcal{F}(a))|\leq C(d,\gamma,N,K) |\log L|^C L^{-2\beta} L^{-d}.
\end{align}
We intend to apply the factorization property \eqref{FactorGaussian} to the matrix $\tilde \Lambda$ with the notation
\begin{align*}
\tilde \Lambda =
\begin{pmatrix}
\tilde A&\tilde B
\\
\tilde B^T&\tilde D
\end{pmatrix}.
\end{align*}
By \eqref{TildeLambdaError} and the bounds $L^{-d}\Id\leq \Var \mathcal{F}(a) \leq CL^{-d}\Id$ (see \eqref{FNondegenerate}) and $\Var a^\RVE_{ij} \leq CL^{-d} |\log L|^d$ (see \eqref{aRVEStretchedExponentialBound}), we deduce
\begin{align}
\label{ErrorTildeD}
|\tilde D^{-1}-(\Var \mathcal{F}(a))^{-1}|\leq C L^d L^{-2\beta}
\end{align}
and
\begin{align}
\label{ErrorTildeBT}
|\tilde B \tilde D^{-1}-\Cov[a_{ij}^\RVE,\mathcal{F}(a)](\Var \mathcal{F}(a))^{-1}|\leq C L^{-2\beta} |\log L|^d.
\end{align}
As a consequence of these estimates and \eqref{TildeLambdaError}, the formula \eqref{FormulaRho} for $|\rho|^2$ implies for $\tilde T:=\tilde A-\tilde B \tilde D^{-1} \tilde B^T$
\begin{align}
\label{ErrorTildeT}
\big|\tilde T-(1-|\rho|^2)\Var a^\RVE_{ij}\big| \leq C L^{-d-2\beta} |\log L|^d.
\end{align}
Using the bounds $\Var a^\RVE_{ij}\leq C L^{-d} |\log L|^d$ and $|\rho|\leq 1$ as well as \eqref{ErrorTildeBT}, \eqref{FormulaRho}, and \eqref{FNondegenerate}, we obtain for any $|y|\leq (\delta+CL^{-\beta})L^{-d/2}$ that
\begin{align}
\label{EstimateBDy}
|\tilde B \tilde D^{-1} y| \leq C\delta |\rho| \sqrt{\Var a^\RVE_{ij}}+ C L^{-d/2-\beta}.
\end{align}
Applying the factorization property \eqref{FactorGaussian} to the first term on the right-hand side of \eqref{ConsequenceGeneralModerateDeviationsTheorem}, we obtain
\begin{align*}
&\int_{\mathbb{R}^{N}} \chi_{\{|y|\leq (\delta+CL^{-\beta}) L^{-d/2}\}}  \int_{\mathbb{R}\setminus [-\tilde S,\tilde S]} \mathcal{N}_{\tilde \Lambda}(x,y) \,dx \,dy
\\&
=\int_{\mathbb{R}^{N}} \int_{\mathbb{R}\setminus [-\tilde S,\tilde S]} 
\mathcal{N}_{\tilde T}\big(x-\tilde B \tilde D^{-1} y\big)
\chi_{\{|y|\leq (\delta+CL^{-\beta}) L^{-d/2}\}} \mathcal{N}_{\tilde D}(y)  \,dx \,dy
\\&
\leq \int_{\mathbb{R}^{N}}
\int_{\mathbb{R}\setminus [-\tilde S,\tilde S]} 
\mathcal{N}_{\tilde T}\big(x\big)
\cdot \exp\big(\tilde T^{-1} x \cdot \tilde B \tilde D^{-1} y\big)
\\&~~~~~~~~~~~~~~~~~~~~~~~~~~~~~~~~~~~~
\times \chi_{\{|y|\leq (\delta+CL^{-\beta}) L^{-d/2}\}} \mathcal{N}_{\tilde D}(y)  \,dx \,dy
\\&
\stackrel{\eqref{EstimateBDy}}{\leq} \int_{\mathbb{R}^{N}}
\int_{\mathbb{R}\setminus [-\tilde S,\tilde S]} 
\frac{1}{\sqrt{2\pi \tilde T}} \exp\bigg(\frac{-x^2 + C\delta |\rho|\sqrt{\Var a^\RVE_{ij}} |x|+C|x|L^{-d/2-\beta}}{2\tilde T}\bigg)
\\&~~~~~~~~~~~~~~~~~~~~~~~~~~~~~~~~~~~~
\times \chi_{\{|y|\leq (\delta+CL^{-\beta}) L^{-d/2}\}} \mathcal{N}_{\tilde D}(y)  \,dx \,dy.
\end{align*}
Assuming that $\tilde S\geq C L^{-d/2-\beta/2}$, we deduce
\begin{align*}
&\int_{\mathbb{R}^{N}} \chi_{\{|y|\leq (\delta+CL^{-\beta}) L^{-d/2}\}}  \int_{\mathbb{R}\setminus [-\tilde S,\tilde S]} \mathcal{N}_{\tilde \Lambda}(x,y) \,dx \,dy
\\&
\leq
\int_{\mathbb{R}^{N}}
\int_{\mathbb{R}\setminus [-\tilde S,\tilde S]} 
\frac{1}{\sqrt{2\pi \tilde T}} \exp\bigg(\frac{-\big(1-\frac{C\delta |\rho|\sqrt{\Var a^\RVE_{ij}}}{\tilde S}-L^{-\beta/2}\big)x^2}{2\tilde T}\bigg)
\\&~~~~~~~~~~~~~~~~~~~~~~~~~~~~~~~~~~~~
\times \chi_{\{|y|\leq (\delta+CL^{-\beta}) L^{-d/2}\}} \mathcal{N}_{\tilde D}(y)  \,dx \,dy
\\&
\leq 
\int_{\mathbb{R}\setminus [-\tilde S,\tilde S]} 
\frac{1}{1-\frac{C\delta |\rho| \sqrt{\Var a_{ij}^\RVE}}{\tilde S}-L^{-\beta/2}}
\mathcal{N}_{V}(x)  \,dx
\\&~~~~~~~~
\times 
\int_{\mathbb{R}^{N}} \chi_{\{|y|\leq (\delta+CL^{-\beta}) L^{-d/2}\}} \mathcal{N}_{\tilde D}(y) \,dy
\end{align*}
with
\begin{align}
\label{DefinitionV}
V:=\frac{\tilde T}{1-C\delta |\rho|\sqrt{\Var a^\RVE_{ij}} \tilde S^{-1}-L^{-\beta/2}}.
\end{align}
Using \eqref{ErrorTildeD} to estimate the last factor in this estimate and assuming for the moment $\tilde S \geq C\delta |\rho| \sqrt{\Var a^\RVE_{ij}}$ as well as $L\geq C(\beta)$ to estimate the quotient in the first factor, we get
\begin{align*}
&\int_{\mathbb{R}^{N}} \chi_{\{|y|\leq (\delta+CL^{-\beta}) L^{-d/2}\}}  \int_{\mathbb{R}\setminus [-\tilde S,\tilde S]} \mathcal{N}_{\tilde \Lambda}(x,y) \,dx \,dy
\\&
\leq 
\bigg(1+\frac{C\delta |\rho|\sqrt{\Var a^\RVE_{ij}}}{\tilde S}+L^{-\beta/2}\bigg)
\int_{\mathbb{R}\setminus [-\tilde S,\tilde S]} 
\mathcal{N}_{V}(x)  \,dx
\\&~~~~~~~~~~~
\times
\int_{\mathbb{R}^{N}} \chi_{\{|y|\leq (\delta+CL^{-\beta}) L^{-d/2}\}} \mathcal{N}_{\Var \mathcal{F}(a)-CL^{-d-2\beta}\Id}(y) \,dy.
\end{align*}
Using the bound $L^{-d}\Id \leq \Var \mathcal{F}(a)$ from \eqref{FNondegenerate} and assuming $L^{-2\beta}\leq c$, we get
\begin{align*}
&\int_{\mathbb{R}^{N}} \chi_{\{|y|\leq (\delta+CL^{-\beta}) L^{-d/2}\}}  \int_{\mathbb{R}\setminus [-\tilde S,\tilde S]} \mathcal{N}_{\tilde \Lambda}(x,y) \,dx \,dy
\\&
\leq 
\bigg(1+\frac{C\delta |\rho|\sqrt{\Var a^\RVE_{ij}}}{\tilde S}+L^{-\beta/2}\bigg)
\int_{\mathbb{R}\setminus [-\tilde S,\tilde S]} 
\mathcal{N}_{V}(x)  \,dx
\\&~~~~~~~~~~~
\times
\int_{\mathbb{R}^{N}} \chi_{\{|y|\leq (\delta+2CL^{-\beta}) L^{-d/2}\}} \mathcal{N}_{\Var \mathcal{F}(a)}(y) \,dy
\end{align*}
and therefore by the upper bound $|\mathcal{N}_{\Var \mathcal{F}(a)}|\leq C (L^{-d/2})^{-d}$ and the estimate on the volume $|\{\delta L^{-d/2}\leq |y|\leq (\delta+2CL^{-\beta}) L^{-d/2}\}|\leq C (L^{-d/2})^{d-1} L^{-d/2-\beta}$
\begin{align*}
&\int_{\mathbb{R}^{N}} \chi_{\{|y|\leq (\delta+CL^{-\beta}) L^{-d/2}\}}  \int_{\mathbb{R}\setminus [-\tilde S,\tilde S]} \mathcal{N}_{\tilde \Lambda}(x,y) \,dx \,dy
\\&
\leq 
\bigg(1+\frac{C\delta |\rho|\sqrt{\Var a^\RVE_{ij}}}{\tilde S}+L^{-\beta/2}\bigg)
\int_{\mathbb{R}\setminus [-\tilde S,\tilde S]} 
\mathcal{N}_{V}(x)  \,dx
\\&~~~~~~~~~~~
\times
\bigg(\int_{\mathbb{R}^{N}} \chi_{\{|y|\leq \delta L^{-d/2}\}} \mathcal{N}_{\Var \mathcal{F}(a)}(y) \,dy+C L^{-\beta}\bigg).
\end{align*}

By $\tilde T\leq (1-|\rho|^2)\Var a_{ij}^\RVE+CL^{-d-\beta}\Id$ (which follows from \eqref{ErrorTildeT}) and $\Var a_{ij}^\RVE \leq C L^{-d} \Id$, we deduce from \eqref{DefinitionV} under the assumptions $\tilde S \geq C\delta |\rho| \sqrt{\Var a^\RVE_{ij}}$ and $L\geq C(\beta)$
\begin{align}
\label{DefinitionTildeV}
V\leq \tilde V:=\bigg(1+\frac{C\delta |\rho|\sqrt{\Var a^\RVE_{ij}}}{\tilde S} \bigg)(1-|\rho|^2) \Var a_{ij}^\RVE + C L^{-d-\beta/2}.
\end{align}
As a consequence, we obtain
\begin{align*}
&\int_{\mathbb{R}^{N}} \chi_{\{|y|\leq (\delta+CL^{-\beta}) L^{-d/2}\}}  \int_{\mathbb{R}\setminus [-\tilde S,\tilde S]} \mathcal{N}_{\tilde \Lambda}(x,y) \,dx \,dy
\\&
\leq
\bigg(1+\frac{C\delta |\rho|\sqrt{\Var a^\RVE_{ij}}}{\tilde S}+L^{-\beta/2}\bigg)
\int_{\mathbb{R}\setminus [-\tilde S,\tilde S]} 
\mathcal{N}_{\tilde V}(x)  \,dx
\\&~~~~~~~~~
\times
\bigg(\int_{\mathbb{R}^{N}} \chi_{\{|y|\leq \delta L^{-d/2}\}} \mathcal{N}_{\Var \mathcal{F}(a)}(y) \,dy+C L^{-\beta}\bigg).
\end{align*}
Plugging in this bound into \eqref{ConsequenceGeneralModerateDeviationsTheorem}, we obtain
\begin{align*}
&\mathbb{P}\big[|a^\RVE_{ij}-\mathbb{E}[a^\RVE_{ij}]|\geq \tilde S+CL^{-d/2-\beta} \text{ and }|\mathcal{F}(a)|\leq \delta L^{-d/2}\big]
\\&
\leq
\bigg(1+\frac{C\delta |\rho|\sqrt{\Var a^\RVE_{ij}}}{\tilde S}+L^{-\beta/2}\bigg)
\int_{\mathbb{R}\setminus [-\tilde S,\tilde S]} 
\mathcal{N}_{\tilde V}(x)  \,dx
\\&~~~~~~~~~
\times
\bigg(\int_{\mathbb{R}^{N}} \chi_{\{|y|\leq \delta L^{-d/2}\}} \mathcal{N}_{\Var \mathcal{F}(a)}(y) \,dy+C L^{-\beta}\bigg)
\\&~~~
+C\exp(-L^{\beta}).
\end{align*}
Inserting the previous estimate into \eqref{ModerateDeviationsConditionalProbabilityRewritten} and using  \eqref{ErrorInProbability}, \eqref{ProbabilitySQSNormal}, and \eqref{ProbabilitySQS} as well as the assumption $\delta^N \geq C L^{-d/2}$ to estimate the denominator, we get
\begin{align*}
&\mathbb{P}\big[|a^\SQS_{ij}-\mathbb{E}[a^\RVE_{ij}]|\geq \tilde S+CL^{-d/2-\beta}\big]
\\&
\leq
\bigg(1+\frac{C\delta \sqrt{\Var a^\RVE_{ij}}}{\tilde S}+\frac{C}{\delta^N} L^{-\beta/2}\bigg) \int_{\mathbb{R}\setminus [-\tilde S,\tilde S]}
\mathcal{N}_{\tilde V}(x)  \,dx
+\frac{C}{\delta^N}\exp(-L^{\beta}).
\end{align*}
Note that we have the estimate $|\mathbb{E}[a^\RVE_{ij}]-a_{\shom,ij}|\leq C L^{-d} |\log L|^C$. By redefining $\tilde S$ (and possibly increasing the constant in \eqref{DefinitionTildeV}; recall that $\tilde S\geq L^{-d/2-\beta/2}$), we obtain
\begin{align*}
&\mathbb{P}\big[|a^\SQS_{ij}-a_{\shom,ij}|\geq \tilde S\big]
\\&
\leq
\bigg(1+\frac{C\delta \sqrt{\Var a^\RVE_{ij}}}{\tilde S}+\frac{C}{\delta^N} L^{-\beta/2}\bigg) \int_{\mathbb{R}\setminus [-\tilde S,\tilde S]}
\mathcal{N}_{\tilde V}(x)  \,dx
+\frac{C}{\delta^N}\exp(-L^{\beta}).
\end{align*}
Finally, we set $\tilde S := \sqrt{(1+\frac{C\delta}{\sqrt{1-|\rho|^2}s})(1-|\rho|^2)\Var a_{ij}^\RVE+L^{-d-\beta/2}} \cdot s$. Upon redefining $\beta$, this yields
the desired estimate \eqref{ModerateDeviationsSQS}.
\end{proof}

\section{The multilevel local dependence structure of the approximation for the effective conductivity}

We now prove that the approximation $a^\RVE$ for the effective conductivity obtained by the representative volume element method may indeed be written as a sum of a family of random variables with multilevel local dependence structure in the sense of Definition~\ref{ConditionRandomVariable}. Furthermore, we show that the same is true for the spatial average of the coefficient field $\mathcal{F}_{avg}(a):=\dashint_{[0,L\varepsilon]^d} a \,dx$ and also for the second-order correction $\mathcal{F}_{2-point}(a)$ to $a^\RVE$ in the setting of small ellipticity contrast.

\begin{proof}[Proof of Proposition~\ref{PropositionApproximabilityByMultilevel}]
{\bf Part 1: The spatial average of the coefficient.}
First, let us show that the average $\mathcal{F}_{avg}(a):=\dashint_{[0,L\varepsilon]^d} a\,dx$ is approximately the sum of a family of random variables with multilevel local dependence structure. Decomposing
\begin{align*}
\mathcal{F}_{avg}(a)=\dashint_{[0,L \varepsilon]^d} a \,dx=\sum_{y\in \varepsilon \mathbb{Z}^d \cap [0,L \varepsilon)^d} \underbrace{\frac{1}{L^d}\dashint_{y+[0,\varepsilon]^d} a \,dx}_{=:X_y^0},
\end{align*}
defining the $X_y^0$ as indicated in this formula, and setting $X_y^m:=0$ for $m\geq 1$, we immediately observe that the average $\mathcal{F}_{avg}(a)$ is the sum of a family of random variables with multilevel local dependence structure with $K:=1$. The bound \eqref{BoundMultilevelDependenceStructure} follows immediately from the uniform bound on $a$ (with $B:=||a||_{L^\infty}$ and arbitrary $\gamma>0$).

{\bf Part 2: The approximation $a^\RVE$ for the effective coefficient.}
Next, let us show that $a^\RVE$ is approximately the sum of a family of random variables with multilevel local dependence structure. For simplicity of notation, let us assume that $\varepsilon=1$.

Recall that the corrector $\phi_i$ associated with the periodized ensemble is the unique $L$-periodic solution to the equation
\begin{align}
\label{CorrectorEquationRepeat}
\nabla \cdot (a(e_i+\nabla \phi_i))=0
\end{align}
with vanishing average $\dashint_{[0,L]^d} \phi_i \,dx=0$. We shall use the decomposition of the ($L$-periodic) corrector $\phi_i$ according to
\begin{align}
\label{Decomposition}
\phi_i(\cdot)=\int_0^\infty u_i(\cdot,s)\,ds
\end{align}
where $u_i=u_i(x,s)$ is the ($L$-periodic) solution to the parabolic PDE
\begin{align*}
\frac{d}{ds} u_i &= \nabla \cdot (a\nabla u_i) &&\text{in }\domain\times [0,\infty),
\\
u_i(\cdot,0) &= \nabla \cdot (ae_i)&&\text{in }\domain.
\end{align*}
Observe that the parabolic PDE directly entails
\begin{align}
\label{PartialDecomposition}
\nabla \cdot \bigg(a\bigg(e_i+\nabla \int_0^t u_i(\cdot,s)\,ds\bigg)\bigg) = u_i(\cdot,t).
\end{align}
Thus, decay of $u_i$ for $t\rightarrow \infty$ implies that $\phi_i$ may indeed be decomposed as $\int_0^\infty u_i(\cdot,s) \,ds$. Note that exponential decay of $u_i$ (with an $L$-dependent constant) is immediate by the standard energy estimate, the vanishing average of $u_i(\cdot,s)$ for any $s\geq 0$ (as the average of the initial conditions on $[0,L]^d$ vanishes), and the Poincar\'e inequality.

Recall the key result from \cite{GloriaOttoNew} which states that under the assumptions of ellipticity, stationarity, and finite range of dependence (A1)-(A3) the full-space variant $u_i^{\mathbb{R}^d}(\cdot,s)$ -- that is, the solution to the equation
\begin{align*}
\frac{d}{ds} u_i^{\mathbb{R}^d} &= \nabla \cdot (a^{\mathbb{R}^d}\nabla u_i^{\mathbb{R}^d}) &&\text{in }\mathbb{R}^d\times [0,\infty),
\\
u_i^{\mathbb{R}^d}(\cdot,0) &= \nabla \cdot (a^{\mathbb{R}^d}e_i)&&\text{in }\mathbb{R}^d,
\end{align*}
with $a^{\mathbb{R}^d}$ denoting a coefficient field from the original (non-periodic) ensemble of coefficient fields -- actually decays like $s^{-(1+d/2)/2}$ in suitable norms:
\begin{theorem}[\cite{GloriaOttoNew}, Corollary~4]
\label{TheoremSemigroupDecay}
Consider an ensemble of random coefficient fields $a^{\mathbb{R}^d}$ subject to the assumptions (A1)-(A3) with range of dependence $\varepsilon:=1$.
Then for any $T>0$ we have the estimate
\begin{subequations}
\label{SemigroupDecay}
\begin{align}
\label{SemigroupDecay1}
\left(\dashint_{\{|x|\leq \sqrt{T}\}} |\nabla u_i^{\mathbb{R}^d}(\cdot,T)|^2 \,dx\right)^{1/2}
&\leq \mathcal{C}(a^{\mathbb{R}^d},T) \, T^{-1-d/4},
\\
\label{SemigroupDecay2}
\left(\dashint_{\{|x|\leq \sqrt{T}\}} |u_i^{\mathbb{R}^d}(\cdot,T)|^2 \,dx\right)^{1/2}
&\leq \mathcal{C}(a^{\mathbb{R}^d},T) \, T^{-1/2-d/4},
\end{align}
\end{subequations}
where the random constant $\mathcal{C}(a^{\mathbb{R}^d},T)$ satisfies for any $\delta>0$ a bound of the form
\begin{align*}
\mathbb{E}\bigg[\exp\bigg(\frac{\mathcal{C}(a^{\mathbb{R}^d},T)^{2-\delta}}{C(d,\lambda,\delta)}\bigg)\bigg] \leq 2.
\end{align*}
\end{theorem}
Note that the second inequality \eqref{SemigroupDecay2} is actually not contained in \cite[Corollary~4]{GloriaOttoNew}. However, it is an easy consequence of \eqref{SemigroupDecay1} (the proof is provided below).

By $\phi_j^*$ and $u_j^*$ we shall denote the corresponding quantities for the adjoint coefficient field $a^*$, i.\,e.\ $\phi_j^*(\cdot):=\int_0^\infty u_j^*(\cdot,s)\,ds$ with $u_j^*$ being the $L$-periodic solution to
\begin{align*}
\frac{d}{ds} u_j^* &= \nabla \cdot (a^* \nabla u_j^*) &&\text{in }\domain\times [0,\infty),
\\
u_j^*(\cdot,0) &= \nabla \cdot (a^* e_j)&&\text{in }\domain.
\end{align*}
The full space variants $u_j^{*,\mathbb{R}^d}$ satisfy also estimates of the form \eqref{SemigroupDecay1}-\eqref{SemigroupDecay2}, as the conditions (A1)-(A3) are invariant under passing to the adjoint coefficient fields.

We introduce a ``cutoff scale'' $L_K$ as the largest integer power of $2$ not larger than $\frac{L}{16 K \log L}$ for some constant $K\geq 1$ that remains to be chosen. Defining $T_L:=(L_K)^2$, we now compute using the properties \eqref{CorrectorEquationRepeat}, \eqref{Decomposition}, and \eqref{PartialDecomposition}
\begin{align*}
&a^\RVE e_i \cdot e_j
=\dashint_{\domain} a(e_i+\nabla \phi_i) \cdot e_j \,dx
\\&
\stackrel{\eqref{CorrectorEquationRepeat}}{=}\dashint_{\domain} a(e_i+\nabla \phi_i) \cdot (e_j+\nabla \phi_j^*) \,dx
\\&
\stackrel{\eqref{Decomposition}}{=}\dashint_{\domain} a\bigg(e_i+\nabla \int_0^{1} u_i(\cdot,s)\,ds\bigg) \cdot \bigg(e_j+\nabla \int_0^{1} u_j^*(\cdot,s)\,ds\bigg) \,dx
\\&~~~
+\sum_{k=0}^{\log_2 L_K} \dashint_{\domain} a \nabla \int_{4^k}^{4^{k+1}} u_i(\cdot,s)\,ds \cdot \bigg(e_j+\nabla \int_0^{4^k} u_j^*(\cdot,s)\,ds\bigg)  \,dx
\\&~~~
+\sum_{k=0}^{\log_2 L_K} \dashint_{\domain} a \bigg(e_i+\nabla \int_0^{4^k} u_i(\cdot,s)\,ds\bigg) \cdot \nabla \int_{4^k}^{4^{k+1}} u_j^*(\cdot,s)\,ds  \,dx
\\&~~~
+\sum_{k=0}^{\log_2 L_K}  \dashint_{\domain} a\nabla \int_{4^k}^{4^{k+1}} u_i(\cdot,s)\,ds \cdot \nabla \int_{4^k}^{4^{k+1}} u_j^*(\cdot,s)\,ds \,dx
\\&~~~
+\dashint_{\domain} a \nabla \int_{4T_L}^\infty u_i(\cdot,s)\,ds \cdot \bigg(e_j+\nabla \int_0^{4T_L} u_j^*(\cdot,s)\,ds\bigg)  \,dx
\\&~~~
+\dashint_{\domain} a \bigg(e_i+\nabla \int_0^{4T_L} u_i(\cdot,s)\,ds\bigg) \cdot \nabla \int_{4T_L}^\infty u_j^*(\cdot,s)\,ds  \,dx
\\&~~~
+\dashint_{\domain} a\nabla \int_{4T_L}^{\infty} u_i(\cdot,s)\,ds \cdot \nabla \int_{4T_L}^{\infty} u_j^*(\cdot,s)\,ds
\\&
\stackrel{\eqref{PartialDecomposition}}{=}\dashint_{\domain} a\bigg(e_i+\nabla \int_0^{1} u_i(\cdot,s)\,ds\bigg) \cdot \bigg(e_j+\nabla \int_0^{1} u_j^*(\cdot,s)\,ds\bigg) \,dx
\\&~~~
-\sum_{k=0}^{\log_2 L_K}  \dashint_{\domain} \int_{4^k}^{4^{k+1}} u_i(\cdot,s)\,ds~ u_j^*(\cdot,4^k) \,dx
\\&~~~
-\sum_{k=0}^{\log_2 L_K}  \dashint_{\domain} u_i(\cdot,4^k) \int_{4^k}^{4^{k+1}} u_j^*(\cdot,s)\,ds \,dx
\\&~~~
+\sum_{k=0}^{\log_2 L_K}  \dashint_{\domain} a\nabla \int_{4^k}^{4^{k+1}} u_i(\cdot,s)\,ds \cdot \nabla \int_{4^k}^{4^{k+1}} u_j^*(\cdot,s)\,ds \,dx
\\&~~~
-\dashint_{\domain} \int_{4T_L}^\infty u_i(\cdot,s)\,ds \, u_j^*(\cdot,4T_L)  \,dx
\\&~~~
-\dashint_{\domain} u_i(\cdot,4T_L) \int_{4T_L}^\infty u_j^*(\cdot,s)\,ds  \,dx
\\&~~~
+\dashint_{\domain} a\nabla \int_{4T_L}^{\infty} u_i(\cdot,s)\,ds \cdot \nabla \int_{4T_L}^{\infty} u_j^*(\cdot,s)\,ds \,dx.
\end{align*}
We now decompose the integrals into integrals over cubes with side length $\sim 2^k$, resulting in
\begin{align}
\label{aRVERewritten}
&a^\RVE e_i \cdot e_j
\\&
\nonumber
=\sum_{x_0\in \mathbb{Z}^d} \frac{1}{L^d} \dashint_{(x_0+[0,1]^d)\cap \domain} a\bigg(e_i+\nabla \int_0^{1} u_i(\cdot,s)\,ds\bigg) \cdot \bigg(e_j+\nabla \int_0^{1} u_j^*(\cdot,s)\,ds\bigg) \,dx
\\&~~~
\nonumber
-\sum_{k=0}^{\log_2 L_K} \sum_{x_0\in 2^k\mathbb{Z}^d} \frac{1}{L^d} \int_{(x_0+[0,2^k]^d)\cap \domain} \int_{4^k}^{4^{k+1}} u_i(\cdot,s)\,ds~ u_j^*(\cdot,4^k) \,dx
\\&~~~
\nonumber
-\sum_{k=0}^{\log_2 L_K} \sum_{x_0\in 2^k\mathbb{Z}^d}  \frac{1}{L^d} \int_{(x_0+[0,2^k]^d)\cap \domain} u_i(\cdot,4^k) \int_{4^k}^{4^{k+1}} u_j^*(\cdot,s)\,ds \,dx
\\&~~~
\nonumber
+\sum_{k=0}^{\log_2 L_K} \sum_{x_0\in 2^k \mathbb{Z}^d} \frac{1}{L^d} \int_{(x_0+[0,2^k]^d)\cap \domain} a\nabla \int_{4^k}^{4^{k+1}} u_i(\cdot,s)\,ds \cdot \nabla \int_{4^k}^{4^{k+1}} u_j^*(\cdot,s)\,ds \,dx
\\&~~~
\nonumber
-\dashint_{\domain} \int_{4T_L}^\infty u_i(\cdot,s)\,ds \, u_j^*(\cdot,4T_L)  \,dx
\\&~~~
\nonumber
-\dashint_{\domain} u_i(\cdot,4T_L) \int_{4T_L}^\infty u_j^*(\cdot,s)\,ds  \,dx
\\&~~~
\nonumber
+\dashint_{\domain} a\nabla \int_{4T_L}^{\infty} u_i(\cdot,s)\,ds \cdot \nabla \int_{4T_L}^{\infty} u_j^*(\cdot,s)\,ds\,dx.
\end{align}
We now intend to replace $u_i$ and $u_j^*$ in each of these expressions by a proxy with localized dependence. To this aim, for any $k\in \mathbb{N}_0$ and any $x_0\in 2^k \mathbb{Z}^d$, define the coefficient field $a_{k,x_0}$ on the full space $\mathbb{R}^d$ as
\begin{align}
\label{DefinitionLocalizeda}
a_{k,x_0}(x):=
\begin{cases}
a(x)&\text{for }|x-x_0|\leq  \sqrt{K|\log L|} \, 2^{k-1},
\\
\Id&\text{otherwise}.
\end{cases}
\end{align}
Define a corresponding $u_{i,k,x_0}$ as the solution to the equation
\begin{subequations}
\label{DefinitionLocalizedu}
\begin{align}
\frac{d}{dt} u_{i,k,x_0}&=\nabla \cdot (a_{k,x_0} \nabla u_{i,k,x_0}),
\\
u_{i,k,x_0}(\cdot,0)&=\nabla \cdot (a_{k,x_0} e_i),
\end{align}
\end{subequations}
and introduce analogously the function $u_{i,k,x_0}^*$ as the solution to the equation with $a_{k,x_0}$ replaced by $a^*_{k,x_0}$. Note that while $u_i$ and $a$ are defined on $[0,L]^d$ and extended to $\mathbb{R}^d$ by periodicity, both $a_{k,x_0}$ and $u_{i,k,x_0}$ are defined on $\mathbb{R}^d$ and lack any periodicity.

By Lemma~\ref{GaussianPropagationBounds} -- applied with $M:=\frac{1}{2} \sqrt{K |\log L|}$ and $r:=2^k$ -- we have
\begin{align}
\label{L2ErrorBoundLocalModification}
&\dashint_{\{|x-x_0|\leq 2d\cdot 2^{k}\}} |u_i(\cdot,t)-u_{i,k,x_0}(\cdot,t)|^2 \,dx
\\&
\nonumber
\leq C \sqrt{K \log L}^{d/2} \exp(-c K |\log L|) \leq C(d,\lambda,K) L^{-c K}
\end{align}
for any $t\leq 4^{k+1}$ and
\begin{align}
\label{H1ErrorBoundLocalModification}
&\int_{0}^{4^{k+1}} \dashint_{\{|x-x_0|\leq d \cdot 2^k\}} |\nabla u_i-\nabla u_{i,k,x_0}|^2 \,dx\,dt
\\&~~~~~~~~~~~~~~~
\nonumber
\leq C \exp(-c K |\log L|) \leq C(d,\lambda,K) L^{-c K}
\end{align}
and analogous estimates for the difference $u_j^*-u_{j,k,x_9}^*$.

As our probability distribution of coefficient fields $a$ on $[0,L]^d$ is the periodization of a probability distribution of coefficient fields $a^{\mathbb{R}^d}$ on $\mathbb{R}^d$, by definition of a periodization (see (A3$_c$)) for each $x_0\in [0,L)^d$ and any $k\leq \log_2 L_K$ the law of $a|_{x_0+K \log L [-2^k,2^k]^d}$ coincides with the law of $a^{\mathbb{R}^d}|_{x_0+K \log L [-2^k,2^k]^d}$. As a consequence, the law of $u_{i,k,x_0}$ coincides with the law of $u_{i,k,x_0}^{\mathbb{R}^d}$, where  $u_{i,k,x_0}^{\mathbb{R}^d}$ is defined analogously to $u_{i,k,x_0}$ (replacing $a$ in the definition by $a^{\mathbb{R}^d}$).
Therefore, any moment bound on $u_{i,k,x_0}^{\mathbb{R}^d}$ carries over to $u_{i,k,x_0}$. Applying Lemma~\ref{GaussianPropagationBounds} to $u_{i,k,x_0}^{\mathbb{R}^d}$, we obtain estimates analogous to \eqref{L2ErrorBoundLocalModification} and \eqref{H1ErrorBoundLocalModification}. The estimates from Theorem~\ref{TheoremSemigroupDecay} therefore carry over to $u_{i,k,x_0}^{\mathbb{R}^d}$, provided that we choose $K\geq C$: We have for $t\in [4^k,4^{k+1}]$ and $T=4^k$ with $2^k\leq L$
\begin{align*}
\left(\dashint_{\{|x-x_0|\leq d\cdot 2^k\}} |u_{i,k,x_0}^{\mathbb{R}^d}(t)|^2 \,dx\right)^{1/2}
&\leq \mathcal{C}(a^{\mathbb{R}^d},t) \, t^{-1/2-d/4},
\\
\left(\dashint_{T}^{4T} \dashint_{\{|x-x_0|\leq d\cdot 2^k\}} |\nabla u_{i,k,x_0}^{\mathbb{R}^d}(T)|^2 \,dx \,dt\right)^{1/2}
&\leq \mathcal{C}(a^{\mathbb{R}^d},T) \, T^{-1-d/4},
\end{align*}
for some random constants $\mathcal{C}(a^{\mathbb{R}^d},t)$, $\mathcal{C}(a^{\mathbb{R}^d},T)$ with
\begin{align*}
||\mathcal{C}(a^{\mathbb{R}^d},t)||_{\exp^{2-\delta}}&\leq C(d,\lambda,K,\delta),
\\
||\mathcal{C}(a^{\mathbb{R}^d},T)||_{\exp^{2-\delta}}&\leq C(d,\lambda,K,\delta),
\end{align*}
for any $\delta>0$.
By coincidence of laws, we get for $t\in [4^k,4^{k+1}]$ and $T=4^k$
\begin{subequations}
\label{BoundsOnLocalizedu}
\begin{align}
\left(\dashint_{\{|x-x_0|\leq d\cdot 2^k\}} |u_{i,k,x_0}(t)|^2 \,dx\right)^{1/2}
&\leq \mathcal{C}(a,t) \, t^{-1/2-d/4},
\\
\left(\dashint_{T}^{4T} \dashint_{\{|x-x_0|\leq d\cdot 2^k\}} |\nabla u_{i,k,x_0}(T)|^2 \,dx \,dt\right)^{1/2}
&\leq \mathcal{C}(a,T) \, T^{-1-d/4},
\end{align}
\end{subequations}
for random constants $\mathcal{C}$ satisfying
\begin{align*}
||\mathcal{C}(a,t)||_{\exp^{2-\delta}}&\leq C(d,\lambda,K,\delta),
\\
||\mathcal{C}(a,T)||_{\exp^{2-\delta}}&\leq C(d,\lambda,K,\delta),
\end{align*}
for any $\delta>0$. Furthermore, the bound \eqref{BoundHeatEquationIntegratedGradientWeakInitialData} yields an estimate of the form
\begin{align}
\label{BoundsOnLocalizedu2}
\Bigg(\dashint_{\{|x-x_0|\leq d\}} \bigg|e_i+\nabla \int_0^1 u_{i,0,x_0} \,ds\bigg|^2 \,dx\Bigg)^{1/2}
\leq C(d,\lambda).
\end{align}
By \eqref{PartialDecomposition}, its analogue for $u_{i,0,x_0}$, and the definition of $a_{0,x_0}$, we have in $\{|x-x_0|\leq 2d\}$ that $-\nabla \cdot (a\nabla(\int_0^1 u_i(\cdot,s)-u_{i,0,x_0}(\cdot,s) \,ds)) = u_i(\cdot,1)-u_{i,0,x_0}(\cdot,1)$, which implies by the Caccioppoli inequality
\begin{align}
\label{H1ErrorBoundLocalModificationLowest}
&\dashint_{\{|x-x_0|\leq d\}} \bigg|e_i+\nabla \int_0^1 u_i \,ds
-\bigg(e_i+\nabla \int_0^1 u_{i,0,x_0} \,ds\bigg)\bigg|^2 \,dx
\\&
\nonumber
\leq C \dashint_{\{|x-x_0|\leq 2d\}} |u_i(\cdot,1)-u_{i,0,x_0}(\cdot,1)|^2 \,dx
\\&~~~
\nonumber
+C \dashint_{\{|x-x_0|\leq 2d\}} \bigg|\int_0^1 u_i(\cdot,s)-u_{i,0,x_0}(\cdot,s) \,ds\bigg|^2 \,dx
\\&
\nonumber
\leq C \dashint_{\{|x-x_0|\leq 2d\}} |u_i(\cdot,1)-u_{i,0,x_0}(\cdot,1)|^2 \,dx
\\&~~~
\nonumber
+C \dashint_{\{|x-x_0|\leq 2d\}} \int_0^1 | u_i(\cdot,s)-u_{i,0,x_0}(\cdot,s) |^2 \,ds \,dx
\\&
\nonumber
\stackrel{\eqref{L2ErrorBoundLocalModification}}{\leq}
C(K) L^{-c K}.
\end{align}

As a consequence of our definition of $u_{i,k,x_0}$, for the choice
\begin{subequations}
\label{DefinitionXky}
\begin{align}
X_{x_0}^0
:=&
\frac{1}{L^d} \int_{(x_0+[0,1]^d)\cap\domain} a\bigg(e_i+\nabla \int_0^{1} u_{i,0,x_0}(\cdot,s)\,ds\bigg)
\\&~~~~~~~~~~~~~~~~~~~~~~~~~~~~~~~~~~~~
\nonumber
\cdot \bigg(e_j+\nabla \int_0^{1} u_{j,0,x_0}^*(\cdot,s)\,ds\bigg) \,dx,
\\
X_{x_0}^{k+1}
:=&-\frac{1}{L^d} \int_{(x_0+[0,2^k]^d)\cap\domain} \int_{4^k}^{4^{k+1}} u_{i,k,x_0}(\cdot,s)\,ds~ u_{j,k,x_0}^*(\cdot,4^k) \,dx
\\&
\nonumber
-\frac{1}{L^d}\int_{(x_0+[0,2^k]^d)\cap \domain} u_{i,k,x_0}(\cdot,4^k) \int_{4^k}^{4^{k+1}} u_{j,k,x_0}^*(\cdot,s)\,ds \,dx
\\&
\nonumber
+\frac{1}{L^d}\int_{(x_0+[0,2^k]^d)\cap \domain} a\nabla \int_{4^k}^{4^{k+1}} u_{i,k,x_0}(\cdot,s)\,ds
\\&~~~~~~~~~~~~~~~~~~~~~~~~~~~~~~~~~~~~~~
\nonumber
\cdot \nabla \int_{4^k}^{4^{k+1}} u_{j,k,x_0}^*(\cdot,s)\,ds \,dx,
\end{align}
\end{subequations}
for $0\leq k\leq \log_2 L_K$, we see by \eqref{DefinitionLocalizeda} and \eqref{DefinitionLocalizedu} and $\sqrt{K\log L}\geq 1$ that $X_{x_0}^k$ is a random variable which depends only on $a|_{x_0+K \log L [-2^k,2^k]^d}$, i.\,e.\ the first condition of Definition~\ref{ConditionRandomVariable} is satisfied. Furthermore, by \eqref{BoundsOnLocalizedu} and \eqref{BoundsOnLocalizedu2} we obtain for any $0<\gamma<1$ an estimate of the form
\begin{align}
\label{EstimateX}
||X_{y}^k||_{\exp^\gamma} \leq C(d,\lambda,\gamma,K) L^{-d}.
\end{align}
We now intend to replace the terms in the first five terms on the right-hand side of \eqref{aRVERewritten} by the $X_{x_0}^k$ with $0\leq k\leq \log_2 L_K+1$, using the estimates \eqref{L2ErrorBoundLocalModification}, \eqref{H1ErrorBoundLocalModification}, \eqref{H1ErrorBoundLocalModificationLowest}, and H\"older's inequality to bound the arising error: For example, we may estimate
\begin{align*}
&\bigg|
-\frac{1}{L^d} \int_{(x_0+[0,2^k]^d)\cap \domain} u_i(\cdot,4^k) \int_{4^k}^{4^{k+1}} u_j^*(\cdot,s)\,ds \,dx
\\&~~~
-\bigg(-\frac{1}{L^d} \int_{(x_0+[0,2^k]^d)\cap\domain} \int_{4^k}^{4^{k+1}} u_{i,k,x_0}(\cdot,s)\,ds~ u_{j,k,x_0}^*(\cdot,4^k) \,dx\bigg)\bigg|
\\&
\leq
\frac{4^{(k+1)/2}}{L^d} \bigg(\int_{(x_0+[0,2^k]^d)\cap \domain} |u_{i}(\cdot,4^k)|^2 \,dx\bigg)^{1/2}
\\&~~~~
\times
\bigg(\int_{(x_0+[0,2^k]^d)\cap \domain}  \int_{4^k}^{4^{k+1}} |u_{j}^*(\cdot,s)-u_{j,k,x_0}^*(\cdot,s)|^2\,ds \,dx\bigg)^{1/2}
\\&~~~~
+\frac{4^{(k+1)/2}}{L^d} \bigg(\int_{(x_0+[0,2^k]^d)\cap \domain} |u_{i}(\cdot,4^k)-u_{i,k,x_0}(\cdot,4^k)|^2 \,dx\bigg)^{1/2}
\\&~~~~~~~~
\times
\bigg(\int_{(x_0+[0,2^k]^d)\cap \domain}  \int_{4^k}^{4^{k+1}} |u_{j,k,x_0}^*(\cdot,s)|^2\,ds \,dx\bigg)^{1/2}
\\&
\stackrel{\eqref{L2ErrorBoundLocalModification}}{\leq}
\frac{C(d,\lambda,K)}{L^d} \bigg(\int_{(x_0+[0,2^k]^d)\cap \domain} |u_{i,k,x_0}(\cdot,4^k)|^2 \,dx+L^{-cK}\bigg)^{1/2} \cdot (2^k)^{d/2} L^{-cK}
\\&~~~
+\frac{C(d,\lambda,K)}{L^d} \cdot (2^k)^{d/2} L^{-cK} \cdot
\bigg(\int_{(x_0+[0,2^k]^d)\cap \domain}  \int_{4^k}^{4^{k+1}} |u_{j,k,x_0}^*(\cdot,s)|^2\,ds \,dx\bigg)^{1/2}
\end{align*}
where in the last step we have used $4^k \leq CL^2$ and $(2^k)^{d/2}\leq C L^{d/2}$, absorbing these factors in the factor $L^{-cK}$ (possible for $cK\geq 4+2d$). Proceeding analogously for the other terms in \eqref{aRVERewritten}, we deduce
\begin{align*}
&\Bigg|a^\RVE e_i \cdot e_j
-\sum_{x_0\in \mathbb{Z}^d\cap [0,L)^d} X_{x_0}^0
-\sum_{k=1}^{1+\log_2 L_K} \sum_{x_0\in 2^k \mathbb{Z}^d\cap [0,L)^d} X_{x_0}^k
\Bigg|
\\&
\leq C \sum_{x_0\in \mathbb{Z}^d\cap [0,L)^d} \frac{1}{L^d} \bigg(\int_{x_0+[0,1]^d} \bigg|e_i+\nabla \int_{0}^1 u_{i,0,x_0}(\cdot,s) \,ds\bigg|^2
\\&~~~~~~~~~~~~~~~~~~~~~~~~~~~~~~~~~~~~
+ \bigg|e_j+\nabla \int_{0}^1 u_{j,0,x_0}^*(\cdot,s) \,ds\bigg|^2 \,dx+ L^{-cK} \bigg)^{1/2}
L^{-cK}
\\&~~
+C\sum_{k=0}^{\log_2 L_K}
\sum_{x_0\in 2^k \mathbb{Z}^d\cap [0,L)^d} \frac{1}{L^d}
\bigg(\int_{x_0+[0,2^k]^d} |u_{i,k,x_0}(\cdot,4^k)|^2+|u_{j,k,x_0}^*(\cdot,4^k)|^2 \,dx+L^{-cK}\bigg)^{1/2}
\\&~~~~~~~~~~~~~~~~~~~~~~~~~~~~~~~~\times
(2^k)^{d/2} L^{-cK}
\\&~~
+C\sum_{k=0}^{\log_2 L_K}
\sum_{x_0\in 2^k \mathbb{Z}^d\cap [0,L)^d}
\frac{1}{L^d}
\bigg(\int_{x_0+[0,2^k]^d} \int_{4^k}^{4^{k+1}} |u_{i,k,x_0}(\cdot,s)|^2+|u_{j,k,x_0}^*(\cdot,s)|^2 \,ds \,dx\bigg)^{1/2}
\\&~~~~~~~~~~~~~~~~~~~~~~~~~~~~~~~~\times
(2^k)^{d/2} L^{-cK}
\\&~~
+C \sum_{k=0}^{\log_2 L_K}
\sum_{x_0\in 2^k \mathbb{Z}^d\cap [0,L)^d}
\frac{1}{L^d}
\bigg(\int_{x_0+[0,2^k]^d} \int_{4^k}^{4^{k+1}} |\nabla u_{i,k,x_0}(\cdot,s)|^2
\\&~~~~~~~~~~~~~~~~~~~~~~~~~~~~~~~~~~~~~~~~~~~~~~~~~~~
+|\nabla u_{j,k,x_0}^*(\cdot,s)|^2 \,ds \,dx +L^{-cK}\bigg)^{1/2} L^{-cK}
\\&~~
+\dashint_{\domain} \bigg|\int_{4T_L}^\infty u_i(\cdot,s)\,ds\bigg| \, |u_j^*(\cdot,4T_L)| \,dx
\\&~~
+\dashint_{\domain} |u_i(\cdot,4T_L)| \bigg|\int_{4T_L}^\infty u_j^*(\cdot,s)\,ds\bigg|  \,dx
\\&~~
+C \dashint_{\domain} \bigg|\nabla \int_{4T_L}^{\infty} u_i(\cdot,s)\,ds\bigg| \, \bigg|\nabla \int_{4T_L}^{\infty} u_j^*(\cdot,s)\,ds\bigg|\,dx.
\end{align*}
Inserting the estimates \eqref{BoundsOnLocalizedu}, \eqref{BoundsOnLocalizedu2}, we get for some $\mathcal{C}(a)$ with $||\mathcal{C}(a)||_{\exp^\gamma} \leq C(d,\lambda,K,\gamma)$ for any $\gamma\in (0,1)$
\begin{align}
\label{ErrorBoundOnDecomposition}
&\Bigg|a^\RVE e_i \cdot e_j
-\sum_{x_0\in \mathbb{Z}^d\cap [0,L)^d} X_{x_0}^0
-\sum_{k=1}^{1+\log_2 L_K} \sum_{x_0 \in 2^k \mathbb{Z}^d\cap [0,L)^d} X_{x_0}^k
\Bigg|
\\&
\nonumber
\leq C L^{-cK} + \mathcal{C}(a) \sum_{k=0}^{\log_2 L_K} L^{-cK} +  \mathcal{C}(a) \sum_{k=0}^{\log_2 L_K} \sqrt{4^k} L^{-cK}+ \mathcal{C}(a) \sum_{k=0}^{\log_2 L_K} L^{-cK}
\\&~~~~
\nonumber
+\dashint_{\domain} \bigg|\int_{4T_L}^\infty u_i(\cdot,s)\,ds\bigg| \, |u_j^*(\cdot,4T_L)| \,dx
\\&~~~~
\nonumber
+\dashint_{\domain} |u_i(\cdot,4T_L)| \bigg|\int_{4T_L}^\infty u_j^*(\cdot,s)\,ds\bigg|  \,dx
\\&~~~~
\nonumber
+C\dashint_{\domain} \bigg|\nabla \int_{4T_L}^{\infty} u_i(\cdot,s)\,ds\bigg| \, \bigg|\nabla \int_{4T_L}^{\infty} u_j^*(\cdot,s)\,ds\bigg|\,dx.
\end{align}
The bound \eqref{L2ErrorBoundLocalModification} and its equivalent for $u_i^{\mathbb{R}^d}$ and $u_{i,k,x_0}^{\mathbb{R}^d}$ enable us to transfer the bounds in Theorem~\ref{TheoremSemigroupDecay} from $u_i^{\mathbb{R}^d}$ to $u_i$: Recalling that $T_L=(L_K)^2$, we obtain
\begin{align}
\label{DecompositionOfui}
&\dashint_{\domain} |u_i(\cdot,T_L)|^2 \,dx
=\sum_{x_0\in L_K \mathbb{Z}^d} L^{-d} \int_{x_0+[0,L_K]^d\cap \domain} |u_i(\cdot,T_L)|^2 \,dx
\\&
\nonumber
\stackrel{\eqref{L2ErrorBoundLocalModification}}{\leq} C\sum_{x_0\in L_K \mathbb{Z}^d} L^{-d} \bigg(L^{-cK} + \int_{y+[0,L_K]^d\cap \domain} |u_{i,\log_2 L_K,x_0}(\cdot,T_L)|^2 \,dx\bigg)
\end{align}
and
\begin{align*}
&\int_{y+[0,L_K]^d\cap \domain} |u_{i,\log_2 L_K,x_0}^{\mathbb{R}^d}(\cdot,T_L)|^2 \,dx
\\&
\leq
2\int_{y+[0,L_K]^d\cap \domain} |u_i^{\mathbb{R}^d}(\cdot,T_L)|^2 \,dx
+2 C L^{-cK/2}.
\end{align*}
The latter estimate entails in view of Theorem~\ref{TheoremSemigroupDecay} (choosing $K\geq C$ and recalling that $\sqrt{T_L}=L_K\leq \frac{L}{4K \log L}$)
\begin{align*}
&\bigg(\dashint_{y+[0,L_K]^d\cap \domain} |u_{i,\log_2 L_K,x_0}^{\mathbb{R}^d}(\cdot,T_L)|^2 \,dx\bigg)^{1/2}
\\&
\leq \mathcal{C}(a^{\mathbb{R}^d},y,T_L)
T_L^{-1/2-d/4}
\end{align*}
where again $||\mathcal{C}(a^{\mathbb{R}^d},y,T_L)||_{\exp^{2-\delta}}\leq C(d,\lambda,K,\delta)$. By coincidence of the laws of $a|_{x_0+K \log L [-L_K,L_K]^d}$ and $a^{\mathbb{R}^d}|_{x_0+K \log L [-L_K,L_K]^d}$, we get for $K\geq C$ from the previous estimate and \eqref{DecompositionOfui} 
\begin{align}
\label{EstimateuiTL}
&\dashint_{\domain} |u_i(\cdot,T_L)|^2 \,dx
\leq \mathcal{C}(a,T_L) T_L^{-1-d/2}
\end{align}
where $||\mathcal{C}(a,T_L)||_{\exp^\gamma}\leq C(d,\lambda,K,\gamma)$ for any $\gamma<1$. An analogous bound holds for $u_j^*$. Finally, the energy estimate for $u_i$ implies
\begin{align*}
\frac{d}{dt} \dashint_{\domain} |u_i|^2 \,dx \leq  -c \dashint_{\domain} |\nabla u_i|^2 \,dx.
\end{align*}
As the average of $u_i$ over $\domain$ vanishes, the Poincar\'e inequality implies for $T\geq T_L$
\begin{align*}
\frac{d}{dt} \dashint_{\domain} |u_i|^2 \,dx \leq  -\frac{c}{2} \dashint_{\domain} |\nabla u_i|^2 \,dx
-\frac{c}{2 C L^2} \dashint_{\domain} |u_i|^2 \,dx
\end{align*}
and as a consequence
\begin{align*}
&\dashint_{\domain} |u_i(\cdot,T)|^2 \,dx+\int_{\max\{T_L,T/2\}}^T  \dashint_{\domain} |\nabla u_i|^2 \,dx \,dt
\\&
\leq C(d,\lambda) \exp\Big(-\frac{T-T_L}{C(d,\lambda) L^2}\Big) \dashint_{\domain} |u_i(\cdot,T_L)|^2 \,dx.
\end{align*}
Note that this estimate yields in particular
\begin{align*}
\dashint_{\domain} \bigg|\int_{T_L}^\infty \nabla u_i \,dt \bigg|^2 \,dx
&\leq 2 \sum_{l=-\log_2 \frac{L^2}{T_L}}^\infty
\dashint_{\domain} \bigg| 2^{l} \int_{2^l L^2}^{2^{l+1}L^2} \nabla u_i \,dt \bigg|^2 \,dx
\\&
\leq 2 \sum_{l=-\log_2 \frac{L^2}{T_L}}^\infty
\dashint_{\domain} 2^{l} (2^l L^2) \int_{2^l L^2}^{2^{l+1}L^2} |\nabla u_i|^2 \,dt\,dx
\\&
\leq C \sum_{l=-\log_2 \frac{L^2}{T_L}}^\infty
2^{2l} L^2 \exp(-c2^l) \dashint_{\domain} |u_i(\cdot,T_L)|^2 \,dx
\\&
\stackrel{\eqref{EstimateuiTL}}{\leq}
\mathcal{C}(a,T_L) \frac{L^2}{T_L} T_L^{-d/2}\leq \mathcal{C}(a,T_L) (K|\log L|)^{d+2} L^{-d/2}
\end{align*}
where in the last step we have used that $\sqrt{T_L}=L_K$ is the largest power of $2$ with $L_K\leq \frac{L}{4 K \log L}$.

Plugging in these bounds and \eqref{EstimateuiTL} into \eqref{ErrorBoundOnDecomposition}, we get for $K\geq C$
\begin{align}
\label{EstimateaRVEMultilevelDecomposition}
&\Bigg|a^\RVE e_i \cdot e_j
-\sum_{k=0}^{1+\log_2 L_K} \sum_{x_0\in 2^k \mathbb{Z}^d\cap [0,L)^d} X_{x_0}^k
\Bigg|
\\&
\nonumber
\leq C L^{-2d}
+\mathcal{C}(a,T_L) (K|\log L|)^{2d+4} L^{-d}
\end{align}
with $||\mathcal{C}(a,T_L)||_{\exp^\gamma}\leq C(d,\lambda,K,\gamma)$ for any $\gamma<1$.
Choosing $\gamma\in (0,1)$ and $B:=C(d,\lambda,K,\gamma) (4K \log L)^{2+d}$ in Definition~\ref{ConditionRandomVariable}, defining the variable $X_0^{\log_2 L+1}$ (which may depend on $a$ on the full volume $[0,L]^d$) to account for the remaining difference $a^\RVE e_i \cdot e_j-\sum_{k=0}^{1+\log_2 L_K} \sum_{x_0\in 2^k \mathbb{Z}^d\cap [0,L)^d} X_{x_0}^k$, and setting the remaining $X_i^{k}:=0$ for $\log_2 L_K+1<k<\log_2 L+1$, this establishes that $a^\RVE$ may be rewritten as a sum of a family of random variables with multilevel local dependence.

{\bf Part 3: The higher-order statistical quantity.}
Next, we derive the multilevel decomposition of the higher-order quantity in the small ellipticity contrast setting $\mathcal{F}_{2-point}$. To do so, we decompose the solution $v_i$ to \eqref{F2Equation} as
\begin{align}
\label{DecompositionF2}
v_i(\cdot)=\int_0^\infty w_i(\cdot,s) \,ds,
\end{align}
where $w_i$ is defined as the solution to the parabolic PDE
\begin{align*}
\frac{d}{dt} w_i &= \Delta w_i,
\\
w_i(\cdot,0)&=\nabla \cdot (a e_i).
\end{align*}
As before, the representation \eqref{DecompositionF2} follows from the exponential decay of $w_i$, as we have $-\Delta \int_0^T w_i (\cdot,t)\,dt = \nabla \cdot (ae_i)-w_i(\cdot,T)$.

We introduce analogous definitions for $v_j^*$. Again, we may assume without loss of generality that $\varepsilon=1$. We then observe following an argument of Mourrat \cite{MourratNumerical} that by formula \eqref{CommutingSemigroup} below
\begin{align*}
\mathcal{F}_{2-point}(a)&=
\dashint_{[0,L]^d} a\nabla v_i \cdot e_j \,dx
\\&
=\dashint_{[0,L]^d} -\nabla v_i \cdot \nabla v_j^* \,dx
\\&
=\dashint_{[0,L]^d} \int_0^\infty \int_0^\infty -\nabla w_i(\cdot,s_1)\cdot \nabla w_j^*(\cdot,s_2) \,ds_1 \,ds_2 \,dx
\\&
\stackrel{\eqref{CommutingSemigroup}}{=}
\dashint_{[0,L]^d} \int_0^\infty \int_0^\infty -\nabla w_i\Big(\cdot,\frac{s_1+s_2}{2}\Big) \cdot \nabla w_j^*\Big(\cdot,\frac{s_1+s_2}{2}\Big) \,ds_1 \,ds_2 \,dx.
\end{align*}
Next, we deduce
\begin{align*}
\mathcal{F}_{2-point}(a)&
=-\dashint_{[0,L]^d} \int_0^\infty 4s \, \nabla w_i(\cdot,s) \cdot \nabla w_j^*(\cdot,s) \,ds \,dx
\\&
=
-\sum_{x_0\in \mathbb{Z}^d} \frac{1}{L^{d}} \int_0^1 \int_{(x_0+[0,1]^d) \cap \domain} 4s \, \nabla w_i(\cdot,s) \cdot \nabla w_j^*(\cdot,s) \,dx \,ds
\\&~~~
-\sum_{k=1}^{\log_2 L_K}
\sum_{x_0\in \mathbb{Z}^d} \frac{1}{L^{d}} \int_{4^{k-1}}^{4^k} \dashint_{[0,L]^d} 4s \, \nabla w_i(\cdot,s) \cdot \nabla w_j^*(\cdot,s) \,dx \,ds
\\&~~~
-\frac{1}{L^{d}} \int_{T_L}^{\infty} \dashint_{[0,L]^d} 4s \, \nabla w_i(\cdot,s) \cdot \nabla w_j^*(\cdot,s) \,dx \,ds.
\end{align*}
We may now proceed to argue just like in the case of $a^\RVE$. The required decay estimates for the semigroup of the form
\begin{align*}
\bigg(\dashint_{\{|x-x_0|\leq \sqrt{T}\}} |\nabla w_i(\cdot,T)|^2 \,dx\bigg)^{1/2}
\leq \mathcal{C}(a,T,x_0) T^{-1-d/4}
\end{align*}
(with $||\mathcal{C}(a,T,x_0)||_{\exp^2}\leq C(d,\lambda)$) are now a consequence of the explicit heat kernel representation of the solution $w_i$ (as we are now dealing with a constant-coefficient parabolic equation), the finite range of dependence $\varepsilon=1$ of the initial data $w_i(\cdot,0)=\nabla \cdot (ae_i)$, and standard Gaussian concentration estimates (or, alternatively, -- though then with a less strong stretched exponential bound -- the concentration estimates of Lemma~\ref{ConcentrationStretchedExponential}).

In the computation above we have used the simple fact that
\begin{align}
\nonumber
&\dashint_{[0,L]^d} \nabla w_i(\cdot,s_1) \otimes \nabla w_j^*(\cdot,s_2) \,dx
\\&
\nonumber
=\dashint_{[0,L]^d} \nabla w_i\Big(\cdot,\frac{s_1+s_2}{2}\Big) \otimes \nabla w_j^*\Big(\cdot,\frac{s_1+s_2}{2}\Big) \,dx
\\&~~~
\nonumber
-\int_0^1 \frac{d}{d\rho} \dashint_{[0,L]^d} \nabla w_i\Big(\cdot,\frac{(2-\rho)s_1+\rho s_2}{2}\Big) \otimes \nabla w_j^*\Big(\cdot,\frac{\rho s_1+(2-\rho)s_2}{2}\Big) \,dx \,d\rho
\\&
\label{CommutingSemigroup}
=\dashint_{[0,L]^d} \nabla w_i\Big(\cdot,\frac{s_1+s_2}{2}\Big) \otimes \nabla w_j^*\Big(\cdot,\frac{s_1+s_2}{2}\Big) \,dx.
\end{align}
{\bf Part 4: Convergence of the variance.}
Finally, we prove that the rescaled variances $L^d \Var a^\RVE$, $L^d \Var \mathcal{F}_{avg}(a)$, and $L^d \Var \mathcal{F}_{2-point}(a)$ and the covariances $L^d \Cov[a^\RVE,\mathcal{F}_{avg}(a)]$, $L^d \Cov[a^\RVE,\mathcal{F}_{2-point}(a)]$, and $L^d \Cov[\mathcal{F}_{avg}(a),\mathcal{F}_{2-point}(a)]$ converge for $L\rightarrow \infty$. We limit ourselves to proving convergence of the rescaled variance $L^d\Var a^\RVE$; the proofs for the convergence of the other variances and the covariances are analogous. Furthermore, to simplify notation we limit ourselves to proving convergence of the variance for $L=2^n$ for some $n\in \mathbb{N}$; the proof in the general case is similar.

By Lemma~\ref{MultilevelVariableStretchedExponentialBound}, we obtain $\Var a^\RVE \leq C(d,\lambda,K) L^{-d} |\log L|^{C(d)}$. Using \eqref{EstimateaRVEMultilevelDecomposition} and this estimate, we deduce
\begin{align*}
&\bigg|\Var a^\RVE
-\sum_{k=0}^{1+\log_2 L_K} \sum_{\tilde k=0}^{1+\log_2 L_K} \sum_{y\in 2^k \mathbb{Z}^d\cap [0,L)^d}
\sum_{\tilde y\in 2^{\tilde k} \mathbb{Z}^d\cap [0,L)^d}
\Cov[X_{y}^k,X_{\tilde y}^{\tilde k}]\bigg|
\\&
\leq C(d,\lambda,K) |\log L|^C L^{-3d/2}.
\end{align*}
Expanding the sum and using stochastic independence of many of these terms, we may write
\begin{align*}
&\Bigg|\Var a^\RVE
-\sum_{k=0}^{1+\log_2 L_K} \sum_{y\in 2^k \mathbb{Z}^d\cap [0,L)^d} \sum_{\substack{\tilde y\in 2^{k} \mathbb{Z}^d\cap [0,L)^d:\\ |y-\tilde y|_{\operatorname{per}}\leq C K \log L \cdot 2^k}}
\Cov[X_{y}^k,X_{\tilde y}^{k}]
\\&
-\sum_{\tilde k=0}^{1+\log_2 L_K} \sum_{k=\tilde k+1}^{1+\log_2 L_K} \sum_{y\in 2^k \mathbb{Z}^d\cap [0,L)^d}
\sum_{\substack{\tilde y\in 2^{\tilde k} \mathbb{Z}^d\cap [0,L)^d:\\ |y-\tilde y|_{\operatorname{per}}\leq C K \log L \cdot 2^k}}
(\Cov[X_{y}^k,X_{\tilde y}^{\tilde k}]+\Cov[X_{\tilde y}^{\tilde k},X_{y}^k])\Bigg|
\\&
\leq C(d,\lambda,K) |\log L|^C L^{-3d/2}.
\end{align*}
Denote by $X_{y}^{k,\mathbb{R}^d}$ the quantities defined as in \eqref{DefinitionXky} but with $u_{i,k,x_0}$ and $u_{j,k,x_0}^*$ replaced by $u_i^{\mathbb{R}^d}$ and $u_j^{*,\mathbb{R}^d}$, i.\,e.\ for example for $k\geq 0$ and $y\in 2^k \mathbb{Z}^d$
\begin{align*}
X_{y}^{k,\mathbb{R}^d}
:=&-\frac{1}{L^d} \int_{(y+[0,2^k]^d)} \int_{4^k}^{4^{k+1}} u_i^{\mathbb{R}^d}(\cdot,s)\,ds~ u_j^{*,\mathbb{R}^d}(\cdot,4^k) \,dx
\\&
-\frac{1}{L^d}\int_{(y+[0,2^k]^d)} u_i^{\mathbb{R}^d}(\cdot,4^k) \int_{4^k}^{4^{k+1}} u_j^{*,\mathbb{R}^d}(\cdot,s)\,ds \,dx
\\&
+\frac{1}{L^d}\int_{(y+[0,2^k]^d)} a\nabla \int_{4^k}^{4^{k+1}} u_i^{\mathbb{R}^d}(\cdot,s)\,ds
\cdot \nabla \int_{4^k}^{4^{k+1}} u_j^{*,\mathbb{R}^d}(\cdot,s)\,ds \,dx.
\end{align*}
Set $X_{y}^{k,\infty}:=L^d X_{y}^{k,\mathbb{R}^d}$.
Note that $\Cov[X_{y}^{k,\infty},X_{\tilde y}^{\tilde k,\infty}]$ does not depend on $L$ (by definition of $X_y^{k,\mathbb{R}^d}$). By the full-space variants of the estimates \eqref{L2ErrorBoundLocalModification}, \eqref{H1ErrorBoundLocalModification}, and \eqref{H1ErrorBoundLocalModificationLowest} (i.\,e.\ the estimates for the differences $u_i^{\mathbb{R}^d}-u_{i,k,x_0}^{\mathbb{R}^d}$ etc., which are derived in exactly the same way) and \eqref{EstimateX} as well as the equality of laws of (products of the) $u_{i,k,x_0}$ etc.\ and (products of the) $u_{i,k,x_0}^{\mathbb{R}^d}$ etc.\,, we get for $k,\tilde k\leq 1+\log_2 L_K$
\begin{align}
\label{UpperBoundCovDifference}
\big|\Cov[X_{\tilde y}^{\tilde k},X_y^{k}]
-\Cov[X_{\tilde y}^{\tilde k,\mathbb{R}^d},X_y^{k,\mathbb{R}^d}]
\big|
\leq C(d,\lambda,K) L^{-cK}.
\end{align}
By the definition of the $X_y^k$ (see \eqref{DefinitionXky}), the definition of the $u_{i,k,x_0}$, and the stationarity of the probability distribution of $a^{\mathbb{R}^d}$, the covariance $\Cov[X_{y}^{k,\mathbb{R}^d},X_{\tilde y}^{\tilde k,\mathbb{R}^d}]$ depends only on $k$, $\tilde k$, $y-\tilde y$, $L$, and the law of $a^{\mathbb{R}^d}$ (but not on $y$ for fixed $y-\tilde y$). Furthermore, by \eqref{EstimateX} we have $|\Cov[X_{\tilde y}^{\tilde k},X_y^k]|\leq C L^{-2d}$.
This implies by \eqref{UpperBoundCovDifference}
\begin{align*}
&\bigg|
\Var a^\RVE
-\sum_{k=0}^{1+\log_2 L_K} \bigg(\frac{L}{2^k}\bigg)^d \sum_{\substack{\tilde y\in 2^{k} \mathbb{Z}^d\cap [-L/2,L/2)^d: \\ |\tilde y-0|\leq C K \log L \cdot 2^k}}
\Cov[X_{0}^{k,\mathbb{R}^d},X_{\tilde y}^{k,\mathbb{R}^d}]
\\&~~
-\sum_{\tilde k=0}^{1+\log_2 L_K} \sum_{k=\tilde k+1}^{1+\log_2 L_K} \bigg(\frac{L}{2^k}\bigg)^d
\sum_{\substack{\tilde y\in 2^{\tilde k} \mathbb{Z}^d\cap [-L/2,L/2)^d:\\ |\tilde y-0| \leq C K \log L \cdot 2^k}}
(\Cov[X_{0}^{k,\mathbb{R}^d},X_{\tilde y}^{\tilde k,\mathbb{R}^d}]+\Cov[X_{\tilde y}^{\tilde k,\mathbb{R}^d},X_{0}^{k,\mathbb{R}^d}])
\bigg|
\\&
\leq C(d,\lambda,K) |\log L|^C L^{-3d/2}
+\sum_{\tilde k=0}^{1+\log_2 L_K} \sum_{k=\tilde k}^{1+\log_2 L_K} \bigg(\frac{L}{2^k}\bigg)^d \cdot C(d,\lambda,K) L^{-cK}
\\&
\leq C(d,\lambda,K) |\log L|^C L^{-3d/2}
\end{align*}
for $K$ chosen large enough.

The fact that (by stochastic independence) we have $\Cov[L^d X_{\tilde y}^{\tilde k},L^d X_y^{k}]=0$ for $|y-\tilde y|_\per \geq C(d) 2^k K \log L$ and $k\geq \tilde k$ implies together with \eqref{UpperBoundCovDifference} and the definition of $X_y^{k,\infty}$ that (by selecting $K$ large enough and by choosing $L$ to be just small enough for $|y-\tilde y|\geq C(d) 2^k K \log L$ to hold in case $|y-\tilde y|\geq C(d) K 2^k$ and otherwise -- i.\,e.\ for $|y-\tilde y|\leq C(d) K 2^k$ -- appealing to the upper bound \eqref{EstimateX})
\begin{align}
\label{BoundCovariances}
\big|\Cov[X_{\tilde y}^{\tilde k,\infty},X_y^{k,\infty}]\big|
\leq C(d,\lambda,K) \exp\Big(-\frac{|y-\tilde y|}{C(d,\lambda) 2^k}\Big).
\end{align}
As a consequence, we obtain
\begin{align*}
&\bigg|
L^d \Var a^\RVE
-\sum_{k=0}^{1+\log_2 L_K} (2^k)^{-d} \sum_{\tilde y\in 2^k \mathbb{Z}^d}
\Cov[X_{0}^{k,\infty},X_{\tilde y}^{k,\infty}]
\\&~~
-\sum_{\tilde k=0}^{1+\log_2 L_K} \sum_{k=\tilde k+1}^{1+\log_2 L_K} (2^k)^{-d}
\sum_{\tilde y\in 2^{\tilde k} \mathbb{Z}^d}
(\Cov[X_{0}^{k,\infty},X_{\tilde y}^{\tilde k,\infty}]
+\Cov[X_{\tilde y}^{\tilde k,\infty},X_{0}^{k,\infty}])
\bigg|
\\&
\leq C(d,\lambda,K) |\log L|^C L^{-d/2}
\\&~~~~
+ \sum_{\tilde k=0}^{1+\log_2 L_K} \sum_{k=\tilde k}^{1+\log_2 L_K} (2^k)^{-d} \sum_{\substack{\tilde y\in 2^{\tilde k} \mathbb{Z}^d:\\ |\tilde y-0|> C K \log L \cdot 2^k}}
C(d,\lambda,K) \exp\bigg(-\frac{|\tilde y-0|}{C 2^k}\bigg)
\\&
\leq C(d,\lambda,K) |\log L|^C L^{-d/2}
\\&~~~~
+ \sum_{\tilde k=0}^{1+\log_2 L_K} \sum_{k=\tilde k}^{1+\log_2 L_K} (2^k)^{-d} \cdot \bigg(\frac{2^k}{2^{\tilde k}}\bigg)^d
C(d,\lambda,K) \exp(-c K \log L)
\\&
\leq C(d,\lambda,K) |\log L|^C L^{-d/2}.
\end{align*}
This implies
\begin{align*}
&\bigg|
L^d \Var a^\RVE
-\sum_{k=0}^{\infty} (2^k)^{-d} \sum_{\tilde y\in 2^k \mathbb{Z}^d}
\Cov[X_{0}^{k,\infty},X_{\tilde y}^{k,\infty}]
\\&~~
-\sum_{\tilde k=0}^{\infty} \sum_{k=\tilde k+1}^{\infty} (2^k)^{-d}
\sum_{\tilde y\in 2^{\tilde k} \mathbb{Z}^d}
(\Cov[X_{0}^{k,\infty},X_{\tilde y}^{\tilde k,\infty}]
+\Cov[X_{\tilde y}^{\tilde k,\infty},X_{0}^{k,\infty}])
\bigg|
\\&
\leq C(d,\lambda,K) |\log L|^C L^{-d/2}
\\&~~~~~
+2\sum_{\tilde k=0}^{1+\log_2 L_K} \sum_{k=2+\log_2 L_K}^\infty (2^k)^{-d} \Bigg|\sum_{\tilde y\in 2^{\tilde k} \mathbb{Z}^d} \Cov[X_0^{k,\infty},X_y^{\tilde k,\infty}]\Bigg|
\\&~~~~~
+2\sum_{\tilde k=2+\log_2 L_K}^\infty \sum_{k=\tilde k}^\infty (2^k)^{-d} \Bigg|\sum_{\tilde y\in 2^{\tilde k} \mathbb{Z}^d} \Cov[X_0^{k,\infty},X_y^{\tilde k,\infty}]\Bigg|.
\end{align*}
We now distinguish the cases $\tilde y\in [-R_k 2^k,R_k 2^k]^d$ and $\tilde y\notin [-R_k 2^k,R_k 2^k]^d$ for some $R_k$ to be chosen. Using \eqref{BoundCovariances} in the latter case, we get
\begin{align*}
&\bigg|
L^d \Var a^\RVE
-\sum_{k=0}^{\infty} (2^k)^{-d} \sum_{\tilde y\in 2^k \mathbb{Z}^d}
\Cov[X_{0}^{k,\infty},X_{\tilde y}^{k,\infty}]
\\&~~
-\sum_{\tilde k=0}^{\infty} \sum_{k=\tilde k+1}^{\infty} (2^k)^{-d}
\sum_{\tilde y\in 2^{\tilde k} \mathbb{Z}^d}
(\Cov[X_{0}^{k,\infty},X_{\tilde y}^{\tilde k,\infty}]
+\Cov[X_{\tilde y}^{\tilde k,\infty},X_{0}^{k,\infty}])
\bigg|
\\&
\leq C(d,\lambda,K) |\log L|^C L^{-d/2}
\\&~~~~~
+2\sum_{\tilde k=0}^{1+\log_2 L_K} \sum_{k=2+\log_2 L_K}^\infty (2^k)^{-d} \Bigg|\sum_{\tilde y\in 2^{\tilde k} \mathbb{Z}^d\cap [-R_k 2^k,R_k 2^k]^d} \Cov[X_0^{k,\infty},X_y^{\tilde k,\infty}]\Bigg|
\\&~~~~~
+2\sum_{\tilde k=0}^{1+\log_2 L_K} \sum_{k=2+\log_2 L_K}^\infty (2^k)^{-d} \cdot C(d,\lambda,K) \bigg(\frac{2^k}{2^{\tilde k}}\bigg)^d \exp\Big(-\frac{R_k}{C}\Big)
\\&~~~~~
+2\sum_{\tilde k=2+\log_2 L_K}^\infty \sum_{k=\tilde k}^\infty (2^k)^{-d} \Bigg|\sum_{\tilde y\in 2^{\tilde k} \mathbb{Z}^d\cap [-R_k 2^k,R_k 2^k]^d} \Cov[X_0^{k,\infty},X_y^{\tilde k,\infty}]\Bigg|
\\&~~~~~
+2\sum_{\tilde k=2+\log_2 L_K}^\infty \sum_{k=\tilde k}^\infty (2^k)^{-d} \cdot C(d,\lambda,K) \bigg(\frac{2^k}{2^{\tilde k}}\bigg)^d \exp\Big(-\frac{R_k}{C}\Big).
\end{align*}
For $\tilde k \leq k$ and $R 2^k \leq L_K$ we have by Lemma~\ref{MultilevelVariableStretchedExponentialBound} and \eqref{EstimateX}
\begin{align*}
&\Bigg|\Cov\Bigg[X_y^{k},\sum_{\tilde y\in 2^{\tilde k} \mathbb{Z}^d \cap [-R 2^k,R 2^k]^d} X_{\tilde y}^{\tilde k}
\Bigg]\Bigg|
\\&
\leq \sqrt{\big|\Var X_y^{k}\big|} \sqrt{\bigg|\Var \sum_{\tilde y\in 2^{\tilde k} \mathbb{Z}^d \cap [-R 2^k,R 2^k]^d} X_{\tilde y}^{\tilde k}\bigg|}
\\&
\leq C(d,\lambda,K) L^{-2d} \bigg(\frac{R 2^k}{2^{\tilde k}}\bigg)^{d/2} \bigg|\log \frac{R2^k}{2^{\tilde k}}\bigg|^{d/2}
\end{align*}
which entails by \eqref{UpperBoundCovDifference} upon choosing $L^{1/2}=R2^k$
\begin{align*}
&\Bigg|\Cov\Bigg[X_y^{k,\infty},\sum_{\tilde y\in 2^{\tilde k} \mathbb{Z}^d \cap [-R 2^k,R 2^k]^d} X_{\tilde y}^{\tilde k,\infty}
\Bigg]\Bigg|
\\&
\leq C(d,\lambda,K) \bigg(\frac{R 2^k}{2^{\tilde k}}\bigg)^{d/2} |\log (R2^k)|^d.
\end{align*}
As a consequence, choosing $R_k=Sk$ for $S\geq 1$ large enough we get
\begin{align*}
&\bigg|
L^d \Var a^\RVE
-\sum_{k=0}^{\infty} (2^k)^{-d} \sum_{\tilde y\in 2^k \mathbb{Z}^d}
\Cov[X_{0}^{k,\infty},X_{\tilde y}^{k,\infty}]
\\&~~
-\sum_{\tilde k=0}^{\infty} \sum_{k=\tilde k+1}^{\infty} (2^k)^{-d}
\sum_{\tilde y\in 2^{\tilde k} \mathbb{Z}^d}
(\Cov[X_{0}^{k,\infty},X_{\tilde y}^{\tilde k,\infty}]
+\Cov[X_{\tilde y}^{\tilde k,\infty},X_{0}^{k,\infty}])
\bigg|
\\&
\leq C(d,\lambda,K) |\log L|^C L^{-d/2}
\\&~~~~~
+2\sum_{\tilde k=0}^{1+\log_2 L_K} \sum_{k=2+\log_2 L_K}^\infty (2^k)^{-d} \cdot C(d,\lambda,K) \bigg(\frac{R_k 2^k}{2^{\tilde k}}\bigg)^{d/2} |\log (R_k 2^k)|^d
\\&~~~~~
+C(d,\lambda,K) \sum_{k=2+\log_2 L_K}^\infty \exp\Big(-\frac{R_k}{C}\Big)
\\&~~~~~
+2\sum_{\tilde k=2+\log_2 L_K}^\infty \sum_{k=\tilde k}^\infty (2^k)^{-d} \cdot C(d,\lambda,K) \bigg(\frac{R_k 2^k}{2^{\tilde k}}\bigg)^{d/2} |\log (R_k 2^k)|^d
\\&~~~~~
+C(d,\lambda,K) (L_K)^{-d} \sum_{k=0}^\infty \exp\Big(-\frac{R_k}{C}\Big)
\\&
\leq C(d,\lambda,K) |\log L|^C L^{-d/2}
\\&~~~~~
+C(d,\lambda,K) (L_K)^{-d/2} |\log L_K|^C
\\&~~~~~
+C(d,\lambda,K) \exp\Big(-\frac{S\log L_K}{C}\Big)
\\&~~~~~
+C(d,\lambda,K) (L_K)^{-d} |\log L_K|^C
\\&~~~~~
+C(d,\lambda,K) (L_K)^{-d}
\\&
\leq C(d,\lambda,K) |\log L|^C L^{-d/2}.
\end{align*}
In total, we have shown convergence of the rescaled variance $L^d \Var a^\RVE$ towards a limit independent of $L$ with the desired rate.

The proof of the other cases is analogous.
\end{proof}

\begin{proof}[Proof of Theorem~\ref{TheoremSemigroupDecay}]
The estimate \eqref{SemigroupDecay1} is contained in \cite[Corollary~4]{GloriaOttoNew}.
In view of the Poincar\'e inequality the bound \eqref{SemigroupDecay2} is a consequence of \eqref{SemigroupDecay1} and an estimate on a (weighted) average of $u_i^{\mathbb{R}^d}$. Hence, we only need to derive a bound on
\begin{align*}
\int u_i^{\mathbb{R}^d}(\cdot,T) \frac{1}{\sqrt{T}^{d}} \psi\Big(\frac{x}{\sqrt{T}}\Big) \,dx
\end{align*}
for a suitably chosen smooth function $\psi$ supported in $\{|x|\leq 1\}$. To this aim, we compute
\begin{align*}
&\int u_i^{\mathbb{R}^d}(\cdot,T) \frac{1}{\sqrt{T}^{d}} \psi\Big(\frac{x}{\sqrt{T}}\Big) \,dx
\\&
=\int u_i^{\mathbb{R}^d}(\cdot,2T) \frac{1}{\sqrt{T}^{d}} \psi\Big(\frac{x}{\sqrt{T}}\Big) \,dx -\int_T^{2T} \int \frac{1}{\sqrt{T}^{d}} \psi\Big(\frac{x}{\sqrt{T}}\Big) \frac{d}{dt} u_i^{\mathbb{R}^d} \,dx\,dt
\\&
=\int u_i^{\mathbb{R}^d}(\cdot,2T) \frac{1}{\sqrt{2T}^{d}} \psi\Big(\frac{x}{\sqrt{2T}}\Big) \,dx
\\&~~~
+\int u_i^{\mathbb{R}^d}(\cdot,2T) \bigg(\frac{1}{\sqrt{T}^{d}} \psi\Big(\frac{x}{\sqrt{T}}\Big)-\frac{1}{\sqrt{2T}^{d}} \psi\Big(\frac{x}{\sqrt{2T}}\Big)\bigg) \,dx
\\&~~~
-\int_T^{2T} \int \frac{1}{\sqrt{T}^{d+1}} \nabla \psi \Big(\frac{x}{\sqrt{T}}\Big) \cdot a\nabla u_i^{\mathbb{R}^d} \,dx\,dt
\end{align*}
which yields upon applying the Poincar\'e inequality to the second term (note that the second factor in the integral has vanishing average) and using the bound \eqref{SemigroupDecay1}
\begin{align*}
&\Bigg|\int u_i^{\mathbb{R}^d}(\cdot,T) ~ \frac{1}{\sqrt{T}^{d}} \psi\Big(\frac{x}{\sqrt{T}}\Big) \,dx
-\int u_i^{\mathbb{R}^d}(\cdot,2T) \frac{1}{\sqrt{2T}^{d}} \psi\Big(\frac{x}{\sqrt{2T}}\Big) \,dx
\Bigg|
\\&
\leq C(d) \mathcal{C}(a,2T) (2T)^{-1/2-d/4}
+ C(d,\lambda) \int_T^{2T} \mathcal{C}(a,t) t^{-1-d/4} \sqrt{T}^{-1} \,dt
\\&
\leq C(d) \mathcal{C}(a,T) T^{-1/2-d/4}
\end{align*}
Summing over a dyadic sequence of times $2^k T$ and using the fact that almost surely
\begin{align*}
\lim_{T\rightarrow \infty} \int u_i^{\mathbb{R}^d}(\cdot,T) \sqrt{T}^{-d} \psi(x/\sqrt{T}) \,dx=0,
\end{align*}
we infer \eqref{SemigroupDecay2} (upon redefining the constant $\mathcal{C}(a,T)$).
\end{proof}

In the previous proofs, we have made use of the following elementary concentration estimate for sums of random variables with multilevel local dependence.
\begin{lemma}[\cite{FischerMultilevelLocalDependence}, Lemma~9]
\label{MultilevelVariableStretchedExponentialBound}
Consider a probability distribution of uniformly elliptic and bounded coefficient fields $a$ on $\mathbb{R}^d$ or a periodization of such a probability distribution, and suppose that assumptions (A1)-(A3) respectively (A1), (A2), (A3$_a$)-(A3$_c$) are satisfied. Let $X=X(a)$ be a random variable that is approximately a sum of random variables with multilevel local dependence in the sense of Definition~\ref{ConditionRandomVariable}. Then for $\tilde \gamma:=\gamma/(\gamma+1)$ the concentration estimate
\begin{align*}
||X-\mathbb{E}[X]||_{\exp^{\tilde \gamma}} \leq C(d,\gamma,K) B |\log L|^{d/2} L^{-d/2}
\end{align*}
holds true.
\end{lemma}

\section{Failure and Success of the Variance Reduction Approaches}
\label{SectionCounterexample}

We now establish our theorems on the failure and the success of the variance reduction approaches in stochastic homogenization. We start with the counterexample that shows that in general there is no guarantee that the variance reduction techniques provide an effective reduction of the variance, even for ``natural'' choices of the statistical quantity $\mathcal{F}(a)$ like the spatial average $\mathcal{F}_{avg}(a):=\dashint_{[0,L\varepsilon]^d} a \,dx$.

\begin{proof}[Proof of Theorem~\ref{TheoremFailureVarianceReduction}]
Before turning to the main result of Theorem~\ref{TheoremFailureVarianceReduction}, the failure of the spatial average $\mathcal{F}_{avg}(a)$ to explain a fraction of the variance of $a^\RVE$ (inequality \eqref{FailureVarianceReductionCov}), let us first show \eqref{FailureVarianceReduction}.
The estimate \eqref{FailureVarianceReduction} is in fact a consequence of the estimate \eqref{SQSErrorVariance} in the proof of Theorem~\ref{TheoremSQS} in combination with \eqref{DifferenceSystematicError} and the lower bound for the variance of $\mathcal{M}^\delta$ which is a straightforward consequence of the formula \eqref{LimitDistribution} and the definition of $\Varaideal=(1-|\rho|^2)\Var a_{ij}^\RVE$.

Note that the derivation of \eqref{FailureVarianceReduction2} from \eqref{FailureVarianceReductionCov} requires the estimate \eqref{FailureVarianceReduction} under the assumption (A2') instead of (A2). However, the only place where the assumption (A2) entered in our analysis is in Proposition~\ref{PropositionApproximabilityByMultilevel}, where it was used to apply the result of \cite{GloriaOttoNew} on the decay of the semigroup. However, the arguments of \cite{GloriaOttoNew} may be modified to yield the corresponding estimate under the assumption of discrete stationarity (A2').

\begin{figure}
\begin{tikzpicture}[scale=0.1]
\foreach \xa in {0,10,20,30,40}
\foreach \ya in {0,5,10,15,20,25,30,35,40}
{
\draw[fill,violet] (\xa+0,\ya+0) -- (\xa+0,\ya+5) -- (\xa+10,\ya+5) -- (\xa+10,\ya+0);
\foreach \y in {0,1,2,3,4}
  {
   \draw[fill,red] (5+\xa+0,\ya+\y+0) -- (5+\xa+0,\ya+\y+0.5) -- (5+\xa+5,\ya+\y+0.5)
                      -- (5+\xa+5,\ya+\y);
   \draw[fill,blue] (5+\xa+0,\ya+\y+0.5) -- (5+\xa+0,\ya+\y+1.0) -- (5+\xa+5,\ya+\y+1.0)
                       -- (5+\xa+5,\ya+\y+0.5);
  }
}
\end{tikzpicture}
\caption{\label{FigureMicro}A single tile with (second-order laminate) microstructure, as used in the proof of Theorem~\ref{TheoremFailureVarianceReduction}. Blue corresponds to the regions with $a(x)=\lambda \Id$, red to the regions with $a(x)=\Id$, and violet to the regions with $a(x)=\mu \Id$.}
\end{figure}
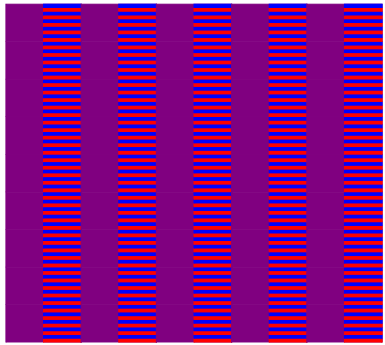
\begin{figure}
\includegraphics[scale=0.5]{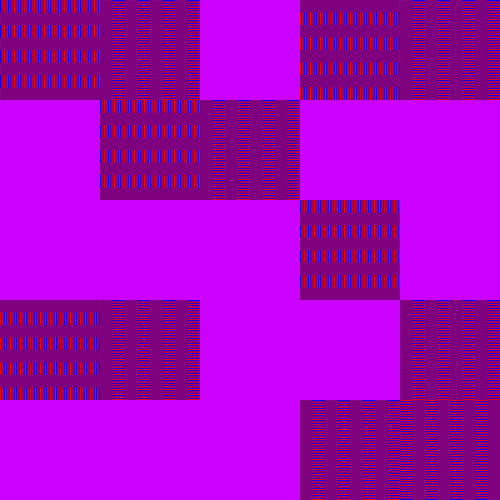}
~~~
\caption{A single realization of the probability distribution of our counterexample (with an exaggerated size of the microstructure in the tiles with microstructure). The tiles with microstructure behave almost like a homogeneous tile with an effective conductivity. Note that the tiles with microstructure are oriented randomly in order to enforce exact isotropy of the (co-)variances $\Var a^\RVE$ and $\Cov[a^\RVE,\dashint a \,dx]$.}
\end{figure}
Let us now turn to the construction of our counterexample featuring the degenerate covariance \eqref{FailureVarianceReductionCov}.
The construction is based on the following ideas:
\begin{itemize}
\item The approximation $a^\RVE$ for the effective coefficient depends in a uniformly continuous way on $a$ as a map $L^\infty([0,L\varepsilon]^d;\mathbb{R}^{d\times d})\rightarrow \mathbb{R}^{d\times d}$, as long as $a$ is uniformly elliptic and bounded.
\item Consider a probability distribution of coefficient fields $a$ for which $a$ is almost surely almost everywhere a multiple of the identity matrix. If in addition the law of $a$ is invariant under reflections of coordinate axes and invariant under exchange of coordinate axes (that is, invariant under diagonal reflections), the covariance
\begin{align*}
\Cov\left[a^\RVE,\dashint_{[0,L\varepsilon]^d} a \,dx\right]
\end{align*}
is a multiple of $\Id\otimes \Id$. For a proof of this fact, see Lemma~\ref{CovScalarLemma} below.
\item Consider the ``periodized random checkerboard'' with the set of tiles $\mathcal{T}:=\{x_0+[0,\varepsilon)^d: x_0\in \varepsilon\mathbb{Z}^d\cap [0,L\varepsilon)^d\}$. On each tile $T\in \mathcal{T}$, choose at random (and independently from the other tiles) $a(x)=\Id$ with probability $0.5$ and $a(x)=\frac{1}{2} \Id$ with probability $0.5$. By Proposition~\ref{PropositionLowerBoundOnCorrelation} and the preceding considerations, for this probability distribution the covariance
\begin{align*}
\Cov\left[a^\RVE,\dashint_{[0,L\varepsilon]^d} a \,dx\right]
\end{align*}
is a positive multiple of $\Id\otimes \Id$; in fact, one has a lower bound of the form $\gtrsim L^{-d}\Id \otimes \Id$.
\item We now consider a ``periodized random checkerboard with microstructure'' with the set of tiles $\mathcal{T}:=[0,\varepsilon)^d+ (\varepsilon\mathbb{Z}^d\cap [0,L\varepsilon)^d)$: Fix some $\tau\ll 1$ with $1/\tau\in 2\mathbb{N}$.  On each tile $T=\varepsilon k +[0,\varepsilon)^d \in \mathcal{T}$, choose at random (and independently from the other tiles) $a_\tau(x)=\sigma \Id$ with probability $0.5$ (where $\sigma>0$ is to be chosen below) and $a_\tau(x)=A_{\tau}((x-\varepsilon k)/\varepsilon)$ with probability $0.5$, where $A_{\tau}:[0,1]^2\rightarrow \mathbb{R}^{2\times 2}$ is the tile described in Figure~\ref{FigureMicro}, rotated and reflected at random (with equal probability for all $8$ orientations and independently on all such tiles).

The probability distribution of $a$ satisfies the same isotropy properties as in the case of the periodized random checkerboard.
Thus, by Lemma~\ref{CovScalarLemma} the covariance
\begin{align*}
\Cov\left[a_\tau^\RVE,\dashint_{[0,L\varepsilon]^d} a_\tau \,dx\right]
\end{align*}
is a multiple of $\Id\otimes \Id$.
\item We shall argue below that for suitable $\sigma,\lambda,\mu>0$ and for $\tau\ll 1$ small enough the covariance
\begin{align*}
\Cov\left[a_\tau^\RVE,\dashint_{[0,L\varepsilon]^d} a_\tau \,dx\right]
\end{align*}
is negative; in fact, one has an upper bound of the form $\lesssim -L^{-d} \Id\otimes \Id$.
\item Linearly interpolating between $a_\tau$ and $a$ -- that is, considering for $\kappa \in [0,1]$ the coefficient field
\begin{align*}
a_{\tau,\kappa}:=(1-\kappa) a + \kappa a_\tau
\end{align*}
defined on the product probability space, i.\,e.\ for independent $a_\tau$ and $a$ --  , we find a probability distribution of coefficient fields $\tilde a$ for which the covariance 
\begin{align*}
\Cov\left[{\tilde a}^\RVE,\dashint_{[0,L\varepsilon]^d} \tilde a \,dx\right]
\end{align*}
vanishes. This is possible by the continuous dependence of $a^\RVE$ and $\dashint_{[0,L\varepsilon]^d} a \,dx$ on $a$ (and hence the continuous dependence on $\kappa \in [0,1]$ in the case of the family $a_{\tau,\kappa}$) and by the fact that for all $\kappa \in [0,1]$ the covariance 
\begin{align*}
\Cov\left[a_{\tau,\kappa}^\RVE,\dashint_{[0,L\varepsilon]^d} a_{\tau,\kappa} \,dx\right]
\end{align*}
is a multiple of $\Id\otimes \Id$ (this latter property holds again by the isotropy properties of the probability distribution and Lemma~\ref{CovScalarLemma} below).
\item For any $\kappa \in (0,1)$ the variances $\Var \dashint_{[0,L\varepsilon]^d} a_{\tau,\kappa}\,dx$ and $\Var a^\RVE_{\tau,\kappa}$ are nondegenerate in the sense $\gtrsim L^{-d}\Id\otimes \Id$. For the spatial average $\dashint_{[0,L\varepsilon]^d} a_{\tau,\kappa}\,dx$ this non-degeneracy is an easy consequence of the formula
\begin{align*}
~~\quad\quad\Var \dashint_{[0,L\varepsilon]^d} a_{\tau,\kappa} \,dx
=(1-\kappa)^2\, \Var \dashint_{[0,L\varepsilon]^d} a \,dx
+\kappa^2\, \Var \dashint_{[0,L\varepsilon]^d} a_{\tau} \,dx
\end{align*}
(which follows from the definition of $a_{\tau,\kappa}$ and the independence of $a$ and $a_\tau$) and the fact that the latter two variances satisfy such a lower bound (note that the spatial average of the coefficient field on a tile with microstructure $A_{\tau}$ does not equal $\sigma \Id$). The non-degeneracy of $\Var a_{\tau,\kappa}^\RVE$ is shown as follows: First, a new coefficient field $a_{\tau,\kappa,\operatorname{eff}}$ is introduced by letting $a_{\tau,\kappa,\operatorname{eff}}=a_{\tau,\kappa}$ on each tile without microstructure but replacing the values of $a_{\tau,\kappa}$ by the effective coefficient from periodic homogenization on each tile with microstructure. Note that $a_{\tau,\kappa,\operatorname{eff}}$ corresponds to a standard random checkerboard. Denote by $a_{\tau,\kappa,\operatorname{eff}}^\RVE$ the approximation for the effective coefficient associated with the coefficient field $a_{\tau,\kappa,\operatorname{eff}}$ (i.\,e.\ the result of formula \eqref{PeriodicaRVE} for the coefficient field $a_{\tau,\kappa,\operatorname{eff}}$). The nondegeneracy of $\Var a_{\tau,\kappa}^\RVE$ now follows from the nondegeneracy $\Var a_{\tau,\kappa,\operatorname{eff},ii}^\RVE \gtrsim L^{-d}$ and the convergence $|a_{\tau,\kappa}^\RVE-a_{\tau,\kappa,\operatorname{eff}}^\RVE|\rightarrow 0$ for $\tau\rightarrow 0$ (uniformly in $\kappa$, see below).

Note that $a_{\tau,\kappa,\operatorname{eff}}^\RVE$ corresponds to a random checkerboard with tiles $(\kappa \sigma + (1-\kappa))\Id$, $\kappa \sigma + (1-\kappa)\cdot \frac{1}{2}\Id$, $\kappa A_{\tau} + (1-\kappa)\Id$, and $\kappa A_{\tau} + (1-\kappa) \cdot \frac{1}{2}\Id$, each tile chosen with probability $\frac{1}{4}$ (and the microscopic tiles rotated and reflected at random). Thus the nondegeneracy of $\Var a_{\tau,\kappa,\operatorname{eff},ii}^\RVE$ for $1\leq i\leq d$ follows from the covariance estimate of Proposition~\ref{PropositionLowerBoundOnCorrelation} and the quantitative upper bound $\Var \dashint_{[0,L\varepsilon]^d} a_{\tau,\kappa,\operatorname{eff}} \,dx\leq C L^{-d}$.
\end{itemize}

To complete the proof, it only remains to establish the negativity of the covariance
\begin{align*}
\Cov\left[a_\tau^\RVE,\dashint_{[0,L\varepsilon]^d} a_\tau \,dx\right]
\end{align*}
for $\tau\ll 1$ small enough and suitable $\sigma$, $\mu$, $\lambda$, as well as the convergence $a_{\tau,\kappa}^\RVE \rightarrow a_{\tau,\kappa,\operatorname{eff}}^\RVE$ for $\tau\rightarrow 0$, uniformly in $\kappa$. The underlying idea for our choice of the tiles in Figure~\ref{FigureMicro} is that we intend to exploit the nonlinear dependence of the effective coefficients in periodic homogenization on the coefficient field, equipping such a tile with an effective coefficient that is unrelated to the spatial average of the coefficient field. Heuristically, by classical results in periodic homogenization we expect the following to happen:
\begin{itemize}
\item Consider our (sub)pattern of periodic horizontal stripes of equal height (i.\,e.\ the red-and-blue subpattern in Figure~\ref{FigureMicro}), in which the coefficient field $a$ alternatingly takes the values $\Id$ and $\lambda \Id$. Then the (large-scale) effective coefficient for this pattern is given by
\begin{align*}
\begin{pmatrix}
\frac{1+\lambda}{2}&0\\0&\frac{2\lambda}{1+\lambda}
\end{pmatrix},
\end{align*}
that is by the arithmetic mean in the horizontal direction and by the harmonic mean in the vertical direction.
\item Consider now the pattern of periodic vertical stripes of equal width, in which the coefficient alternatingly takes the value $\mu \Id$ respectively is given by the pattern of horizontal stripes from the previous step. The effective coefficient for this (second-order laminate) pattern is (at least in the limit of an infinitesimally fine horizontal pattern) given by the arithmetic mean of the effective coefficients in the vertical direction and the harmonic mean of the effective coefficients in the horizontal direction, that is by
\begin{align*}
\begin{pmatrix}
\frac{2\mu(1+\lambda)}{2\mu+1+\lambda}&0\\0&\frac{\lambda}{1+\lambda}+\frac{\mu}{2}.
\end{pmatrix}
\end{align*}
Choosing $\mu:=\frac{3\lambda^2+(1-\lambda)\sqrt{9\lambda^2+14\lambda+9}+2\lambda+3}{4(\lambda+1)}$ -- which is positive for any $\lambda\in (0,1]$ -- , the effective coefficient becomes a multiple of the identity matrix.

Note that the spatial average of the coefficient field on a tile is given by
\begin{align*}
\frac{\mu+\frac{\lambda+1}{2}}{2} \Id.
\end{align*}
\item Consider the coefficient field $a_{\tau,\operatorname{eff}}$ that is obtained from our random checkerboard with microstructure $a_\tau$ by replacing $a_\tau$ on the tiles with microstructure with the effective coefficient $(\frac{\lambda}{1+\lambda}+\frac{\mu}{2})\Id$. The coefficient field $a_{\tau,\operatorname{eff}}$ is now just a usual random checkerboard; by Lemma~\ref{CovScalarLemma} and Proposition~\ref{PropositionLowerBoundOnCorrelation}, the covariance
\begin{align*}
\Cov\left[a_{\tau,\operatorname{eff}}^\RVE,\dashint_{[0,L\varepsilon]^d} a_{\tau,\operatorname{eff}} \,dx\right]
\end{align*}
is a positive multiple of $\Id\otimes \Id$, and we have a lower bound of the form $\geq cL^{-d} \Id \otimes \Id$ for the choice of $\lambda$, $\mu$, and $\tau$ to be made below. Note that $a_{\tau,\operatorname{eff}}$ -- and hence also the preceding covariance -- is actually independent of $\tau$ (we just keep the $\tau$ to emphasize that $a_{\tau,\operatorname{eff}}$ is the coefficient field obtained from $a_\tau$ in the homogenization limit $\tau\rightarrow 0$). We shall prove below that $a_\tau^\RVE$ is (quantitatively) close to $a_{\tau,\operatorname{eff}}^\RVE$ for $\tau\ll 1$ small enough, which implies that
\begin{align*}
\Cov\left[a_\tau^\RVE,\dashint_{[0,L\varepsilon]^d} a_{\tau,\operatorname{eff}} \,dx\right]
\end{align*}
is close to a positive multiple of $\Id\otimes \Id$ (again with a lower bound of the form $\geq c L^{-d} \Id \otimes \Id$).
\item The average $\dashint_{[0,L\varepsilon]^d} a_{\tau} \,dx$ is an affine function of  $\dashint_{[0,L\varepsilon]^d} a_{\tau,\operatorname{eff}} \,dx$:
The coefficient field $a_{\tau,\operatorname{eff}}$ is constant on each tile and may only take the values $\sigma \Id$ or $(\frac{\lambda}{1+\lambda}+\frac{\mu}{2})\Id$. On the tiles on which the value of $a_{\tau,\operatorname{eff}}$ is $\sigma \Id$, $a_\tau$ also takes the constant value $\sigma \Id$. However, on the tiles on which $a_{\tau,\operatorname{eff}}$ is given by $(\frac{\lambda}{1+\lambda}+\frac{\mu}{2})\Id$ (i.\,e.\ on the tiles on which $a_\tau$ features a microstructure), the average of $a_\tau$ is $\frac{2\mu+\lambda+1}{4}\Id$. We thus have
\begin{align*}
~~~~
\dashint_{[0,L\varepsilon]^d} a_{\tau} \,dx
&=\frac{\sharp\operatorname{microtiles}}{L^d} \cdot \frac{2\mu+\lambda+1}{4} \Id + \frac{L^d-\sharp\operatorname{microtiles}}{L^d} \cdot \sigma \Id
\\&
=\sigma \Id+\frac{\sharp\operatorname{microtiles}}{L^d} \cdot \bigg(\frac{2\mu+\lambda+1}{4}-\sigma\bigg) \Id
\end{align*}
and
\begin{align*}
~~~~~~~~~~~~
\dashint_{[0,L\varepsilon]^d} a_{\tau,\operatorname{eff}} \,dx
&=\frac{\sharp\operatorname{microtiles}}{L^d} \cdot \bigg(\frac{\lambda}{1+\lambda}+\frac{\mu}{2}\bigg) \Id + \frac{L^d-\sharp\operatorname{microtiles}}{L^d} \cdot \sigma \Id
\\&
=\sigma\Id +\frac{\sharp\operatorname{microtiles}}{L^d} \cdot \bigg(\frac{\lambda}{1+\lambda}+\frac{\mu}{2}-\sigma\bigg).
\end{align*}
Choosing $\sigma$ such that $\sigma>\frac{\lambda}{1+\lambda}+\frac{\mu}{2}$ but $\sigma<\frac{2\mu+\lambda+1}{4}$
-- which is possible for $\lambda>0$ small enough -- , we obtain a relation of the form
\begin{align*}
\dashint_{[0,L\varepsilon]^d} a_{\tau} \,dx=A\Id - B\dashint_{[0,L\varepsilon]^d} a_{\tau,\operatorname{eff}} \,dx
\end{align*}
for suitable positive constants $A$ and $B$.
Thus, the sign of the covariance flips upon replacing the $a_{\tau,\operatorname{eff}}$ by $a_\tau$ in the spatial average, i.\,e.\ 
\begin{align*}
\Cov\left[a_\tau^\RVE,\dashint_{[0,L\varepsilon]^d} a_{\tau} \,dx\right]
\end{align*}
must be a negative multiple of $\Id \otimes \Id$, with an upper bound of the form $\leq -cL^{-d} \Id\otimes \Id$. 
\end{itemize}
It now only remains to prove two things: We need to show that $a_\tau^\RVE$ is quantitatively close to $a_{\tau,\operatorname{eff}}$ if we choose the width $\tau$ of the vertical stripes and the height $\tau^2$ of the horizontal stripes in the pattern in Figure~\ref{FigureMicro} small enough and we need to establish the corresponding assertion for the interpolated coefficient field $a_{\tau,\kappa,\operatorname{eff}}$. As the latter result is shown similarly -- though with two different microscopic tiles $\kappa A_{\tau}+(1-\kappa)\frac{1}{2}\Id$ and $\kappa A_{\tau}+(1-\kappa)\frac{1}{2}\Id$, depending on whether the random checkerboard $a$ equals $\Id$ or $\frac{1}{2}\Id$ on the tile (and correspondingly, with two sets of homogenization correctors and two characteristic functions $\chi_{microtile1}$ and $\chi_{microtile2}$, see below for this notation) -- , we only provide the proof of the latter result.

For the remainder of the proof, we shall fix without loss of generality $\varepsilon:=1$ to avoid even more cumbersome notation. Again to avoid even more cumbersome notation, we only give the proof in the case that all tiles with microstructure have the same orientation as in Figure~\ref{FigureMicro}.

To see this quantitative closeness, we construct an approximate homogenization corrector $\phi_{i,\operatorname{appr}}$ for $a_\tau^\RVE$. To this aim, let $\phi_{i,\operatorname{eff}}$ be the homogenization corrector associated with the coefficient field $a_{\tau,\operatorname{eff}}$, that is let $\phi_{i,\operatorname{eff}}$ solve
\begin{align*}
-\nabla \cdot (a_{\tau,\operatorname{eff}}(e_i+\nabla \phi_{i,\operatorname{eff}}))=0
\end{align*}
on $[0,L]^2$ with periodic boundary conditions. We now intend to build the approximate homogenization corrector $\phi_{i,\operatorname{appr}}$ for $a_\tau^\RVE$ by a nested two-scale expansion, using the homogenization correctors for the periodic laminate microstructures.

By Meyer's estimate, there exists $p>2$ with
\begin{align}
\label{HigherIntegrabilityPhi}
\dashint_{[0,L]^2} |\nabla \phi_{i,\operatorname{eff}}|^p \,dx \leq C(d,\lambda).
\end{align}
Furthermore, $a_{\tau,\operatorname{eff}}$ is constant on each tile $k + [0,1)^2$, which implies on each tile $T=k + [0,1)^2$ (with $k\in \mathbb{Z}^2$) for each $x\in T$ by regularity theory for constant coefficient equations 
\begin{align}
\nonumber
|\nabla^2 \phi_{i,\operatorname{eff}}(x)|
&\leq \frac{C}{\dist(x,\partial T)} \left(\dashint_{\{|y-x|\leq \dist(x,\partial T)/2\}} |e_i+\nabla \phi_{i,\operatorname{eff}}|^2 \,dy\right)^{1/2}
\\&
\label{RegularityBound}
\leq \frac{C}{\dist(x,\partial T)^{1+d/2}} \left(\dashint_{T} |e_i+\nabla \phi_{i,\operatorname{eff}}|^2 \,dy\right)^{1/2}.
\end{align}
Let $\rho_\delta$ denote a standard mollifier. The $L^p$ estimate and the estimate on $\nabla^2 \phi_{i,\operatorname{eff}}$ imply (for notational convenience we extend $\phi_{i,\operatorname{eff}}$ by periodicity)
\begin{align}
\label{MollificationPhiEffEstimate}
\dashint_{[0,L]^2} \big|\nabla \phi_{i,\operatorname{eff}}-\nabla (\rho_{\delta} \ast \phi_{i,\operatorname{eff}})\big|^{(p+2)/2} \,dx \leq C \delta^\alpha
\end{align}
for some $\alpha>0$ (for a proof of this estimate, split the domain into a
neighborhood of size $\delta^{1/5}$ of the tile boundaries $\partial T$, on which one uses the H\"older inequality and the $L^p$ bound on $\nabla \phi_{i,\operatorname{eff}}$ in \eqref{HigherIntegrabilityPhi}, and the interior $\{x\in T:\dist(x,\partial T)\geq \delta^{1/5}\}$, where one applies the regularity estimate \eqref{RegularityBound}).

Let $\phi_{i,h}$ denote the $2$-periodic homogenization corrector for the coefficient field $a_h(x,y)$ associated with the pattern of horizontal stripes in Figure~\ref{FigureMicro} (i.\,e.\ let $a_h(x,y)=a_h(y)$ take alternatingly on intervals of length $1$ the values $\Id$ and $\lambda\Id$). Note that $\phi_{1,h}\equiv 0$ and that $\phi_{2,h}$ is explicitly given by
\begin{align*}
\phi_{2,h}(x,y)= \frac{1}{\dashint_0^{2} \frac{1}{e_2\cdot a_h(x,\tilde y)e_2} \,d\tilde y} \int_0^y \frac{1}{e_2\cdot a_h(x,\tilde y)e_2} \,d\tilde y-y.
\end{align*}
We shall frequently use the uniform bound on the gradient $|\nabla \phi_{i,h}|\leq C$ derived easily from this formula.

Let $\phi_{i,v}$ denote the $2$-periodic homogenization corrector associated with the pattern of vertical stripes of width $1$, in which the coefficient field $a_v(x,y)=a_v(x)$ alternatingly takes the values $\mu \Id$ and
\begin{align*}
\begin{pmatrix}
\frac{1+\lambda}{2}&0\\0&\frac{2\lambda}{1+\lambda}
\end{pmatrix}.
\end{align*}
Note that we have $\phi_{2,v}\equiv 0$ and that $\phi_{1,v}$ is given explicitly by
\begin{align*}
\phi_{1,v}(x,y)=\frac{1}{\dashint_0^{2} \frac{1}{e_1\cdot a_v(\tilde x,y)e_1} \,d\tilde x} \int_0^x \frac{1}{e_1 \cdot a_v(\tilde x,y)e_1 } \,d\tilde x - x.
\end{align*}
We shall again frequently use the uniform bound on the gradient $|\nabla \phi_{i,v}|\leq C$.

We define the vector potential for the flux correction $\sigma_{h,ijk}$, skew-symmetric in its last two indices, as $\sigma_{h,212}:=0$ and
\begin{align}
\label{SigmaH}
\sigma_{h,112} := \int_0^y (a_h(\tilde y)-a_{h,\operatorname{eff}})e_1 \cdot e_1 \,d\tilde y.
\end{align}
Note that with this definition $\sigma_{h,ijk}$ satisfies $\nabla \cdot \sigma_{h,i} = a_h (e_i+\nabla \phi_{i,h})-a_{h,\operatorname{eff}}e_i$, as one checks by a case-by-case analysis.

Similarly, we define $\sigma_{v,ijk}$, skew-symmetric in its last two indices, as $\sigma_{v,121}:=0$ and
\begin{align}
\label{SigmaV}
\sigma_{v,221} := \int_0^x (a_v(\tilde x)-a_{v,\operatorname{eff}})e_2 \cdot e_2 \,d\tilde x
\end{align}
which then satisfies $\nabla \cdot \sigma_{v,i} = a_v (e_i+\nabla \phi_{i,v})-a_{v,\operatorname{eff}}e_i$.

Let us denote the indicator function of the tiles with microstructure by $\chi_{microtile}$ (i.\,e.\ $\chi_{microtile}$ is $1$ on all tiles $k+[0,1)^d\subset [0,L)^d$ with microstructure and $0$ on the other tiles). Similarly, we denote by $\chi_{vmicrostripe}$ the indicator functions of all vertical stripes that according to Figure~\ref{FigureMicro} contain a micropattern of horizontal stripes.
We then build our approximate correctors as
\begin{align*}
\phi_{i,\operatorname{appr},1}:=
\rho_{\delta_0} \ast \phi_{i,\operatorname{eff}}
+(\rho_{\tau \delta_1} \ast \chi_{microtile}) \sum_j (\delta_{ij}+\partial_j (\rho_{\delta_0} \ast \phi_{i,\operatorname{eff}})) \Big( \rho_{\delta_1 \tau} \ast \tau \phi_{j,v}\Big(\frac{\cdot}{\tau}\Big)\Big)
\end{align*}
and
\begin{align*}
\phi_{i,\operatorname{appr},2}:=
\phi_{i,\operatorname{appr},1}+
(\rho_{\tau^2 \delta_2} \ast \chi_{vmicrostripe}) \sum_k (\partial_k \phi_{i,\operatorname{appr},1}
+\delta_{ik}) \tau^2 \phi_{k,h}\Big(\frac{\cdot}{\tau^2}\Big).
\end{align*}
We observe that $\phi_{i,\operatorname{appr},1}$ satisfies the estimate
\begin{align}
\label{Phi1UpperBound}
|\nabla \phi_{i,\operatorname{appr},1}|
\leq \bigg(\frac{C}{\min\{1,\delta_1\}}+\frac{C\tau}{\delta_0}\bigg) (\rho_{2\delta_0} \ast |\nabla \phi_{i,\operatorname{eff}}|+1).
\end{align}
We also have the bound
\begin{align}
\label{Phi2UpperBound}
|\nabla \phi_{i,\operatorname{appr},2}|
&\leq \frac{C}{\min\{1,\delta_2\}} (|\nabla \phi_{i,\operatorname{appr},1}|+1)
\\&
\nonumber
+\bigg(C+\frac{C\tau^2}{\delta_0}+\frac{C\tau}{\delta_1^2}+\frac{C\tau^3}{\delta_0^2}+\frac{C\tau}{\delta_1}\bigg)
(\rho_{2\delta_0} \ast |\nabla \phi_{i,\operatorname{eff}}|+1).
\end{align}
Furthermore, if we are at least $\tau \delta_1$ away from the tile boundaries and the boundaries of the vertical stripes (note that $\rho_{\delta_1 \tau} \ast \nabla \phi_{j,v}(\cdot/\tau)$ is then equal to $\nabla \phi_{j,v}(\cdot/\tau)$ as the latter quantity is constant in each stripe; note also that then $\rho_{\tau\delta_1} \ast \chi_{microtile}$ is locally constant $=0$ or $=1$ and that we have a uniform bound on $\nabla \phi_{j,v}$), we have by \eqref{RegularityBound} on each tile $T=k + [0,1)^2$, $k\in \mathbb{Z}^d\cap [0,L)^d$,
\begin{align}
\label{TwoScaleExpansionHorizontal}
&\bigg|e_i+\nabla \phi_{i,\operatorname{appr},1} - \sum_j (e_j+\chi_{microtile} \nabla \phi_{j,v}(\cdot/\tau)) (\delta_{ij}+\partial_j \phi_{i,\operatorname{eff}}) \bigg|
\\&
\nonumber
\leq \frac{C}{\dist(\cdot,\partial T)^2} \bigg(\dashint_T |e_i+\nabla \phi_{i,\operatorname{eff}}|^2 \,dx\bigg)^{1/2} \big(\delta_0 +\tau\big).
\end{align}
If we are at least $\tau \delta_1$ away from the tile boundaries and the boundaries of the vertical stripes and at least $\tau^2 \delta_2$ away from the boundary of the horizontal stripes, we get (note that $\rho_{\delta_2 \tau^2} \ast \nabla \phi_{k,h}(\cdot/\tau^2)$ is then equal to $\nabla \phi_{k,h}(\cdot/\tau^2)$ as the latter quantity is constant in each small horizontal stripe; note also that then $\rho_{\tau^2 \delta_2} \ast \chi_{hmicrostripe}$ is locally constant $=0$ or $=1$ and that we have a uniform bound on $\nabla \phi_{k,h}$)
\begin{align*}
&\bigg|e_i+\nabla \phi_{i,\operatorname{appr},2}
\\&~~~
- \sum_{k} (e_k+\chi_{vmicrostripe}\nabla \phi_{k,h}(\cdot/\tau^2)) \sum_j \big(\delta_{jk}+\chi_{microtile}\partial_k \phi_{j,v}(\cdot/\tau)\big) (\delta_{ij}+\partial_j \phi_{i,\operatorname{eff}}) \bigg|
\\&
\stackrel{\eqref{TwoScaleExpansionHorizontal}}{\leq} C \bigg|e_i+\nabla \phi_{i,\operatorname{appr},1} - \sum_j (e_j+\chi_{microtile} \nabla \phi_{j,v}(\cdot/\tau)) (\delta_{ij}+\partial_j \phi_{i,\operatorname{eff}}) \bigg|
\\&~~~~~
+C\tau^2 |\nabla^2 \phi_{i,\operatorname{appr},1}|
+\frac{C}{\dist(\cdot,\partial T)^2} \bigg(\dashint_T |e_i+\nabla \phi_{i,\operatorname{eff}}|^2 \,dx\bigg)^{1/2} \big(\delta_0 +\tau\big)
\\&
\stackrel{\eqref{TwoScaleExpansionHorizontal},\eqref{RegularityBound}}{\leq}
\frac{C}{\dist(\cdot,\partial T)^2} \bigg(\dashint_T |e_i+\nabla \phi_{i,\operatorname{eff}}|^2 \,dx\bigg)^{1/2} \Big(\delta_0 + \tau + \tau^2 + \frac{\tau^3}{\delta_0}\Big).
\end{align*}
Using the fact that by Meyers inequality we have for some $p=p(\lambda)>2$
\begin{align*}
\dashint_{[0,L]^2} |e_i+\nabla \phi_{i,\operatorname{eff}}|^p \,dx \leq C(d,\lambda),
\end{align*}
we obtain by choosing $\delta_0$, $\delta_1$, and $\delta_2$ as appropriate powers of $\tau$ and using \eqref{Phi2UpperBound}
\begin{align}
\nonumber
&\dashint_{[0,L\varepsilon]^d} \bigg|e_i+\nabla \phi_{i,\operatorname{appr},2}
- \sum_{k} (e_k+\chi_{vmicrostripe}\nabla \phi_{k,h}(\cdot/\tau^2))
\\&~~~~~~~~~~~~~~~~~~~~~~~~~~~~~~~~~~~~\times
\nonumber
\sum_j \big(\delta_{jk}+\chi_{microtile}\partial_k \phi_{j,v}(\cdot/\tau)\big) (\delta_{ij}+\partial_j \phi_{i,\operatorname{eff}}) \bigg|^2 \,dx
\\&~~~~~~~~~~~~~~
\leq C(d,\lambda) \tau^\eta
\label{ErrorMultipleTwoScale}
\end{align}
for some $\eta>0$.

Having bounded the error in the gradient, we next estimate the error in the flux. In an analogous fashion to the definition of $a_{\tau,\operatorname{eff}}$ as the effective coefficient from periodic homogenization on each tile, we define $a_{\tau,\operatorname{veff}}$ as equal to $a_{\tau,\operatorname{eff}}=a_\tau$ on the tiles without microstructure and equal to the effective coefficient from periodic homogenization on each vertical stripe of width $\tau$ on each tile with microstructure.
Recalling the definitions \eqref{SigmaH} and \eqref{SigmaV}, we may rewrite the error in the flux in a pointwise way as
\begin{align}
\nonumber
&a_\tau \sum_{k} (e_k+\chi_{vmicrostripe}\nabla \phi_{k,h}(\cdot/\tau^2)) \sum_j \big(\delta_{jk}+\chi_{microtile}\partial_k \phi_{j,v}(\cdot/\tau)\big) (\delta_{ij}+\partial_j \phi_{i,\operatorname{eff}})
\\&
\nonumber
-a_{\tau,\operatorname{eff}} (e_i+\nabla \phi_{i,\operatorname{eff}})
\\&
\nonumber
=\sum_j \Big(a_\tau \sum_{k} (e_k+\chi_{vmicrostripe}\nabla \phi_{k,h}(\cdot/\tau^2)) - a_{\tau,\operatorname{veff}} e_k\Big)
\\&~~~~~~~~~~~\times
\nonumber
\big(\delta_{jk}+\chi_{microtile}\partial_k \phi_{j,v}(\cdot/\tau)\big) (\delta_{ij}+\partial_j \phi_{i,\operatorname{eff}})
\\&~~~
\nonumber
+\sum_j \Big(a_{\tau,\operatorname{veff}}\big(e_j+\chi_{microtile}\nabla \phi_{j,v}(\cdot/\tau)\big)-a_{\tau,\operatorname{eff}}e_j\Big) (\delta_{ij}+\partial_j \phi_{i,\operatorname{eff}})
\\&
\label{EquationFluxRewritten}
=\chi_{vmicrostripe} \sum_k (\nabla \cdot (\tau^2\sigma_{h,k}(\cdot/\tau^2))) \sum_j \big(\delta_{jk}+\chi_{microtile}\partial_k \phi_{j,v}(\cdot/\tau)\big) (\delta_{ij}+\partial_j \phi_{i,\operatorname{eff}})
\\&~~~
\nonumber
+\chi_{microtile} \sum_j (\nabla \cdot (\tau \sigma_{v,j}(\cdot/\tau))) (\delta_{ij}+\partial_j \phi_{i,\operatorname{eff}}).
\end{align}
Thus, having choosen $\delta_0$, $\delta_1$, and $\delta_2$ as suitable powers of $\tau$, we obtain by \eqref{ErrorMultipleTwoScale}, \eqref{RegularityBound}, and \eqref{HigherIntegrabilityPhi}
\begin{align*}
&\bigg|\dashint_{[0,L]^2} a_\tau (e_i+\nabla \phi_{i,\operatorname{appr},2}) \,dx
-\dashint_{[0,L]^2} a_{\tau,\operatorname{eff}} (e_i+\nabla \phi_{i,\operatorname{eff}}) \,dx\bigg|
\\&
\leq C(d,\lambda) \tau^\eta.
\end{align*}
It now only remains to show that $\nabla \phi_{i,\operatorname{appr},2}$ is a good approximation for $\nabla \phi_i$. To do so, we consider the difference $\phi_i-\phi_{i,\operatorname{appr},2}$ and observe that it satisfies the PDE
\begin{align*}
&-\nabla \cdot (a_\tau (\nabla \phi_i - \nabla \phi_{i,\operatorname{appr},2}))
\\&
=\nabla \cdot (a_\tau (e_i+\nabla \phi_{i,\operatorname{appr},2}))
\\&
=\nabla \cdot (a_\tau (e_i+\nabla \phi_{i,\operatorname{appr},2})
-a_{\tau,\operatorname{eff}}(e_i+\nabla \phi_{i,\operatorname{eff}})).
\end{align*}
We now replace the divergence-form right-hand side using \eqref{ErrorMultipleTwoScale}
\begin{align*}
-\nabla \cdot (a_\tau (\nabla &\phi_i - \nabla \phi_{i,\operatorname{appr},2})
=\nabla \cdot g
\\&~~~
+ \nabla \cdot \bigg(a_\tau \sum_{k} (e_k+\chi_{vmicrostripe}\nabla \phi_{k,h}(\cdot/\tau^2))
\\&~~~~~~~~~~~~~~~~~~~~~~~~~\times
\sum_j \big(\delta_{jk}+\chi_{microtile}\partial_k \phi_{j,v}(\cdot/\tau)\big) (\delta_{ij}+\partial_j \phi_{i,\operatorname{eff}})
\\&~~~~~~~~~~~~~~~~
-a_{\tau,\operatorname{eff}} (e_i+\nabla \phi_{i,\operatorname{eff}})\bigg)
\end{align*}
for some $g$ with $\dashint_{[0,L]^2} |g|^2 \leq C \tau^\eta$ (recall that $\delta_1$ and $\delta_2$ have been chosen as a suitable small powers of $\tau$ and recall also the uniform $L^p$ bound for $\nabla \phi_{i,\operatorname{eff}}$ in \eqref{HigherIntegrabilityPhi}).
This expression in turn may be rewritten by \eqref{MollificationPhiEffEstimate} and \eqref{EquationFluxRewritten} for any $\beta>0$ small enough as
\begin{align*}
&-\nabla \cdot (a_\tau (\nabla \phi_i - \nabla \phi_{i,\operatorname{appr},2}))
\\&
=\nabla \cdot \tilde g
+ \nabla \cdot \bigg((\rho_{\tau^{1+\beta}}\ast \chi_{vmicrostripe}) \sum_k (\nabla \cdot (\tau^2\sigma_{h,k}(\cdot/\tau^2)))
\\&~~~~~~~~~~~~~~~~~~~~~~~~~~~
\times\sum_j \big(\delta_{jk}+\rho_{\tau^{1+\beta}}\ast \chi_{microtile}\partial_k \phi_{j,v}(\cdot/\tau)\big) (\delta_{ij}+\rho_{\tau^\beta}\ast \partial_j \phi_{i,\operatorname{eff}})
\\&~~~~~~~~~~~~~~~~~~~~~
+(\rho_{\tau^{\beta}}\ast \chi_{microtile}) \sum_j (\nabla \cdot (\tau \sigma_{v,j}(\cdot/\tau))) (\delta_{ij}+\rho_{\tau^\beta}\ast \partial_j \phi_{i,\operatorname{eff}})\bigg).
\end{align*}
for some $\tilde g$ with $\dashint_{[0,L]^2} |\tilde g|^2 \leq C \tau^{delta(\beta)}$. 

Using the skew-symmetry of $\sigma_{v,i}$ and $\sigma_{h,i}$, we obtain
\begin{align*}
&-\nabla \cdot (a_\tau (\nabla \phi_i - \nabla \phi_{i,\operatorname{appr},2}))
\\&
=\nabla \cdot \tilde g
\\&~~
+ \sum_k (\nabla \cdot (\tau^2\sigma_{h,k}(\cdot/\tau^2)))
\\&~~~~
\cdot \nabla \sum_j (\rho_{\tau^{1+\beta}}\ast \chi_{vmicrostripe}) \big(\delta_{jk}+\rho_{\tau^{1+\beta}}\ast \chi_{microtile}\partial_k \phi_{j,v}(\cdot/\tau)\big) (\delta_{ij}+\rho_{\tau^\beta}\ast \partial_j \phi_{i,\operatorname{eff}})
\\&~~
+\sum_j (\nabla \cdot (\tau \sigma_{v,j}(\cdot/\tau))) \cdot \nabla \big((\rho_{\tau^{\beta}}\ast \chi_{microtile}) (\delta_{ij}+\rho_{\tau^\beta}\ast \partial_j \phi_{i,\operatorname{eff}})\big).
\end{align*}
Using again the skew-symmetry of $\sigma_{v,i}$ and $\sigma_{h,i}$, we get
\begin{align*}
&-\nabla \cdot (a_\tau (\nabla \phi_i - \nabla \phi_{i,\operatorname{appr},2}))
\\&
=\nabla \cdot \tilde g
\\&~~~
-\nabla \cdot \bigg( \sum_k  \tau^2\sigma_{h,k}(\cdot/\tau^2)
\cdot \nabla \sum_j (\rho_{\tau^{1+\beta}}\ast \chi_{vmicrostripe})
\\&~~~~~~~~~~~~~~~~~~~~~~~~~~
\times \big(\delta_{jk}+\rho_{\tau^{1+\beta}}\ast\chi_{microtile}\partial_k \phi_{j,v}(\cdot/\tau)\big) (\delta_{ij}+\rho_{\tau^{\beta}}\ast\partial_j \phi_{i,\operatorname{eff}})\bigg)
\\&~~~
-\nabla \cdot \bigg(\sum_j \tau \sigma_{v,j}(\cdot/\tau) \nabla \big((\rho_{\tau^{\beta}}\ast \chi_{microtile}) (\delta_{ij}+\rho_{\tau^{\beta}}\ast\partial_j \phi_{i,\operatorname{eff}})\big)\bigg).
\end{align*}
Choosing $\beta>0$ small enough, we finally end up with
\begin{align*}
-\nabla \cdot (a_\tau (\nabla \phi_i - \nabla \phi_{i,\operatorname{appr},2}))
= \nabla \cdot \hat g
\end{align*}
with $\dashint_{[0,L]^d} |\hat g|^2 \leq C \tau^{\tilde \nu}$ for some $\tilde \nu>0$. A standard energy estimate now implies
\begin{align*}
\dashint_{[0,L)^d} |\nabla \phi_i-\nabla \phi_{i,\operatorname{appr},2}|^2 \,dx
\leq C\tau^{\tilde \nu}.
\end{align*}
\end{proof}

\begin{lemma}
\label{CovScalarLemma}
Consider a probability distribution of coefficient fields $a$ subject to the conditions (A1), (A2), and (A3$_a$)-(A3$_c$). Suppose in addition that $a$ is almost surely almost everywhere a multiple of the identity matrix. If in addition the law of $a$ is invariant under reflections of coordinate axes (i.\,e.\ maps of the form $x\mapsto (x_1,\ldots,-x_i,\ldots,x_d)$) and invariant under exchange of coordinate axes (i.\,e.\ maps of the form $x\mapsto (x_1,\ldots,x_{i-1},x_j,x_{i+1},\ldots,x_{j-1},x_i,x_{j+1},\ldots,x_d)$), the covariance
\begin{align*}
\Cov\left[a^\RVE,\dashint_{[0,L\varepsilon]^d} a \,dx\right]
\end{align*}
is a multiple of $\Id\otimes \Id$.
\end{lemma}
\begin{proof}
For such a probability distribution of coefficient fields $a$, the spatial average $\dashint_{[0,L\varepsilon]^d} a \,dx$ is almost surely a multiple of the identity matrix, which entails that
\begin{align*}
\Cov\left[a^\RVE,\dashint_{[0,L\varepsilon]^d} a \,dx\right]
=B \otimes \Id
\end{align*}
for some $B\in \mathbb{R}^{d\times d}$.

The matrix $B$ must also be a multiple of the identity matrix: Under reflection of the $i$-th coordinate, by the corrector equation \eqref{EquationCorrector} and the fact that $a$ is pointwise a multiple of the identity matrix we have that the $i$-th corrector for the reflected coefficient field $\hat a(x)=a(x_1,\ldots,-x_i,\ldots,x_d)$ is given by $\hat \phi_i(x)=-\phi_i(x_1,\ldots,-x_i,\ldots,x_d)$. Thus, the off-diagonal entries of $a^\RVE$ which are given by (for $i\neq j$, using also that $a(x)=a_{scalar}(x) \Id$)
\begin{align*}
a^\RVE e_i \cdot e_j = \dashint_{[0,L\varepsilon]^d} a(e_i+\nabla \phi_i)\cdot e_j \,dx
= \dashint_{[0,L\varepsilon]^d} a_{scalar}(x) (e_j\cdot \nabla)\phi_i(x) \,dx
\end{align*}
switch sign under such reflections, while the average $\dashint_{[0,L\varepsilon]^d} a \,dx$ remains invariant. As our probability distribution is invariant under reflections, the off-diagonal entries of $B$ must be zero. Similarly, as our probability distribution is invariant under exchange of coordinates, all diagonal entries of $B$ must coincide; therefore the covariance must be a multiple of $\Id\otimes \Id$.
\end{proof}

We now turn to the proof of our theorem on successful variance reduction for random coefficient fields that are obtained by applying a ``monotone'' functions to a collection of iid random variables.
\begin{proof}[Proof of Proposition~\ref{PropositionLowerBoundOnCorrelation}]
Without loss of generality (by rescaling), we may consider the case $\varepsilon=1$.

Given any $\xi \in \mathbb{R}^d$, the $L$-periodic correctors associated with two $L$-periodic coefficient fields $a$ and $\tilde a$ are given as the solutions to the PDEs
\begin{align}
\label{SolutionStartCoefficient}
-\nabla \cdot (a\nabla \phi^{L,a}_\xi) = \nabla \cdot (a\xi)
\end{align}
and
\begin{align*}
-\nabla \cdot (\tilde a\nabla \phi^{L,\tilde a}_\xi) = \nabla \cdot (\tilde a\xi).
\end{align*}
Define $\phi^{L,(1-s)a+s\tilde a}_\xi$ as the $L$-periodic solution to
\begin{align}
\label{SolutionInterpolatedCoefficient}
-\nabla \cdot (((1-s)a+s\tilde a)\nabla \phi^{L,(1-s)a+s\tilde a}_\xi) = \nabla \cdot (((1-s)a+s\tilde a)\xi).
\end{align}
Setting
\begin{align*}
a^{\RVE,s}\xi \cdot \xi := \dashint_{[0,L]^d} ((1-s)a+s\tilde a)(\xi+\nabla \phi^{L,(1-s)a+s\tilde a}_\xi) \cdot \xi \,dx
\end{align*}
we then obtain
\begin{align*}
&\frac{d}{ds} a^{\RVE,s}\xi \cdot \xi
\\&
=
\frac{d}{ds}
\dashint_{\domain} ((1-s)a+s\tilde a)(\xi+\nabla \phi_\xi^{L,(1-s)a+s\tilde a})\cdot \xi \,dx
\\&
\stackrel{\eqref{SolutionInterpolatedCoefficient}}{=}
\frac{d}{ds}
\dashint_{\domain} ((1-s)a+s\tilde a)(\xi+\nabla \phi_\xi^{L,(1-s)a+s\tilde a})\cdot (\xi+\nabla \phi_\xi^{L,(1-s)a+s\tilde a}) \,dx
\\&
=
\dashint_{\domain} (\tilde a-a)(\xi+\nabla \phi^{L,(1-s)a+s\tilde a}_\xi) \cdot (\xi+\nabla \phi^{L,(1-s)a+s\tilde a}_\xi) \,dx
\\&~~~
+2\dashint_{\domain} ((1-s)a+s\tilde a)\nabla \frac{d}{ds} \phi^{L,(1-s)a+s\tilde a}_\xi \cdot (\xi+\nabla \phi^{L,(1-s)a+s\tilde a}_\xi) \,dx
\\&
\stackrel{\eqref{SolutionInterpolatedCoefficient}}{=}\dashint_{\domain} (\tilde a-a)(\xi+\nabla \phi^{L,(1-s)a+s\tilde a}_\xi) \cdot (\xi+\nabla \phi^{L,(1-s)a+s\tilde a}_\xi) \,dx.
\end{align*}
Given two coefficient fields $a$ and $\tilde a$ with $a-\tilde a\geq 0$, we therefore have the estimate
\begin{align}
\label{LowerBoundExchangea}
&a^{\RVE,a}\xi \cdot \xi
-a^{\RVE,\tilde a}\xi \cdot \xi
\\&
\nonumber
\geq
\int_0^1 \dashint_{\domain} (a-\tilde a)(\xi+\nabla \phi^{L,(1-s)a+s\tilde a}_\xi) \cdot (\xi+\nabla \phi^{L,(1-s)a+s\tilde a}_\xi) \,dx \,ds.
\end{align}
We now would like to derive a lower bound for the term on the right-hand side. We have by \eqref{SolutionStartCoefficient} and \eqref{SolutionInterpolatedCoefficient}
\begin{align*}
-\nabla \cdot \big(((1-s)a+s\tilde a)(\nabla \phi^{L,(1-s)a+s\tilde a}_\xi-\nabla \phi^{L,a}_\xi)\big) = \nabla \cdot (s(\tilde a-a)(\xi+\nabla \phi^{L,a}_\xi)).
\end{align*}
Testing this PDE by the solution (note that $(1-s)a+s\tilde a$ is $\lambda$-uniformly elliptic) yields
\begin{align*}
&\dashint_{\domain} \lambda |\nabla \phi^{L,a}_\xi-\nabla \phi^{L,(1-s)a+s\tilde a}_\xi|^2 \,dx
\\&
\leq s \dashint_{\domain} (\tilde a-a)(\xi+\nabla \phi_\xi^{L,a})\cdot \big(\nabla \phi^{L,a}_\xi-\nabla \phi^{L,(1-s)a+s\tilde a}_\xi \big) \,dx
\end{align*}
and therefore by Young's inequality (note that the matrix $a-\tilde a$ is symmetric and by (A1) bounded by $\frac{1}{\lambda}$ in the natural matrix norm)
\begin{align*}
\dashint_{\domain} |\nabla \phi^{L,a}_\xi-\nabla \phi^{L,(1-s)a+s\tilde a}_\xi|^2 \,dx
\leq \frac{s^2}{\lambda^4} \dashint_{\domain} (a-\tilde a)(\xi+\nabla \phi_\xi^{L,a})\cdot (\xi+\nabla \phi_\xi^{L,a}) \,dx.
\end{align*}
In particular, we obtain by \eqref{LowerBoundExchangea} (and the analogous version of the previous estimate for $\phi_\xi^{L,\tilde a}$ instead of $\phi_\xi^{L,a}$) and $a\geq \tilde a$
\begin{align*}
&a^{\RVE,a}\xi \cdot \xi
-a^{\RVE,\tilde a}\xi \cdot \xi
\\&
\geq
\int_0^{\lambda^2/2} \frac{1}{2} \dashint_{\domain} (a-\tilde a)(\xi+\nabla \phi^{L,a}_\xi) \cdot (\xi+\nabla \phi^{L,a}_\xi) \,dx
\\&~~~~~~~~~~~
- 2\dashint_{\domain} (a-\tilde a)(\nabla \phi^{L,a}_\xi-\nabla \phi^{L,(1-s)a+s\tilde a}_\xi) \cdot (\nabla \phi^{L,a}_\xi-\nabla \phi^{L,(1-s)a+s\tilde a}_\xi) \,dx  \,ds
\\&~~~
+\int_{1-\lambda^2/2}^1 \frac{1}{2} \dashint_{\domain} (a-\tilde a)(\xi+\nabla \phi^{L,\tilde a}_\xi) \cdot (\xi+\nabla \phi^{L,\tilde a}_\xi) \,dx
\\&~~~~~~~~~~~
- 2\dashint_{\domain} (a-\tilde a)(\nabla \phi^{L,\tilde a}_\xi-\nabla \phi^{L,(1-s)a+s\tilde a}_\xi) \cdot (\nabla \phi^{L,\tilde a}_\xi-\nabla \phi^{L,(1-s)a+s\tilde a}_\xi) \,dx  \,ds
\\&
\geq
\int_0^{\lambda^2/2} \left(\frac{1}{2}-\frac{2s^2}{\lambda^4}\right) \dashint_{\domain} (a-\tilde a)(\xi+\nabla \phi^{L,a}_\xi) \cdot (\xi+\nabla \phi^{L,a}_\xi) \,dx \,ds
\\&~~~
+\int_{1-\lambda^2/2}^1 \left(\frac{1}{2}-\frac{2(1-s)^2}{\lambda^4}\right) \dashint_{\domain} (a-\tilde a)(\xi+\nabla \phi^{L,\tilde a}_\xi) \cdot (\xi+\nabla \phi^{L,\tilde a}_\xi) \,dx \,ds
\\&
\geq
\frac{\lambda^2}{8} \dashint_{\domain} (a-\tilde a)(\xi+\nabla \phi^{L,a}_\xi) \cdot (\xi+\nabla \phi^{L,a}_\xi) \,dx
\\&~~~
+\frac{\lambda^2}{8} \dashint_{\domain} (a-\tilde a)(\xi+\nabla \phi^{L,\tilde a}_\xi) \cdot (\xi+\nabla \phi^{L,a}_\xi) \,dx.
\end{align*}
This entails
\begin{align*}
&\left(a^{\RVE,a}\xi \cdot \xi
-a^{\RVE,\tilde a}\xi \cdot \xi\right)
\left(\dashint_\domain a \xi \cdot \xi \,dx
-\dashint_\domain \tilde a \xi \cdot \xi \,dx \right)
\\&
\geq
\frac{\lambda^2}{8} \left(\dashint_\domain (a-\tilde a)\xi\cdot \xi \,dx\right) \dashint_{\domain}  (a-\tilde a)(\xi+\nabla \phi^{L,a}_\xi) \cdot (\xi+\nabla \phi^{L,a}_\xi) \,dx
\\&~~~
+\frac{\lambda^2}{8} \left(\dashint_\domain (a-\tilde a)\xi\cdot \xi \,dx\right) \dashint_{\domain} (a-\tilde a)(\xi+\nabla \phi^{L,\tilde a}_\xi) \cdot (\xi+\nabla \phi^{L,\tilde a}_\xi) \,dx
\end{align*}
The estimate \eqref{LowerBoundCovSecondVersion} from Lemma~\ref{CovarianceEstimateFromBelow} implies
\begin{align*}
&\Cov\left[a^\RVE \xi \cdot \xi, \dashint_\domain a\xi\cdot \xi \,dx\right]
\\&
\geq
\frac{\lambda^2}{16} L^{-d} \mathbb{E}\Bigg[\sum_{k\in \mathbb{Z}^d\cap [0,L)^d} 
\sqrt{\dashint_\domain \big|(a(\Gamma)-a(\Delta_{k,\tilde \Gamma_k}\Gamma))\xi\cdot \xi\big| \,dx}
\\&~~~~~~~~~~~~~~~~~~~~~~~\times
\sqrt{\dashint_{\domain} \big|(a(\Gamma)-a(\Delta_{k,\tilde \Gamma_k}\Gamma))(\xi+\nabla \phi^{L,a}_\xi) \cdot (\xi+\nabla \phi^{L,a}_\xi) \big| \,dx}\Bigg]^2
\\&
\geq
\frac{\lambda^2}{16} L^{-d} \mathbb{E}\Bigg[\sum_{k\in \mathbb{Z}^d\cap [0,L)^d} L^{-d}
\sqrt{\int_\domain \big|(a-a(\Delta_{k,\tilde \Gamma_k}\Gamma))\xi\cdot \xi\big| \,dx}
\\&~~~~~~~~~~~~~~~~~~~~~~~~~~~~~\times
(2K)^{-d/2} \int_{\domain} \big|(a(\Gamma)-a(\Delta_{k,\tilde \Gamma_k}\Gamma))^{1/2}(\xi+\nabla \phi^{L,a}_\xi)\big| \,dx\Bigg]^2,
\end{align*}
where in the last step we have used the H\"older inequality and the fact that $a(x,\Gamma)-a(x,\Delta_{k,\tilde \Gamma_k}\Gamma)$ is only nonzero for $|x-k|\leq K$.

By our assumption \eqref{QuantifiedMonotonicity} we infer
\begin{align*}
&\Cov\left[a^\RVE \xi \cdot \xi, \dashint_\domain a\xi\cdot \xi \,dx\right]
\\&
\geq
\frac{\lambda^2}{16} L^{-d} \mathbb{E}\bigg[
L^{-d} (2K)^{-d/2} \int_{\domain} \nu \big|\xi+\nabla \phi^{L,a}_\xi\big| \,dx\bigg]^2
\\&
\geq
\frac{\lambda^2}{16} L^{-d} (2K)^{-d} \nu^2 \mathbb{E}\bigg[
 \bigg|\dashint_{\domain} \xi+\nabla \phi^{L,a}_\xi \,dx\bigg|\bigg]^2
\\&
\geq
\frac{\lambda^2}{16} L^{-d} (2K)^{-d} \nu^2 |\xi|^2.
\end{align*}
To conclude our proof, by
\begin{align*}
\rho_{a^\RVE\xi\cdot \xi,\mathcal{F}(a)} = \frac{\Cov[a^\RVE \xi\cdot \xi,\mathcal{F}(a)]}{\sqrt{\Var a^\RVE \xi\cdot \xi}\sqrt{\Var \mathcal{F}(a)}}
\end{align*}
it suffices to bound $\Var a^\RVE \xi\cdot \xi$ and $\Var \mathcal{F}(a)$ by $C(d,\lambda,K) L^{-d} |\xi|^2$. A corresponding bound for $\Var a^\RVE \xi \cdot \xi$ is provided e.\,g.\ by the methods of Gloria and Otto \cite{GloriaOttoNew}. To estimate $\Var \mathcal{F}(a)$, we simply apply \eqref{UpperBoundCov}, which yields
\begin{align*}
&\Var \mathcal{F}(a)
\\&
\leq \frac{1}{2} \sum_{k\in \mathbb{Z}^d\cap [0,L)^d} \mathbb{E}\bigg[\bigg(\dashint_{[0,L)^d} (a(x,\Gamma)-a(x,\Delta_{k,\tilde \Gamma(k)}\Gamma))\xi \cdot \xi \,dx\bigg)^2\bigg]
\\&
\leq \frac{1}{2} \sum_{k\in \mathbb{Z}^d\cap [0,L)^d} \mathbb{E}\Big[|\xi|^2 L^{-2d} (2K)^{2d} \Big]
\\&
\leq (2K)^{2d} |\xi|^2 L^{-d}.
\end{align*}
\end{proof}
In the previous proof, we have used the following standard estimate for covariances of nonlinear functions of a finite number of independent random variables.
\begin{lemma}
\label{CovarianceEstimateFromBelow}
Let $f:[0,1]^N \rightarrow \mathbb{R}$, $g:[0,1]^N \rightarrow \mathbb{R}$ be two functions that are monotonous with respect to each of their arguments. Let $X_i:\Omega\rightarrow [0,1]$, $1\leq i\leq N$, and $Y_i:\Omega\rightarrow [0,1]$, $1\leq i\leq N$, be $2N$ independent identically distributed random variables. Define
\begin{align*}
&h_n(X,x,y):=
\\&~~~~~~
|f(X_1,\ldots,X_{n-1},x,X_{n+1},\ldots,X_N)-f(X_1,\ldots,X_{n-1},y,X_{n+1},\ldots,X_N)|
\\&~~~~~~\times|g(X_1,\ldots,X_{n-1},x,X_{n+1},\ldots,X_N)-g(X_1,\ldots,X_{n-1},y,X_{n+1},\ldots,X_N)|
\end{align*}
and
\begin{align*}
&H_n(X,x,y):=
\\&~~~~
\frac{1}{2}|f(X_1,\ldots,X_{n-1},x,X_{n+1},\ldots,X_N)-f(X_1,\ldots,X_{n-1},y,X_{n+1},\ldots,X_N)|^2
\\&~~~~+\frac{1}{2}|g(X_1,\ldots,X_{n-1},x,X_{n+1},\ldots,X_N)-g(X_1,\ldots,X_{n-1},y,X_{n+1},\ldots,X_N)|^2.
\end{align*}
Then
\begin{align}
\label{LowerBoundCov}
\operatorname{Cov}[f(X),g(X)] \geq \frac{1}{2}\sum_{n=1}^N \mathbb{E}\Big[\sqrt{h_n(X,X_n,Y_n)}\Big]^2
\end{align}
and by Jensen's inequality
\begin{align}
\label{LowerBoundCovSecondVersion}
\operatorname{Cov}[f(X),g(X)] \geq \frac{1}{2}N^{-1} \mathbb{E}\bigg[\sum_{n=1}^N \sqrt{h_n(X,X_n,Y_n)}\bigg]^2.
\end{align}
Furthermore, we have
\begin{align}
\label{UpperBoundCov}
\operatorname{Cov}[f(X),g(X)] \leq \frac{1}{2}\sum_{n=1}^N \mathbb{E}\big[H_n(X,X_n,Y_n)\big].
\end{align}
\end{lemma}
\begin{proof}
The proof proceeds similarly to the proof of the standard form of this lemma which provides the weaker assertion $\operatorname{Cov}[f(X),g(X)]\geq 0$; see for example \cite[page 24]{LiuMonteCarloStrategies} or \cite[Lemma 2.1]{BlancCostaouecLeBrisLegoll}.

We have by the identity of laws of $(X_1,\ldots,X_{n-1},Y_n,Y_1,\ldots,Y_{n-1},X_n)$ and $(X_1,\ldots,X_n,Y_1,\ldots,Y_n)$ (which allows us to swap $X_n$ and $Y_n$ in the expectations below)
\begin{align*}
&\mathbb{E}\big[f(X_1,\ldots,X_{n-1},X_{n},\ldots,X_N)g(Y_1,\ldots,Y_{n-1},X_{n},\ldots,X_N)\big]
\\&
=\frac{1}{2}\mathbb{E}\Big[\big(f(X_1,\ldots,X_{n-1},X_n,X_{n+1},\ldots,X_N)-f(X_1,\ldots,X_{n-1},Y_n,X_{n+1},\ldots,X_N)\big)
\\&~~~~~~~~~~~\times
\big(g(Y_1,\ldots,Y_{n-1},X_{n},X_{n+1},\ldots,X_N)-g(Y_1,\ldots,Y_{n-1},Y_{n},X_{n+1},\ldots,X_N)\big)\Big]
\\&~~~
+\mathbb{E}\big[f(X_1,\ldots,X_n,X_{n+1},\ldots,X_N)g(Y_1,\ldots,Y_n,X_{n+1},\ldots,X_N)\big].
\end{align*}
By the independence of the $X_i$ and the $Y_i$, we infer
\begin{align*}
&\mathbb{E}\big[f(X_1,\ldots,X_{n-1},X_{n},\ldots,X_N)g(Y_1,\ldots,Y_{n-1},X_{n},\ldots,X_N)\big]
\\&
=\frac{1}{2}\mathbb{E}\bigg[\int_{[0,1]^{n-1}} f(x,X_n,X_{n+1},\ldots,X_N)-f(x,Y_n,X_{n+1},\ldots,X_N) \,d\mathbb{P}_{(X_1,\ldots,X_{n-1})}(x)
\\&~~~~~~~~~~~\times
\int_{[0,1]^{n-1}} g(y,X_n,X_{n+1},\ldots,X_N)-g(y,Y_n,X_{n+1},\ldots,X_N) \,d\mathbb{P}_{(Y_1,\ldots,Y_{n-1})}(y)\bigg]
\\&~~~
+\mathbb{E}\big[f(X_1,\ldots,X_n,X_{n+1},\ldots,X_N)g(Y_1,\ldots,Y_n,X_{n+1},\ldots,X_N)\big].
\end{align*}
As both $f$ and $g$ are increasing functions in each of their arguments, the integrands in this formula are either nonnegative (for $X_n\geq Y_n$) or nonpositive (for $X_n\leq Y_n$). Thus, we have
\begin{align}
\nonumber
&\mathbb{E}\big[f(X_1,\ldots,X_{n-1},X_{n},\ldots,X_N)g(Y_1,\ldots,Y_{n-1},X_{n},\ldots,X_N)\big]
\\&
\label{CovEqual}
=\frac{1}{2}\mathbb{E}\bigg[\int_{[0,1]^{n-1}} |f(x,X_n,X_{n+1},\ldots,X_N)-f(x,Y_n,X_{n+1},\ldots,X_N)| \,d\mathbb{P}_{(X_1,\ldots,X_{n-1})}(x)
\\&~~~~~~~~~~~\times
\nonumber
\int_{[0,1]^{n-1}} |g(y,X_n,X_{n+1},\ldots,X_N)-g(y,Y_n,X_{n+1},\ldots,X_N)| \,d\mathbb{P}_{(Y_1,\ldots,Y_{n-1})}(y)\bigg]
\\&~~~\nonumber
+\mathbb{E}\big[f(X_1,\ldots,X_n,X_{n+1},\ldots,X_N)g(Y_1,\ldots,Y_n,X_{n+1},\ldots,X_N)\big]
\end{align}
and therefore by H\"older's inequality
\begin{align*}
&\mathbb{E}\big[f(X_1,\ldots,X_{n-1},X_{n},\ldots,X_N)g(Y_1,\ldots,Y_{n-1},X_{n},\ldots,X_N)\big]
\\&
\geq \frac{1}{2}\mathbb{E}\bigg[\sqrt{|f(X_1,\ldots,X_{n-1},X_n,X_{n+1},\ldots,X_N)-f(X_1,\ldots,X_{n-1},Y_n,X_{n+1},\ldots,X_N)|}
\\&~~~~~~~~~~~\times
\sqrt{|g(X_1,\ldots,X_{n-1},X_n,X_{n+1},\ldots,X_N)-g(X_1,\ldots,X_{n-1},Y_n,X_{n+1},\ldots,X_N)|}\bigg]^2
\\&~~~
+\mathbb{E}\big[f(X_1,\ldots,X_n,X_{n+1},\ldots,X_N)g(Y_1,\ldots,Y_n,X_{n+1},\ldots,X_N)\big].
\end{align*}
Taking the sum of these formulas for $n=1,\ldots,N$, we infer
\begin{align*}
&\mathbb{E}\big[f(X)g(X)\big]
\geq \frac{1}{2} \sum_{n=1}^N \mathbb{E}\Big[\sqrt{h_n(X,X_n,Y_n)}\Big]^2
+\mathbb{E}\big[f(X)g(Y)\big],
\end{align*}
which establishes the desired lower bound \eqref{LowerBoundCov} for the covariance.

To obtain \eqref{UpperBoundCov}, we apply Young's inequality and subsequently Jensen's inequality to \eqref{CovEqual}, which yields
\begin{align*}
&\mathbb{E}\big[f(X_1,\ldots,X_{n-1},X_{n},\ldots,X_N)g(Y_1,\ldots,Y_{n-1},X_{n},\ldots,X_N)\big]
\\&
\leq\frac{1}{2}\mathbb{E}\bigg[\frac{1}{2}\int_{[0,1]^{n-1}} |f(x,X_n,X_{n+1},\ldots,X_N)-f(x,Y_n,X_{n+1},\ldots,X_N)|^2 \,d\mathbb{P}_{(X_1,\ldots,X_{n-1})}(x)
\\&~~~~~~~~~~~
+\frac{1}{2}\int_{[0,1]^{n-1}} |g(y,X_n,X_{n+1},\ldots,X_N)-g(y,Y_n,X_{n+1},\ldots,X_N)|^2 \,d\mathbb{P}_{(Y_1,\ldots,Y_{n-1})}(y)\bigg]
\\&~~~
+\mathbb{E}\big[f(X_1,\ldots,X_n,X_{n+1},\ldots,X_N)g(Y_1,\ldots,Y_n,X_{n+1},\ldots,X_N)\big].
\end{align*}
This is equivalent to
\begin{align*}
&\mathbb{E}\big[f(X_1,\ldots,X_{n-1},X_{n},\ldots,X_N)g(Y_1,\ldots,Y_{n-1},X_{n},\ldots,X_N)\big]
\\&
\leq \frac{1}{2}\mathbb{E}\bigg[\frac{1}{2} |f(X_1,\ldots,X_{n-1},X_n,X_{n+1},\ldots,X_N)-f(X_1,\ldots,X_{n-1},Y_n,X_{n+1},\ldots,X_N)|^2
\\&~~~~~~~~~~~
+\frac{1}{2} |g(X_1,\ldots,X_{n-1},X_n,X_{n+1},\ldots,X_N)-g(X_1,\ldots,X_{n-1},Y_n,X_{n+1},\ldots,X_N)|^2 \bigg]
\\&~~~
+\mathbb{E}\big[f(X_1,\ldots,X_n,X_{n+1},\ldots,X_N)g(Y_1,\ldots,Y_n,X_{n+1},\ldots,X_N)\big].
\end{align*}
Taking the sum with respect to $n$ entails
\begin{align*}
&\mathbb{E}\big[f(X)g(X)\big]
\leq \frac{1}{2} \sum_{n=1}^N \mathbb{E}[H_n(X,X_n,Y_n)]
+\mathbb{E}\big[f(X)g(Y)\big],
\end{align*}
which establishes the upper bound \eqref{UpperBoundCov} for the covariance.
\end{proof}

\appendix

\section{Gaussian propagation bounds for parabolic PDEs}

We now collect some elementary energy and propagation estimates for second-order linear parabolic equations. By a \emph{nongrowing weak solution} $u$ to the equation $\partial_t u = \nabla \cdot (a\nabla u)$ with initial data $u(\cdot,0)=g$, we understand a function $u\in L^2_{loc}(\mathbb{R}^d\times [0,\infty))$ with $\nabla u\in L^2_{loc}(\mathbb{R}^d\times [0,\infty))$ satisfying the usual weak formulation of the PDE with test functions in $C^\infty_{cpt}(\mathbb{R}^d\times [0,\infty))$ and additionally the estimate
\begin{align*}
\sup_{r\geq 0} \int_0^T \dashint_{\{|x|\leq r\}} |u|^2 \,dx\,dt<\infty
\end{align*}
for any $T>0$. Note that for initial data $u(\cdot,0)=\nabla \cdot b$ for some vector field $b\in L^\infty(\mathbb{R}^d;\mathbb{R}^d)$, the initial data is incorporated into the weak formulation in a weak form, i.\,e.\ as
\begin{align*}
-\int_0^\infty \int_{\mathbb{R}^d} u \partial_t \eta \,dx\,dt
=-\int_{\mathbb{R}^d} b\cdot \nabla \eta \,dx
+\int_0^\infty \int_{\mathbb{R}^d} a\nabla u \cdot \nabla \eta \,dx\,dt.
\end{align*}
Many of our computations in the next sections will be formal, but can be justified by the appropriate standard approximation arguments. Note also that the estimates which we shall prove ensure the \emph{existence} of such nongrowing weak solutions for merely $b\in L^\infty(\mathbb{R}^d;\mathbb{R}^d)$, as they ensure that one may construct a solution by constructing solutions with the initial data $b$ truncated outside of some large ball $\{|x|\leq R\}$ (in which case the standard existence theorems apply) and then passing to the limit $R\rightarrow \infty$.

\begin{lemma}
\label{GaussianPropagationBounds}
Let $a$ be a uniformly elliptic and bounded coefficient field on $\mathbb{R}^d$. For $r\geq 0$ and $M\geq 5d$, define the coefficient field
\begin{align*}
a_{r,M}(x):=
\begin{cases}
a(x)&\text{for }|x|\leq M r,
\\
\Id&\text{otherwise}.
\end{cases}
\end{align*}
Consider the unique nongrowing weak solutions $u_i$ and $u_{i,r,M}$ to the equations
\begin{align*}
\frac{d}{dt} u_i&=\nabla \cdot (a \nabla u_i),
\\
u_i(\cdot,0)&=\nabla \cdot (a e_i),
\end{align*}
and
\begin{align*}
\frac{d}{dt} u_{i,r,M}&=\nabla \cdot (a_{r,M} \nabla u_{i,r,M}),
\\
u_{i,r,M}(\cdot,0)&=\nabla \cdot (a_{r,M} e_i).
\end{align*}
Then we have
\begin{align}
\label{L2ErrorBound}
\dashint_{\{|x|\leq 2dr\}} |u_i(\cdot,t)-u_{i,r,M}(\cdot,t)|^2 \,dx
\leq \frac{C M^{d/2}}{t}\exp\bigg(-c\frac{M^2 r^2}{t}\bigg)
\end{align}
for any $t\leq 16 M^2 r^2$ and
\begin{align}
\label{H1ErrorBound}
\int_{0}^{16 r^2} \dashint_{\{|x|\leq dr\}} |\nabla u_i-\nabla u_{i,r,M}|^2 \,dx\,dt
\leq \frac{C}{r^2} \exp(-cM^2).
\end{align}
\end{lemma}
\begin{proof}
For an arbitrary function $\psi\in L^2(\mathbb{R}^d)$ supported in $\{|x|\leq 2dr\}$ and any $T\in [0,16M^2 r^2]$, consider the solutions $v_\psi$ and $v_{\psi,r,M}$ to the dual equations
\begin{align*}
-\frac{d}{dt}v_{\psi} &= \nabla \cdot (a^*\nabla v_\psi),
\\
v_\psi(\cdot,T)&=\psi.
\end{align*}
and
\begin{align*}
-\frac{d}{dt}v_{\psi,r,M} &= \nabla \cdot (a_{r,M}^*\nabla v_{\psi,r,M}),
\\
v_{\psi,r,M}(\cdot,T)&=\psi.
\end{align*}
We then have
\begin{align*}
&\int_{\mathbb{R}^d} (u_i-u_{i,r,M})(\cdot,T) \psi \,dx
\\&
=\int_{\mathbb{R}^d} u_i(\cdot,0)v_{\psi}(\cdot,0)-u_{i,r,M}(\cdot,0) v_{\psi,r,M}(\cdot,0) \,dx
\\&~~~~
+\int_0^T \frac{d}{dt}\int_{\mathbb{R}^d} u_i v_{\psi}-u_{i,r,M} v_{\psi,r,M} \,dx \,dt
\\&
=-\int_{\mathbb{R}^d} a e_i \cdot \nabla v_\psi(\cdot,0)-a_{r,M}e_i \cdot \nabla v_{\psi,r,M}(\cdot,0) \,dx+0
\\&
\leq
C\int_{\{|x|\geq \frac{M}{2}r\}} |\nabla v_\psi(\cdot,0)|+|\nabla v_{\psi,r,M}(\cdot,0)| \,dx
\\&~~~~~
+C\int_{\{|x|\leq \frac{M}{2}r\}}|\nabla (v_\psi(\cdot,0)-v_{\psi,r,M}(\cdot,0))| \,dx
\\&
\leq
C\left(\int_{\{|x|\geq \frac{M}{2}r\}} \left(\frac{|x|}{Mr}\right)^{2d} \big(|\nabla v_\psi(\cdot,0)|^2+|\nabla v_{\psi,r,M}(\cdot,0)|^2\big) \,dx\right)^{1/2} (Mr)^{d/2}
\\&~~~~~
+C\left(\int_{\{|x|\leq \frac{M}{2}r\}}|\nabla (v_\psi(\cdot,0)-v_{\psi,r,M}(\cdot,0))|^2 \,dx\right)^{1/2} (Mr)^{d/2}.
\end{align*}
The penultimate term may be estimated by Lemma~\ref{ParabolicPDERegularity} (applied to the backward-in-time equations for $v_\psi$ and $v_{\psi,r,M}$ and breaking up the ``initial'' condition $\psi$ into pieces supported on scale $\sqrt{T}$ if necessary), resulting in the bound (note that $2dr\leq \frac{Mr}{4}$)
\begin{align*}
&\bigg|\int_{\mathbb{R}^d} (u_i-u_{i,r,M})(\cdot,T) \psi \,dx\bigg|
\\&
\leq
C\sum_{x_0\in \frac{1}{d}\sqrt{T}\mathbb{Z}^d \cap \{|x|\leq 2dr\}}
\bigg(\int_{\{|x|\geq \frac{M}{2}r\}} \left(\frac{|x|}{Mr}\right)^{2d} \cdot \sqrt{T}^d \exp\Big(-\frac{|x-x_0|^2}{C T}\Big)
\\&~~~~~~~~~~~~~~~~~~~~~~~~~~~~~~~~~~~~
\times C(d,\lambda) \dashint_{\{|x-x_0|\leq \sqrt{T}\}} |\psi(\tilde x)|^2 \,d\tilde x \cdot T^{-1} \,dx\bigg)^{1/2} (Mr)^{d/2}
\\&~~~~~
+C\left(\int_{\{|x|\leq \frac{M}{2}r\}}|\nabla (v_\psi(\cdot,0)-v_{\psi,r,M}(\cdot,0))|^2 \,dx\right)^{1/2} (Mr)^{d/2}
\end{align*}
and therefore by $\sqrt{T}\leq 4 M r$
\begin{align}
\nonumber
&\bigg|\int_{\mathbb{R}^d} (u_i-u_{i,r,M})(\cdot,T) \psi \,dx\bigg|
\\&
\leq
\nonumber
C\Big(\frac{4Mr}{\sqrt{T}}\Big)^d\times \exp\Big(-\frac{M^2r^2}{CT}\Big) T^{-1/2} \bigg(\int |\psi|^2 \,dx\bigg)^{1/2} \times (Mr)^{d/2}
\\&~~~~~
\label{EquationDifferenceuuk}
+C\left(\int_{\{|x|\leq \frac{M}{2}r\}}|\nabla (v_\psi(\cdot,0)-v_{\psi,r,M}(\cdot,0))|^2 \,dx\right)^{1/2} (Mr)^{d/2}.
\end{align}
An estimate for the last term on the right-hand side of \eqref{EquationDifferenceuuk} can be obtained as follows: Observe that
\begin{align*}
-\frac{d}{dt}(v_\psi-v_{\psi,r,M}) &= \nabla \cdot (a^*\nabla (v_\psi-v_{\psi,r,M})) + \nabla \cdot ((a^*-a_{r,M}^*)\nabla v_{\psi,r,M}),
\\
(v_\psi-v_{\psi,r,M})(\cdot,T) &= 0.
\end{align*}
We rewrite $(v_\psi-v_{\psi,r,M})(\cdot,0)$ as $(v_\psi-v_{\psi,r,M})(\cdot,0)=\int_0^T w_t(\cdot,0) \,dt$ with $w_{t_0}$ being the solution to the equation
\begin{align*}
-\frac{d}{dt} w_{t_0} &= \nabla \cdot (a^* \nabla w_{t_0}),
\\
w_{t_0} (\cdot,t_0) &= \nabla \cdot ((a^*-a_{r,M}^*)\nabla v_{\psi,r,M}(\cdot,t_0)).
\end{align*}
Considering the estimate \eqref{BoundHeatEquationGradientWeakInitialData} centered at $x_0$ (instead of $0$) and integrating over the set $\{|x_0|\leq \frac{M}{2}r\}$ and applying it to the backward-in-time equation for $w_{t_0}$, we obtain using also the condition $t_0 \leq T\leq C M^2 r^2$
\begin{align*}
&\int_{\{|x|\leq \frac{M}{2}r\}} |\nabla w_{t_0}(\cdot,0)|^2 \,dx
\\&
\leq C(d,\lambda) t_0^{-2} \int_{\{|x_0|\leq \frac{M}{2}r\}} \int  |(a^*-a_{r,M}^*)\nabla v_{\psi,r,M}|^2(x) \cdot t_0^{-d/2} \exp\bigg(-\frac{|x-x_0|^2}{C t_0}\bigg)  \,dx \,dx_0
\\&
\leq C(d,\lambda) t_0^{-2} \int_{\{|x_0|\leq \frac{M}{2}r\}} \int_{\{|x|\geq Mr\}} |\nabla v_{\psi,r,M}(x)|^2 \cdot t_0^{-d/2} \exp\bigg(-\frac{|x-x_0|^2}{C t_0}\bigg)  \,dx \,dx_0
\\&
\leq C(d,\lambda) t_0^{-2} \exp\bigg(-\frac{M^2 r^2}{2C t_0}\bigg) \int  |\nabla v_{\psi,r,M}(x)|^2  \cdot \exp\bigg(-\frac{|x|^2}{4C t_0}\bigg)  \,dx.
\end{align*}
Lemma~\ref{ParabolicPDERegularity} (applied to $v_{\psi,r,M}$) implies by breaking up the ``initial'' condition $\psi$ into contributions supported on balls of size $\sqrt{T-t_0}$
\begin{align*}
\int_{\mathbb{R}^d} |\nabla v_{\psi,r,M}(x)|^2  \,dx
\leq \frac{C}{T-t_0} \int_{\mathbb{R}^d} |\psi|^2 \,dx.
\end{align*}
Combining the previous two estimates, we deduce
\begin{align*}
&\int_{\{|x|\leq \frac{M}{2}r\}} |\nabla w_{t_0}(\cdot,0)|^2 \,dx
\leq C(d,\lambda) t_0^{-2}  \exp\bigg(-\frac{M^2 r^2}{2C t_0}\bigg) (T-t_0)^{-1} \int |\psi(x)|^2 \,dx.
\end{align*}
Taking the square root and integrating with respect to $t_0$, this entails
\begin{align*}
&\bigg(\int_{\{|x|\leq \frac{M}{2}r\}} |\nabla (v_\psi-v_{\psi,r,M})|^2 \,dx\bigg)^{1/2}
\\&
\leq C(d,\lambda) \int_0^T t_0^{-1} (T-t_0)^{-1/2} \exp\bigg(-\frac{M^2 r^2}{C t_0}\bigg) \,dt_0 \cdot \bigg(\int |\psi(x)|^2 \,dx\bigg)^{1/2}
\\&
\leq C(d,\lambda) \bigg(\frac{T^{1/2}}{M^2r^2} + \frac{1}{T^{1/2}} \bigg) \exp\bigg(-\frac{M^2 r^2}{C T}\bigg) \bigg(\int |\psi(x)|^2 \,dx\bigg)^{1/2}.
\end{align*}
Using $T\leq C M^2 r^2$ and plugging in this bound into \eqref{EquationDifferenceuuk}, we get by $M\geq 5d$
\begin{align*}
&\bigg|\int_{\mathbb{R}^d} (u_i-u_{i,r,M})(\cdot,T) \psi \,dx\bigg|
\\&
\leq
C(d,\lambda) T^{-1/2} \exp\bigg(-\frac{M^2 r^2}{C T}\bigg) \cdot \bigg(\int |\psi|^2 \,dx\bigg)^{1/2} (Mr)^{d/2}.
\end{align*}
Passing to the supremum over all $\psi$ supported in $\{|x|\leq 2dr\}$ with $\int |\psi|^2 \,dx\leq 1$, we deduce our bound \eqref{L2ErrorBound}.

Now choose a cutoff $\eta$ with $\eta\equiv 1$ in $\{|x|\leq dr\}$ and $\eta \equiv 0$ outside of $\{|x|\leq 2dr\}$. For any $t\leq 16 r^2$, we obtain by testing the equation for the difference $u_i-u_{i,r,M}$ with $(u_i-u_{i,r,M})\eta^2$
\begin{align*}
&\int_t^{2t} \int \eta^2 \lambda |\nabla (u_i-u_{i,r,M})|^2 \,dx \,d\tilde t
\\&
\leq \int_{\{|x|\leq 2dr\}} \eta^2 |u_i(\cdot,t)-u_{i,r,M}(\cdot,t)|^2 \,dx
\\&~~~~
+C\int_t^{2t} \int |\nabla \eta|^2 |u_i(\cdot,t)-u_{i,r,M}(\cdot,t)|^2 \,dx \,d\tilde t
\\&
\leq \int_{\{|x|\leq 2dr\}} |u_i(\cdot,t)-u_{i,r,M}(\cdot,t)|^2 \,dx
\\&~~~~
+C\int_t^{2t} \int_{\{|x|\leq 2dr\}} \frac{C}{r^2} |u_i(\cdot,t)-u_{i,r,M}(\cdot,t)|^2 \,dx \,d\tilde t.
\end{align*}
Using our bound \eqref{L2ErrorBound} and $t\leq 16r^2$, we get
\begin{align*}
&\int_t^{2t} \dashint_{\{|x|\leq dr\}} |\nabla (u_i-u_{i,r,M})|^2 \,dx \,d\tilde t
\leq C(d,\lambda) \frac{C M^{d/2}}{t}\exp\bigg(-c\frac{M^2 r^2}{t}\bigg).
\end{align*}
Taking the sum over all $t=2^k$ for $2^k\leq T$, we deduce our desired estimate \eqref{H1ErrorBound}.
\end{proof}

\begin{lemma}
\label{DeterministicBoundsParabolic}
Let $a\in L^\infty(\mathbb{R}^d;\mathbb{R}^{d\times d})$ be a uniformly elliptic and bounded coefficient field in the sense of (A1). Let $b\in L^\infty(\mathbb{R}^d;\mathbb{R}^d)$ be a bounded vector field. Then the unique nongrowing weak solution $w$ to the equation
\begin{align*}
\frac{d}{dt} w &= \nabla \cdot (a\nabla w),
\\
w(\cdot,0) &= \nabla \cdot b,
\end{align*}
satisfies for any $T>0$ the estimate
\begin{align}
\label{BoundHeatEquationWeakInitialData}
&\bigg(\dashint_{\{|x|\leq \sqrt{T}\}} |w(\cdot,T)|^2 \,dx\bigg)^{1/2}
\\&~~~~~~
\nonumber
\leq C(d,\lambda) T^{-1/2} \bigg(\int  |b(x)|^2 \cdot T^{-d/2} \exp\bigg(-\frac{|x|^2}{C T}\bigg)  \,dx\bigg)^{1/2}.
\end{align}
Furthermore, we have the bounds
\begin{align}
\label{BoundHeatEquationIntegratedGradientWeakInitialData}
\dashint_{\{|x|\leq 1\}} \bigg|\nabla \int_{0}^1 w(\cdot,t) \,dt\bigg|^2 \,dx
\leq C(d,\lambda) ||b||_{L^\infty}^2
\end{align}
and
\begin{align}
\label{BoundHeatEquationGradientWeakInitialData}
&\bigg(\dashint_{\{|x|\leq \sqrt{T}\}} |\nabla w(\cdot,T)|^2 \,dx\bigg)^{1/2}
\\&~~~~~~
\nonumber
\leq C(d,\lambda) T^{-1} \bigg(\int  |b(x)|^2 \cdot T^{-d/2} \exp\bigg(-\frac{|x|^2}{C T}\bigg)  \,dx\bigg)^{1/2}.
\end{align}
\end{lemma}
\begin{proof}
Let $T>0$ and let $g\in L^2(\mathbb{R}^d)$ be a function supported in $\{|x|\leq \sqrt{T}\}$. Introducing the solution $v$ to the dual (backward-in-time) equation
\begin{align*}
-\frac{d}{dt} v &= \nabla \cdot (a^*\nabla v),
\\
v(\cdot,T)&=g,
\end{align*}
we see that we have
\begin{align*}
&\int w(\cdot,T) g \,dx
\\
&= -\int b \cdot \nabla v(\cdot,0) \,dx + \int_0^T \frac{d}{dt} \int w v \,dx\,dt
\\&
= -\int b \cdot \nabla v(\cdot,0) \,dx + \int_0^T \int -a\nabla w \cdot \nabla v + \nabla w \cdot a^* \nabla v \,dx\,dt
\\&
=-\int b \cdot \nabla v(\cdot,0) \,dx.
\end{align*}
Introducing 
\begin{align*}
\Theta_T(x):= \exp\bigg(\frac{|x|^2}{CT}\bigg),
\end{align*}
we obtain
\begin{align*}
&\int w(\cdot,T) g \,dx
\leq \bigg(\int |b(x)|^2 \frac{1}{\Theta_T(x)} \,dx\bigg)^{1/2} \bigg(\int |\nabla v(\cdot,0)|^2 \Theta_T(\cdot) \,dx\bigg)^{1/2}.
\end{align*}
Lemma~\ref{ParabolicPDERegularity} (applied to $v$, which solves a parabolic PDE backward in time) provides the estimate
\begin{align*}
\int |\nabla v(\cdot,0)|^2 \Theta_T(\cdot) \,dx
\leq C(d,\lambda) T^{-1} \int |g|^2 \,dx.
\end{align*}
Inserting this estimate in the previous inequality and passing to the supremum over all $g\in L^2$ supported in $\{|x|\leq \sqrt{T}\}$ with $\int_{\{|x|\leq \sqrt{T}\}} |g|^2 \,dx \leq 1$, we get
\begin{align*}
\bigg(\int_{\{|x|\leq \sqrt{T}\}} |w(\cdot,T)|^2 \,dx\bigg)^{1/2}
\leq C T^{-1/2} \bigg(\int |b(x)|^2 \frac{1}{\Theta_T(x)} \,dx\bigg)^{1/2}.
\end{align*}
This establishes the estimate \eqref{BoundHeatEquationWeakInitialData}.

To prove the estimate \eqref{BoundHeatEquationIntegratedGradientWeakInitialData}, we first observe that we have
\begin{align*}
\nabla \cdot \bigg(a\nabla \int_0^1 w(\cdot,t) \,dt + b\bigg) = w(\cdot,1).
\end{align*}
Testing this PDE with $\eta^2 \int_0^1 w(\cdot,t) \,dt$ where $\eta$ is a standard cutoff with $\eta\equiv 1$ in $\{|x|\leq 1\}$ and $\eta\equiv 0$ outside of $\{|x|\leq 2\}$, we obtain
\begin{align*}
&\int \eta^2 \bigg|\nabla \int_0^1 w(\cdot,t) \,dt \bigg|^2 \,dx
\\&
\leq \int C |\eta|^2 |b|^2 + C (\eta^2+|\nabla \eta|^2) \bigg|\int_0^1 w(\cdot,t) \,dt \bigg|^2 + C |\eta|^2 |w(\cdot,1)|^2 \,dx
\\&
\leq \int C |\eta|^2 |b|^2 + C (\eta^2+|\nabla \eta|^2) \int_0^1 \sqrt{t} |w(\cdot,t)|^2 \,dt + C |\eta|^2 |w(\cdot,1)|^2 \,dx.
\end{align*}
The estimate \eqref{BoundHeatEquationWeakInitialData} entails
\begin{align*}
&\bigg(\dashint_{\{|x|\leq \sqrt{t}\}} |w(\cdot,T)|^2 \,dx\bigg)^{1/2}
\leq C t^{-1/2} ||b||_{L^\infty}.
\end{align*}
The previous two estimates yield \eqref{BoundHeatEquationIntegratedGradientWeakInitialData}.

Finally, to prove \eqref{BoundHeatEquationGradientWeakInitialData}, we first deduce from \eqref{BoundHeatEquationWeakInitialData}
\begin{align*}
&\bigg(\dashint_{\{|x-x_0|\leq \sqrt{T/2}\}} |w(\cdot,T/2)|^2 \,dx\bigg)^{1/2}
\\&~~~~~~
\nonumber
\leq C(d,\lambda) T^{-1/2} \bigg(\int  |b(x)|^2 \cdot T^{-d/2} \exp\bigg(-\frac{|x-x_0|^2}{C T}\bigg)  \,dx\bigg)^{1/2}.
\end{align*}
Splitting the function $w(\cdot,T/2)$ into pieces each supported on a ball of size $\sqrt{T}/2$ -- that is, splitting $w(\cdot,T/2)=\sum_l \eta_l w(\cdot,T/2)$ with a partition of unity $\eta_l$ subordinate to the set of balls $\{|x-x_0|\leq \sqrt{T/2}\}$, $x_0 \in \frac{1}{d} \sqrt{T/2} \mathbb{Z}^d$ -- and applying Lemma~\ref{ParabolicPDERegularity} to the solutions of the parabolic equation with initial data $\eta_l w(\cdot,T/2)$ for all $l$ (note that $w$ is equal to the sum of all of these solutions), we obtain
\begin{align*}
&\Bigg(\dashint_{\{|x|\leq \sqrt{T/2}\}} |\nabla w(\cdot,T)|^2 \,dx\Bigg)^{1/2}
\\&
\leq \sum_{x_0 \in \frac{1}{d} \sqrt{T/2} \mathbb{Z}^d}
\exp\bigg(-\frac{|x_0|^2}{CT/2}\bigg)
\Bigg(\dashint_{\{|x-x_0|\leq \sqrt{T/2}\}} |w(\cdot,T/2)|^2 \,dx\Bigg)^{1/2} T^{-1/2}
\\&
\leq \sum_{x_0 \in \frac{1}{d} \sqrt{T/2} \mathbb{Z}^d}
T^{-d/2} \exp\bigg(-\frac{|x_0|^2}{CT/2}\bigg)
\\&~~~~~~~~~~~~~
\times
C(d,\lambda) T^{-1} \bigg(\int  |b(x)|^2 \cdot T^{-d/2} \exp\bigg(-\frac{|x-x_0|^2}{C T}\bigg)  \,dx\bigg)^{1/2}.
\end{align*}
A straightforward estimate then entails \eqref{BoundHeatEquationGradientWeakInitialData} (with a different constant $C$).
\end{proof}
\begin{lemma}
\label{ParabolicPDERegularity}
Let $a\in L^\infty(\mathbb{R}^d;\mathbb{R}^{d\times d})$ be a uniformly elliptic and bounded coefficient field in the sense of (A1) and let $T>0$. Let $g\in L^2(\mathbb{R}^d)$ be a function supported in $\{|x|\leq \sqrt{T}\}$. Then there exists $C=C(d,\lambda)>0$ such that the unique nongrowing weak solution $w$ to the equation
\begin{align*}
\frac{d}{dt} w &= \nabla \cdot (a\nabla w),
\\
w(\cdot,0) &= g,
\end{align*}
satisfies the estimate
\begin{align*}
&\Bigg(\int_{\mathbb{R}^d} |\nabla w(\cdot,T)|^2 \frac{1}{\sqrt{T}^d} \exp\bigg(\frac{|x|^2}{CT}\bigg) \,dx\Bigg)^{1/2}
\leq C(d,\lambda) \bigg(\dashint_{\{|x|\leq \sqrt{T}\}} |g|^2 \,dx\bigg)^{1/2} T^{-1/2}.
\end{align*}
\end{lemma}
\begin{proof}
As
\begin{align*}
\Theta_T^m(x,t):=\exp\bigg(\frac{|x|^2}{4C_m(T+t)}\bigg)
\end{align*}
satisfies
\begin{align}
\label{PropertyTheta}
\frac{d}{dt} \Theta_T^m + C_m\frac{|\nabla \Theta_T^m|^2}{\Theta_T^m} = \bigg(-\frac{|x|^2}{4C_m(T+t)^2}+C_m\Big|\frac{2x}{4C_m(T+t)}\Big|^2\bigg)\Theta_T^m \leq 0,
\end{align}
we have for $C_1\geq C(d,\lambda)$
\begin{align*}
&\frac{d}{dt} \int |w|^2 \Theta_T^1 \,dx
\\&
=\int |w|^2 \frac{d}{dt} \Theta_T^1 + 2\Theta_T^1 a\nabla w \cdot \nabla w +  2a\nabla w \cdot w \nabla \Theta_T^1 \,dx
\\&
\leq -\lambda \int \Theta_T^1 |\nabla w|^2 \,dx.
\end{align*}
This provides the bound
\begin{align}
\label{FirstBound}
\int_0^\infty \int |\nabla w|^2 \Theta_T^1 \,dx \,dt
\leq C \int |g|^2 \,dx.
\end{align}
We now would like to show (basically) $\nabla w\in C^\gamma([\frac{1}{2}T,\frac{3}{2}T];L^2_{\Theta_T^m})$ for some $\gamma>0$ and some $m$. To this aim, we abbreviate $\Theta_{T,t}^m := \Theta_{T}^m(\cdot,t)$ and compute
\begin{align*}
&\int |w(\cdot,t+h)-w(\cdot,t)|^2 \Theta_{T,t}^m \,dx
=\int \int_t^{t+h} \frac{d}{ds}w(\cdot,s) \,ds \,(w(\cdot,t+h)-w(\cdot,t)) \Theta_{T,t}^m  \,dx
\\&
= -\int \Theta_{T,t}^m \int_t^{t+h} a\nabla w \,ds \cdot \nabla (w(\cdot,t+h)-w(\cdot,t)) \,dx
\\&~~~
-\int (w(\cdot,t+h)-w(\cdot,t)) \int_t^{t+h} a\nabla w \,ds \cdot \nabla \Theta_{T,t}^m \,dx.
\end{align*}
Applying the H\"older inequality to the first term and Young's inequality (and absorption) to the second term, we get
\begin{align}
\label{FractionalDerivativeBound}
&\int |w(\cdot,t+h)-w(\cdot,t)|^2 \Theta_{T,t}^m \,dx
\\&
\nonumber
\leq C \bigg(\int \Theta_{T,t}^m \bigg|\int_t^{t+h} a\nabla w \,ds\bigg|^2 \,dx\bigg)^{1/2} \bigg(\int \Theta_{T,t}^m \big|\nabla w(\cdot,t+h)-\nabla w(\cdot,t)\big|^2 \,dx\bigg)^{1/2}
\\&~~~
\nonumber
+C\int \bigg|\int_t^{t+h} a\nabla w \,ds\bigg|^2 \frac{|\nabla \Theta_{T,t}^m|^2}{\Theta_{T,t}^m} \,dx
\\&
\nonumber
\leq C \sqrt{h} \bigg(\int \int_t^{t+h} \Theta_{T,t}^m |\nabla w|^2 \,ds \,dx\bigg)^{1/2} \bigg(\int \Theta_{T,t}^m \big|\nabla w(\cdot,t+h)-\nabla w(\cdot,t)\big|^2 \,dx\bigg)^{1/2}
\\&~~~
\nonumber
+C h\int \int_t^{t+h} |\nabla w|^2 \,ds \frac{|\nabla \Theta_{T,t}^m|^2}{\Theta_{T,t}^m} \,dx.
\end{align}
Choosing a weight $\Theta_T^2$ with slower growth than in \eqref{FirstBound} -- for example, setting $C_2:=4C_1$ -- , we may ensure that $\frac{|\nabla \Theta_T^2(\cdot,t)|^2}{\Theta_T^2(\cdot,t)}\leq \frac{C}{T} \Theta_T^1(\cdot,\tilde t)$ and $\Theta_T^2(\cdot,t)\leq \Theta_T^1(\cdot,\tilde t)$ for any $t,\tilde t\in [0,\frac{T}{3}]$. As a consequence, we may find for any $h\leq \frac{T}{10}$ a suitable $t\in [0,\frac{T}{10}]$ with
\begin{align*}
\int_t^{t+h} \int |\nabla w|^2 \bigg(\Theta_{T,t}^2+T\frac{|\nabla \Theta_{T,t}^2|^2}{\Theta_{T,t}^2}\bigg) \,dx \,dt
&\leq C\frac{h}{T} 
\int_0^{\frac{T}{5}} \int |\nabla w|^2 \Theta_T^1 \,dx \,dt,
\\
\int \Theta_{T,t}^2 \big(\big|\nabla w(\cdot,t+h)|^2 + |\nabla w(\cdot,t)\big|^2\big) \,dx
&\leq \frac{C}{T} \int_0^{\frac{T}{5}} \int |\nabla w|^2 \Theta_T^1 \,dx \,dt.
\end{align*}
Plugging in these bounds in the previous estimate and using \eqref{FirstBound}, we obtain for this $t$
\begin{align}
\label{Existencetfirst}
\int |w(\cdot,t+h)-w(\cdot,t)|^2 \Theta_T^2(\cdot,t) \,dx \leq C \bigg(\frac{h}{T}+\frac{h^2}{T^2}\bigg) \int |g|^2 \,dx.
\end{align}
Abbreviating $\Delta_h w (\cdot,t) := w(\cdot,t+h)-w(\cdot,t)$, we compute for $C_m\geq C(d,\lambda)$
\begin{align}
\label{DissipationFractionalTime}
&\frac{d}{dt} \int |\Delta_h w |^2 \Theta_T^m \,dx
\\&
\nonumber
=\int 2 \Theta_T^m  \Delta_h w ~ \frac{d}{dt}  \Delta_h w
+|\Delta_h w|^2 \frac{d}{dt} \Theta_T^m \,dx
\\&
\nonumber
=\int -2\Theta_T^m a\nabla \Delta_h w \cdot \nabla \Delta_h w
- 2 \Delta_h w ~ a\nabla \Delta_h w \cdot \nabla \Theta_T^m
+ |\Delta_h w|^2 \frac{d}{dt} \Theta_T^m \,dx
\\&
\nonumber
\stackrel{\eqref{PropertyTheta}}{\leq} -\lambda \int \Theta_T^m \big|\Delta_h \nabla w\big|^2 \,dx.
\end{align}
Combining this with the existence of $t\in [0,\frac{T}{10}]$ for which the bound \eqref{Existencetfirst} holds, this entails for any $h\leq \frac{T}{10}$
\begin{align*}
\int_{\frac{T}{5}}^\infty \int \Theta_T^2 \big|\Delta_h \nabla w\big|^2 \,dx \,dt \leq C \frac{h}{T} \int |g|^2 \,dx.
\end{align*}
We intend to plug back this estimate into \eqref{FractionalDerivativeBound}. First, for any $h\in [0,\frac{T}{10}]$ we infer the existence of $t\in [\frac{T}{5},\frac{T}{3}]$ which in addition to the bound
\begin{align*}
\int_t^{t+h} \int |\nabla w|^2 \bigg(\Theta_{T,t}^2+T\frac{|\nabla \Theta_{T,t}^2|^2}{\Theta_{T,t}^2}\bigg) \,dx \,dt
&\leq C\frac{h}{T} 
\int_0^{\frac{T}{3}} \int |\nabla w|^2 \Theta_T^1 \,dx \,dt
\end{align*}
satisfies
\begin{align*}
\int \Theta_{T,t}^2 \big|\Delta_h \nabla w (\cdot,t)\big|^2 \,dx &\leq \frac{C}{T} \int_{\frac{T}{5}}^{\frac{T}{3}} \int \Theta_T^2 \big|\Delta_h \nabla w\big|^2 \,dx \,dt.
\end{align*}
Plugging these three estimates and \eqref{FirstBound} back into \eqref{FractionalDerivativeBound}, we obtain for some $t\in [\frac{T}{5},\frac{T}{3}]$ the improved bound
\begin{align}
\label{Existencetsecond}
\int |w(\cdot,t+h)-w(\cdot,t)|^2 \Theta_{T,t}^2 \,dx \leq C \bigg(\frac{h^{3/2}}{T^{3/2}}+\frac{h^2}{T^2}\bigg) \int |g|^2 \,dx.
\end{align}
By \eqref{DissipationFractionalTime} we obtain for any $h\in [0,\frac{T}{10}]$
\begin{align*}
\int_{\frac{T}{3}}^\infty \int \Theta_{T}^2 \big|\Delta_h \nabla w\big|^2 \,dx \,dt
\leq C \frac{h^{3/2}}{T^{3/2}} \int |g|^2 \,dx.
\end{align*}
In other words, $\nabla w$ belongs to the Nikolskii space on the time interval $[\frac{T}{3},2T]$ with order of differentiability $\frac{3}{4}$, integrability $2$, and values in $L^2_{\Theta_{T,2T}^2}(\mathbb{R}^d)$; furthermore, the Nikolskii seminorm is subject to a bound of the order $\frac{C}{T^{3/2}} \int |g|^2 \,dx$. By the embedding theorem for Nikolskii spaces, we deduce
\begin{align*}
&\sup_{t\in [T/3,2T]} \int \Theta_T^2(\cdot,2T) |\nabla w(\cdot,t)|^2 \,dx
\\&
\leq C\dashint_{T/3}^{2T} \int \Theta_T^2(\cdot,2T) |\nabla w|^2 \,dx \,dt
\\&~~~~
+ C T^{3/2} \sup_{h\in [0,T]} \dashint_{T/3}^{2T-h} \int \Theta_T^2(\cdot,2T) |h^{-3/4} \Delta_h \nabla w|^2 \,dx \,dt
\\&
\stackrel{\eqref{FirstBound}}{\leq} \frac{C}{T} \int |g|^2 \,dx.
\end{align*}
This establishes our lemma.
\end{proof}

\section{Calculus for random variables with stretched exponential moments}

\label{AppendixStretchedExponential}
On the space of random variables $X$ with stretched exponential moments in the sense
\begin{align*}
\mathbb{E}\bigg[\exp\bigg(\frac{|X|^\gamma}{C}\bigg)\bigg]\leq 2
\end{align*}
for some $\gamma>0$ and some $C>0$, it is convenient to work with the norm
\begin{align*}
||X||_{\exp^\gamma} := \sup_{p\geq 1} \frac{1}{p^{1/\gamma}} \mathbb{E}\big[|X|^p\big]^{1/p}.
\end{align*}
For $\gamma\geq 1$, this norm is equivalent to the Luxemburg norm associated with the convex function $\exp(x^\gamma)-1$. However, it has two advantages: First, it simplifies calculus when considering the integrability of products of random variables or the concentration properties of independent random variables. Secondly and more importantly, it is also a well-defined norm for $\gamma\in (0,1)$, a parameter range which we shall employ heavily.

\begin{lemma}
Let $\gamma>0$.
Consider a random variable $X$ on some probability space. Define the quasinorm
\begin{align*}
||X||_{\exp^\gamma,\operatorname{quasi}}:=
\inf\bigg\{s>0:
\mathbb{E}\bigg[\exp\bigg(\frac{|X|^\gamma}{s^\gamma}\bigg)\bigg]\leq 2
\bigg\}.
\end{align*}
Then we have $||X||_{\exp^\gamma,\operatorname{quasi}}<\infty$ if and only if $||X||_{\exp^\gamma}<\infty$ and there exist constants $c(\gamma),C(\gamma)$ such that the estimate
\begin{align*}
c(\gamma) ||X||_{\exp^\gamma}
\leq ||X||_{\exp^\gamma,\operatorname{quasi}}
\leq C(\gamma) ||X||_{\exp^\gamma}
\end{align*}
is satisfied.
\end{lemma}
\begin{proof}
The function
\begin{align*}
f_q(x):=x^q \exp(-x)
\end{align*}
satisfies $f_q'(x)= (q-x) x^{q-1} \exp(-x)$ and attains its maximal value
\begin{align*}
\sup_{x\geq 0} f_q(x)=q^{q} \exp(-q)
\end{align*}
at $x=q$.
Applying the resulting estimate $x^q\leq q^q \exp(x)$ to $x:=|X|^{\gamma}/||X||_{\exp^\gamma,\operatorname{quasi}}^\gamma$ we deduce
\begin{align*}
\mathbb{E}\big[|X|^{\gamma q}\big]
&\leq q^{q} ||X||_{\exp^\gamma,\operatorname{quasi}}^{\gamma q} \mathbb{E}[\exp(|X|^\gamma/||X||_{\exp^\gamma,\operatorname{quasi}}^\gamma)].
\end{align*}
By definition we have $\mathbb{E}[\exp(|X|^\gamma/||X||_{\exp^\gamma,\operatorname{quasi}}^\gamma)]\leq 2$. Setting $p:=\gamma q$ and taking the $p$-th root, we obtain for any $p\geq 1$
\begin{align*}
\mathbb{E}\big[|X|^{p}\big]^{1/p} \leq C(\gamma) p^{1/\gamma} ||X||_{\exp^\gamma,\operatorname{quasi}}.
\end{align*}
This proves $||X||_{\exp^\gamma} \leq C(\gamma) ||X||_{\exp^\gamma,\operatorname{quasi}}$.

To establish the reverse inequality, observe that for $z\in [q,q+1)$ we have $z^q e^{-z/2}\geq q^q e^{-(q+1)/2}$. This entails for all $z\geq 0$
\begin{align*}
\exp(z/2)\leq \sqrt{e}+\sum_{q=1}^\infty \sqrt{e} \frac{z^q}{(q/\sqrt{e})^q}
\end{align*}
and therefore for any $b>0$ (by setting $z=|X|^\gamma/b^\gamma$)
\begin{align*}
\mathbb{E}\big[\exp(|X|^\gamma/2b^\gamma)\big]
\leq \sqrt{e}+\sum_{q=1}^\infty \sqrt{e} \frac{\mathbb{E}[|X|^{\gamma q}]}{b^{\gamma q} (q/\sqrt{e})^q}.
\end{align*}
As a consequence, we obtain
\begin{align*}
\mathbb{E}\big[\exp(|X|^\gamma/2b^\gamma)\big]
\leq \sqrt{e}+\sum_{q=1}^\infty \frac{C(\gamma) (\gamma q)^{q \gamma/\gamma} ||X||_{\exp^\gamma}^{\gamma q}}{b^{\gamma q} (q/\sqrt{e})^q}.
\end{align*}
Setting $b:=a||X||_{\exp^\gamma}$, we get
\begin{align*}
\mathbb{E}\big[\exp(|X|^\gamma/2b^\gamma)\big]
\leq \sqrt{e}+\sum_{q=1}^\infty \frac{C(\gamma) \gamma^q \sqrt{e}^q}{a^{\gamma q} }
\leq \sqrt{e}+\frac{C(\gamma)a^{-\gamma}}{1-\gamma \sqrt{e}/a^\gamma}.
\end{align*}
Choosing $a:=C(\gamma)$ large enough, we deduce
\begin{align*}
\mathbb{E}\bigg[\exp\bigg(\frac{|X|^\gamma}{C(\gamma)||X||_{\exp^\gamma}^\gamma}\bigg)\bigg]\leq 2,
\end{align*}
which entails $||X||_{\exp^\gamma,\operatorname{quasi}}\leq C(\gamma)||X||_{\exp^\gamma}$.
\end{proof}

\begin{lemma}[Calculus for random variables with stretched exponential moments]
\label{CalculusStretchedExponential}
Let $X$, $Y$ be random variables with stretched exponential moments in the sense $||X||_{\exp^\gamma}<\infty$ and $||Y||_{\exp^\beta}<\infty$ for some $\gamma,\beta>0$.
\begin{itemize}
\item[a)] The product $XY$ has stretched exponential moments with exponent $\alpha$ given by $\frac{1}{\alpha}=\frac{1}{\gamma}+\frac{1}{\beta}$ and satisfies the bound
\begin{align*}
||XY||_{\exp^\alpha}\leq C(\beta,\gamma) ||X||_{\exp^\gamma} ||Y||_{\exp^\beta}.
\end{align*}
\item[b)] There exists constants $c=c(\gamma)>0$, $C=C(\gamma)<\infty$, with the following property: For any $K\geq 0$, we have the estimate
\begin{align*}
\mathbb{P}\big[|X|\geq K||X||_{\exp^\gamma}\big] \leq C\exp(-c K^\gamma).
\end{align*}
\end{itemize}
\end{lemma}
\begin{proof}
For the first assertion, we estimate for any $p\geq 1$ by H\"older's inequality
\begin{align*}
\frac{1}{p^{1/\alpha}} \mathbb{E}\big[|XY|^p\big]^{1/p}
&=\frac{1}{p^{1/\gamma+1/\beta}} \mathbb{E}\big[|XY|^p\big]^{1/p}
\leq \frac{1}{p^{1/\gamma+1/\beta}} \mathbb{E}\big[|X|^{2p}\big]^{1/2p}\mathbb{E}\big[|Y|^{2p}\big]^{1/2p}
\\&
\leq 2^{1/\gamma+1/\beta} ||X||_{\exp^\gamma} ||Y||_{\exp^\beta}.
\end{align*}
This establishes the first assertion.

For the second assertion, we estimate for any $p\geq 1$ and any $K> 0$
\begin{align*}
\mathbb{P}\big[|X|\geq K||X||_{\exp^\gamma,\operatorname{quasi}}\big]
\leq \frac{\mathbb{E}\Big[\exp \Big(\Big(\frac{|X|}{||X||_{\exp^\gamma,\operatorname{quasi}}}\Big)^\gamma \Big)\Big]}{\exp(K^\gamma)}
\leq 2\exp(-K^\gamma).
\end{align*}
Using the fact that $||X||_{\exp^\gamma,\operatorname{quasi}}\leq C(\gamma) ||X||_{\exp^\gamma}$, the second assertion follows upon redefining $K$.
\end{proof}

For independent random variables with stretched exponential moments, a standard argument via an inequality by Burkholder \cite{Burkholder} provides a simple concentration estimate.
\begin{lemma}
\label{ConcentrationStretchedExponential}
Let $X_1,\ldots,X_M$ be independent random variables with vanishing expectation and uniformly bounded stretched exponential moments
\begin{align*}
||X_m||_{\exp^{\gamma_0}}\leq b
\end{align*}
for some $\gamma_0>0$ and some $b>0$. Then the sum
\begin{align*}
X:=\sum_{m=1}^M X_m
\end{align*}
has uniformly bounded stretched exponential moments
\begin{align*}
||X||_{\exp^{\tilde \gamma}} \leq C(\gamma_0) \sqrt{M} b
\end{align*}
for $\tilde \gamma := \gamma_0/(\gamma_0+1)$.
\end{lemma}
\begin{proof}
The discrete-time stochastic process
\begin{align*}
m\mapsto \sum_{\tilde m=1}^m X_{\tilde m}
\end{align*}
is a square-integrable martingale. An estimate by Burkholder \cite[Theorem~3.2]{Burkholder} -- applied for ``timestep'' $m:=M$ -- yields for any $k\in \mathbb{N}$
\begin{align*}
\mathbb{E}\big[|X|^{2k}\big]^{1/2k}
\leq C\cdot 2k \mathbb{E}\bigg[\bigg|\sum_{m=1}^M |X_m|^2 \bigg|^k\bigg]^{1/2k}.
\end{align*}
This entails
\begin{align}
\label{ConcentrationBound}
\mathbb{E}\big[|X|^{2k}\big]^{1/2k}
\leq C \cdot 2k \sqrt{M} \, \bigg(\frac{1}{M}\sum_{m=1}^M \mathbb{E}\big[|X_m|^{2k}\big]\bigg)^{1/2k}
\end{align}
and therefore
\begin{align*}
\mathbb{E}\big[|X|^{2k}\big]^{1/2k}
&\leq C\cdot 2k \sqrt{M} \, \bigg(\frac{1}{M}\sum_{m=1}^M (2k)^{2k/\gamma_0} ||X_m||_{\exp^{\gamma_0}}^{2k}\bigg)^{1/2k}
\\&
\leq C \sqrt{M} (2k)^{1+1/\gamma_0} b.
\end{align*}
We infer
\begin{align*}
(2k)^{-1/\gamma_0-1}
\mathbb{E}\big[|X|^{2k}\big]^{1/2k}
\leq C \sqrt{M} b
\end{align*}
for any $k\in \mathbb{N}$ which by H\"older's inequality entails
\begin{align*}
p^{-1/\gamma_0-1}
\mathbb{E}\big[|X|^{p}\big]^{1/p}
\leq C(\gamma_0) \sqrt{M} b
\end{align*}
for any $p\geq 1$.
\end{proof}

\bibliographystyle{abbrv}
\bibliography{stochastic_homogenization}

\end{document}